\documentclass[12pt,a4paper]{amsart}

\usepackage[latin1]{inputenc}

\usepackage{hyperref}

\usepackage{enumerate}

\usepackage{float}
\restylefloat{table}

\usepackage{graphicx}

\usepackage{bm}

\usepackage{amsmath}

\usepackage{amssymb}

\usepackage{amsthm}

\usepackage{tikz-cd}

\usepackage{a4wide}

\usepackage{bbm}

\usepackage{comment}

\newtheoremstyle{plain}
  {6pt}   
  {6pt}   
  {\itshape}  
  {0pt}       
  {\bfseries} 
  {.}         
  {5pt plus 1pt minus 1pt} 
  {}          

\newtheoremstyle{definition}
  {6pt}   
  {6pt}   
  {\normalfont}  
  {0pt}       
  {\bfseries} 
  {.}         
  {5pt plus 1pt minus 1pt} 
  {}          

\theoremstyle{plain}
\newtheorem*{thm*}{Theorem}
\newtheorem{thm}{Theorem}[section]
\newtheorem{prop}[thm]{Proposition}
\newtheorem{cor}[thm]{Corollary}
\newtheorem{lem}[thm]{Lemma}
\newtheorem{theorem}{Theorem}

\theoremstyle{definition}
\newtheorem{defn}[thm]{Definition}

\newtheorem{mex}[thm]{Meta-Example}
\newtheorem{ex}[thm]{Example}
\newtheorem{rmk}[thm]{Remark}
\newtheorem{q}[thm]{Question}
\newtheorem{ass}[thm]{Assumption}

\numberwithin{equation}{thm}

\newcommand{\emphbf}[1]{\emph{\textbf{#1}}}

\DeclareMathAlphabet{\mathpzc}{OT1}{pzc}{m}{it}

\DeclareMathOperator{\Kopf}{top}
\DeclareMathOperator{\soc}{soc}
\DeclareMathOperator{\id}{id}
\DeclareMathOperator{\Hom}{Hom}
\DeclareMathOperator{\Ext}{Ext}

\DeclareMathOperator{\Mod}{Mod}
\DeclareMathOperator{\modu}{mod}
\DeclareMathOperator{\rep}{rep}
\DeclareMathOperator{\repfd}{rep_{fd}}

\DeclareMathOperator{\adj}{adj}
\DeclareMathOperator{\im}{Im}
\DeclareMathOperator{\coker}{Coker}

\DeclareMathOperator{\image}{Im}

\DeclareMathOperator{\res}{res}
\DeclareMathOperator{\Coker}{Coker}

\DeclareMathOperator{\Id}{Id}

\DeclareMathOperator{\proj}{proj}
\DeclareMathOperator{\inj}{inj}
\DeclareMathOperator{\pd}{projdim}
\DeclareMathOperator{\indim}{injdim}

\DeclareMathOperator{\op}{op}

\DeclareMathOperator{\CAT}{CAT}

\DeclareMathOperator{\arrowin}{in}
\DeclareMathOperator{\arrowout}{out}

\DeclareMathOperator{\Lex}{Lex}
\DeclareMathOperator{\Rex}{Rex}
\DeclareMathOperator{\Fun}{Fun}
\DeclareMathOperator{\Ab}{Ab}
\DeclareMathOperator{\supr}{sup}

\DeclareMathOperator{\incl}{incl}
\DeclareMathOperator{\Coh}{Coh}

\DeclareMathOperator{\Mono}{Mono}
\DeclareMathOperator{\Monofd}{Mono_{fd}}
\DeclareMathOperator{\Epi}{Epi}
\DeclareMathOperator{\Epifd}{Epi_{fd}}
\DeclareMathOperator{\fp}{f.p.}

\DeclareMathOperator{\Ambi}{Ambi}

\newcommand{\unit}[2]{\eta^{#1\dashv #2}}
\newcommand{\counit}[2]{\varepsilon^{#1\dashv #2}}
\newcommand{\relGinj}[2]{\mathcal{G}_{f_*}\operatorname{inj}}
\newcommand{\relGproj}[2]{\mathcal{G}_{f_!}\operatorname{proj}}

\begin{document}

\title{A functorial approach to monomorphism categories for species I}

\author{Nan Gao}
\author{Julian K\"ulshammer\and Sondre Kvamme}
\author{Chrysostomos Psaroudakis}
\date{\today}

\thanks{The first author was supported by Natural Science foundation of China (11771272). The third author was supported by a public grant as part of FMJH. The third author wants to thank Xiao Jie for helpful discussions. The authors want to thank Steffen Koenig for suggestions to make the introduction to the paper more accessible. }

\address{Nan Gao\\
Departement of Mathematics, Shanghai University\\
Shanghai 200444, PR China
} \email{nangao@shu.edu.cn}

\address{Julian K\"ulshammer\and Sondre Kvamme\\
Department of Mathematics, 
Uppsala University \\ Box 480 \\ 75106 Uppsala,
Sweden} \email{julian.kuelshammer@math.uu.se\and sondre.kvamme@math.uu.se}

\address{Chrysostomos Psaroudakis\\
Department of Mathematics, 
Aristotle University of Thessaloniki \\ Thessaloniki, 54124, Greece}  \email{chpsaroud@math.auth.gr}

\begin{abstract}
We introduce a very general extension of the monomorphism category as studied by Ringel and Schmidmeier which in particular covers generalised species over locally bounded quivers. We prove that analogues of the kernel and cokernel functor send almost split sequences over the path algebra and the preprojective algebra to split or almost split sequences in the monomorphism category. We derive this from a general result on preservation of almost split morphisms under adjoint functors whose counit is a monomorphism. Despite of its generality, our monomorphism categories still allow for explicit computations as in the case of Ringel and Schmidmeier. 
\end{abstract}

\maketitle

\tableofcontents

\section{Introduction}

Let $A$ be an Artin algebra. In \cite{RS06, RS08}, Ringel and Schmidmeier studied two subcategories of the category $\mathcal{H}(A)$ of homomorphisms of $A$-modules, i.e. the category with objects $(M,N,h)$ where $M$ and $N$ are $A$-modules and $h\colon M\to N$ is a homomorphism (morphisms in $\mathcal{H}(A)$ are given by commutative squares). The two categories are the full subcategories of $\mathcal{H}(A)$ where $h$ is a monomorphism, denoted by $\mathcal{S}(A)$, and where $h$ is an epimorphism, denoted by $\mathcal{F}(A)$. They proved that these subcategories are functorially finite and extension-closed, whence they have almost split sequences. Further, they studied the relationship between  almost split sequences in $\mathcal{S}(A)$ and $\mathcal{F}(A)$ and almost split sequences in the category of all homomorphisms. One of their key results in this situation is that the kernel and the cokernel functor $\mathcal{H}(A)\to \mathcal{S}(A)$ and $\mathcal{H}(A)\to \mathcal{F}(A)$ send almost split sequences in the category of homomorphisms to a direct sum of almost split sequences and split sequences in the category of monomorphisms and epimorphisms, respectively.

The subcategory of monomorphisms is closely related to the category of Gorenstein projective modules over the triangular matrix ring $\left(\begin{smallmatrix}A&A\\0&A\end{smallmatrix}\right)$. Namely, a module over the triangular matrix ring is Gorenstein projective if and only if it is a monomorphism between Gorenstein projective modules whose cokernel is also Gorenstein projective, see \cite[Corollary 3.6]{JK11}. Recently, these results by Ringel and Schmidmeier as well as the connections to Gorenstein projective modules have been further explored, see e.g. \cite{Zha11, XZ12, LZ13, XZZ14, RZ17, Kul17, Kva17, Kva16}. A special case has also been studied in \cite{BBOS20} in the context of topological data analysis. Noting that $\mathcal{H}(A)\cong \modu(A\otimes \Bbbk A_2)\cong \modu \left(\begin{smallmatrix}A&A\\0&A\end{smallmatrix}\right)$, where $A_2$ denotes the Dynkin quiver $A_2$, the generalisations are in two directions: Firstly, they concern monomorphism categories for algebras of the form $A\otimes \Bbbk Q$ where $Q$ is a finite acyclic quiver, or more generally categories of functors $Q\to \modu A$ where $Q$ is a locally bounded quiver, and secondly monomorphism categories in module categories of algebras of the form $\left(\begin{smallmatrix}A&X\\0&B\end{smallmatrix}\right)$, where $A$ and $B$ are Artin algebras and $X$ is a $B$-$A$-bimodule. The purpose of this paper is to introduce a very general extension of the monomorphism category, which in particular covers both of these generalisations, but much more. 

Our original motivation for writing this paper was to study monomorphism categories associated to (generalised) species as defined and studied in \cite{Li12, GLS16, LY15, Geu17, Kul17}. Throughout the introduction, let $Q$ be a finite acyclic quiver. (In the paper we deal with the slightly more general setup of locally bounded quivers for which slight modifications need to be made.) A generalised species $\Lambda$ associated to  $Q$ is an association of a finite dimensional algebra $\Lambda_\mathtt{i}$ to each vertex $\mathtt{i}$ of $Q$ and a $\Lambda_\mathtt{j}$-$\Lambda_\mathtt{i}$-bimodule $\Lambda_\alpha$ to each arrow $\alpha\colon \mathtt{i}\to \mathtt{j}$ such that $\Lambda_\alpha$ is projective as a left $\Lambda_\mathtt{j}$-module as well as a right $\Lambda_\mathtt{i}$-module, and that 
\[\Lambda_{\alpha^*}:=\Hom_{\Lambda_\mathtt{j}}(\Lambda_\alpha,\Lambda_\mathtt{j})\cong \Hom_{\Lambda_\mathtt{i}^{\op}}(\Lambda_\alpha,\Lambda_\mathtt{i})\]
as $\Lambda_\mathtt{i}$-$\Lambda_\mathtt{j}$-bimodules for each $\alpha\colon \mathtt{i}\to \mathtt{j}$ in $Q$. In \cite{Kul17} this is called a dualisable pro-species of algebras. A representation of $\Lambda$ is a tuple $(M_\mathtt{i},M_\alpha)_{\mathtt{i}\in Q_0, \alpha\in Q_1}$ where $M_\mathtt{i}\in \modu \Lambda_\mathtt{i}$ and $M_\alpha\colon \Lambda_\alpha\otimes_{\Lambda_{s(\alpha)}} M_{s(\alpha)}\to M_{t(\alpha)}$ is a $\Lambda_{t(\alpha)}$-linear map. Furthermore, denote by $M_{\alpha^*}$ the adjoint map $M_{s(\alpha)}\to \Lambda_{\alpha^*}\otimes_{\Lambda_{t(\alpha)}} M_{t(\alpha)}$ of the action $M_\alpha$. For each vertex $\mathtt{i}\in Q_0$ define the maps
\[
M_{\mathtt{i},\arrowin}\colon \bigoplus_{\substack{\alpha\in Q_1\\t(\alpha)=\mathtt{i}}} \Lambda_\alpha\otimes_{\Lambda_{s(\alpha)}} M_{s(\alpha)}\xrightarrow{(M_\alpha)_\alpha} M_\mathtt{i}\text{ and }M_{\mathtt{i},\arrowout}\colon M_\mathtt{i}\xrightarrow{(M_{\alpha^*})_\alpha}\bigoplus_{\substack{\alpha\in Q_1\\s(\alpha)=\mathtt{i}}} \Lambda_{\alpha^*}\otimes M_{t(\alpha)}.\]

Then for a dualisable pro-species $\Lambda$ define the monomorphism and the epimorphism category to be the following subcategories of $\rep \Lambda$:
\begin{align*}
\Mono(\Lambda)&=\{M\in \rep\Lambda\,|\,M_{\mathtt{i},\arrowin} \text{ is a monomorphism }\}\\
\Epi(\Lambda)&=\{M\in \rep\Lambda\,|\,M_{\mathtt{i},\arrowout} \text{ is an epimorphism }\}
\end{align*}


Note that we follow the terminology of \cite{LZ13} in calling this category the monomorphism category while (in the setup of quivers with relations) \cite{XZ17, XZ19} call this category the separated monomorphism category. The main reason for doing so is that for us this is a special case of our more general setup (see below) in which there is only one morphism which is demanded to be a monomorphism (or an epimorphism). In this more general setup, speaking of separation doesn't make any sense.

We define functors $\nu\colon \rep(\Lambda)\to \Epi(\Lambda)$ and $\nu^-\colon \rep(\Lambda)\to \Mono(\Lambda)$ analogous to the cokernel and kernel functors of Ringel and Schmidmeier which behave like a `relative' (inverse) Nakayama functor. As for the cokernel and kernel functor, they are defined very explicitly, which we illustrate by a few running examples. The functor $\nu^ -$ is defined on objects via the following short exact sequence:

\begin{equation}
\label{def:taunuspecies}
0\to (\nu^-(M_\mathtt{j},M_\alpha))_\mathtt{i}\hookrightarrow \bigoplus_{\substack{q\in Q_{\geq 0}\\t(q)=\mathtt{i}}} \Lambda_q\otimes_{\Lambda_{s(q)}} M_{s(q)}\xrightarrow{\xi} \bigoplus_{\substack{p\in Q_{\geq 0}\\t(p)=\mathtt{i}}}\bigoplus_{\substack{\alpha\in Q_1\\s(\alpha)=s(p)}} \Lambda_p\otimes_{\Lambda_{s(\alpha)}}\Lambda_{\alpha^*}\otimes_{\Lambda_{t(\alpha)}} M_{t(\alpha)}
\end{equation}
where $Q_{\geq 0}$ denotes the set of paths in $Q$, $\Lambda_q=\Lambda_{\alpha_{t}}\otimes_{\Lambda_{s(\alpha_t)}}\dots \otimes_{\Lambda_{s(\alpha_2)}} \Lambda_{\alpha_1}$ for $q=\alpha_t\dots\alpha_1$ a path written as a concatenation of arrows $\alpha_i$, and $\xi$ is the difference between the coevaluation $\Lambda_q\otimes_{ \Lambda_{s(q)}} M_{s(q)}\to \Lambda_{q}\otimes_{\Lambda_{t(\alpha)}} \Lambda_\alpha\otimes_{\Lambda_{s(\alpha)}} \Lambda_{\alpha^*}\otimes_{\Lambda_{t(\alpha)}} M_{t(\alpha)}$ and $M_{\alpha^*}$.
The functor $\nu$ is defined as the left adjoint of $\nu^-$. 

We prove the following generalisation of a theorem of Ringel and Schmidmeier:

\begin{theorem}[Theorem \ref{thmCmaintext}, Corollary \ref{imagesofseveralalmostsplitsequences} and Corollary \ref{Existence of almost split sequences}]\label{thmC}
Let $\Lambda$ be a dualisable pro-species on $Q$. 
\begin{enumerate}[(i)]
\item The categories $\Mono(\Lambda)$ and $\Epi(\Lambda)$ have almost split sequences. 
\item Let $g$ be a right almost split morphism in $\rep(\Lambda)$. Then $\nu^-(g)$ is a split epimorphism or a right almost split morphism in $\Mono(\Lambda)$. A dual statement holds for the epimorphism category. 
\item Let $0\to A''\xrightarrow{g'} A\xrightarrow{g} A'\to 0$ be an almost split sequence in $\rep(\Lambda)$ with $A'\in \Epi(\Lambda)$ such that $A'$ is not projective in $\Epi(\Lambda)$. Then, 
\[0\to \nu^-(A'')\xrightarrow{\nu^-(g')} \nu^-(A)\xrightarrow{\nu^-(g)} \nu^-(A')\to 0\]
is a direct sum of an almost split sequence and a split exact sequence in $\Mono(\Lambda)$. A dual statement holds for the epimorphism category 
\end{enumerate}
\end{theorem}

Furthermore, given a dualisable pro-species one can associate to it its preprojective algebra $\Pi(\Lambda)$ as in \cite{Kul17}. Denote by $g^*$ the restriction functor $\modu \Pi(\Lambda)\to \rep(\Lambda)$.  Furthermore, define 
 $\tau^-$, the `relative' inverse Auslander--Reiten translation, via the cokernel of the map $\xi$ in the sequence \eqref{def:taunuspecies} and $M\in \rep(\Lambda)$. 
One then obtains the following theorem, the first part of which generalises a theorem of Ringel \cite[Theorem A]{Ri98} while the second part generalises unpublished work of the first and last named author in the case of the preprojective algebra associated to Ringel--Schmidmeier's situation, which is isomorphic to the Morita ring, see \cite{Kul17}. 

\begin{theorem}[Theorem \ref{ringelsdescriptionofpreprojective} and Theorem \ref{thmDmaintext}]\label{thmD}
Let $\Lambda$ be a dualisable pro-species of algebras. 
\begin{enumerate}[(i)]
\item The category $\modu \Pi(\Lambda)$ is equivalent to the category of morphisms $\tau^-(M)\to M$.
\item Let $0\to L\to M\to N\to 0$ be an almost split sequence in $\modu \Pi(\Lambda)$. Assume that $\nu^-g^*(N)$ is not projective. Then the sequence $0\to (\nu^-g^*)(L)\to (\nu^-g^*)(M)\to (\nu^-g^*)(N)\to 0$ is exact. Furthermore, it is either split exact or a direct sum of a split exact and an almost split sequence in $\Mono(\Lambda)$. 
\end{enumerate}
\end{theorem}

In fact, Theorems \ref{thmC} and \ref{thmD} (and their dual statements) are consequences of the theory of adjunctions with Nakayama functors as developed by the third named author in \cite{Kva16} as follows: We consider the abelian category $\mathcal{C}=\prod_{\mathtt{i}\in Q_0} \modu \Lambda_\mathtt{i}$ and on it the `push-along' functor 
\[X\colon \mathcal{C}\to \mathcal{C}, (M_\mathtt{i})_{\mathtt{i}\in Q_0}\mapsto \left(\bigoplus_{\substack{\alpha\in Q_1\\t(\alpha)=\mathtt{i}}} \Lambda_\alpha\otimes M_{s(\alpha)}\right)_{\mathtt{i}\in Q_0}.\] 
Then $\rep \Lambda\cong \mathcal{C}^{T(X)}$, the Eilenberg--Moore category of the free monad on the endofunctor $X$. 
 
The functor $X$ also satisfies three technical conditions, from which we in fact derive even more general analogues of Theorem \ref{thmC} and \ref{thmD} in this abstract setup. Namely,
\begin{itemize}
\item The functor $X$ is a Frobenius functor, i.e. its left adjoint $Y$ is also a right adjoint;
\item the coproduct of the powers of $X$ coincides with the product of the powers of $X$ (and the same condition for $Y$), see Assumption \ref{standard1};
\item there are no epimorphisms $X(M)\to M$ and no monomorphisms $M\to Y(M)$ for every non-zero $M\in \mathcal{C}$. This is an analogue of Nakayama's lemma.
\end{itemize}

Under these assumptions, the monomorphism category coincides with the category of (relative) Gorenstein projective objects with respect to the Eilenberg--Moore adjunction $\begin{tikzcd}\mathcal{C}^{T(X)}\arrow[yshift=-0.5ex]{r} &\mathcal{C}\arrow[yshift=0.5ex]{l}\end{tikzcd}$ which admits a (relative) Nakayama functor $\nu$ in the sense of \cite{Kva16}. We furthermore prove that if $\mathcal{C}$ has enough projectives and injectives then the third condition is redundant. 

Our reasoning for considering this more abstract setup (beside covering some obvious generalisations of species where the underlying abelian category doesn't need to have enough projectives) is two-fold. Firstly, the abstract setup removes some of the technical difficulties in defining the monomorphism category for a generalised species and makes the categorical reasons, why certain notions are defined in a specific way, more transparent. Secondly, there is a striking similarity of our abstract  conditions with conditions arising in the literature in connection with the theory of categorification, see e.g. \cite{Lau06} for the first condition and \cite{EH17} for an analogue of the second condition. 

In a second paper in preparation \cite{GKKP19b} we will provide an analogue of the explicit description of the Auslander--Reiten translation obtained in \cite{RS08} for the monomorphism category of a dualisable pro-species. 

The above theorems are in fact consequences of the following master theorem on preservation of almost split morphisms under adjoint functors:

\begin{theorem}[Theorem \ref{theoremAmaintext}]\label{thmA}
Let $\mathcal{A}$ and $\tilde{\mathcal{A}}$ be abelian categories. Let $(L,R)$ be an adjoint pair with $L\colon \tilde{\mathcal{A}}\to \mathcal{A}$. Assume that the counit of the adjunction $\counit{L}{R}\colon LR\to \Id_\mathcal{A}$ is a monomorphism. If $g\colon A\to A'$ is a right almost split morphism in $\tilde{\mathcal{A}}$, then $R(g)\colon R(A)\to R(A')$ is either split or right almost split in $\image(R)$. 
\end{theorem}

In the context of Ringel and Schmidmeier, the adjoint pair $(L,R)$ is in fact given by the cokernel-kernel adjunction. The corresponding images are the epimorphism and the monomorphism category, respectively. Another application of Theorem \ref{thmA} arises by interpreting the cokernel-kernel adjunction as an adjunction between  $\Fun(\proj \Bbbk A_2, \modu \Lambda)$ and $\Fun(\inj \Bbbk A_2, \modu \Lambda)$, where $A_2$ denotes the Dynkin quiver $A_2$. Let $\mathcal{A}$ and $\mathcal{B}$ be abelian categories, and assume $\mathcal{B}$ is essentially small and has enough projectives and injectives. Let $\mathsf{L}\colon \Fun(\proj\mathcal{B},\mathcal{A})\to \Fun(\inj\mathcal{B},\mathcal{A})$ be the functor given by the composite
\[
\mathsf{L}\colon \Fun(\proj\mathcal{B},\mathcal{A}) \xrightarrow{\widehat{\bullet}} \Fun(\mathcal{B},\mathcal{A})\xrightarrow{\res} \Fun(\inj\mathcal{B},\mathcal{A})
\]
where $\widehat{\bullet}$ extends a functor $F\colon \proj\mathcal{B}\to  \mathcal{A}$ to a right exact functor $\widehat{F}\colon \mathcal{B}\to \mathcal{A}$ as in diagram (3.2.1), and where $\res$ is the restriction functor. Dually, let $\mathsf{R}\colon \Fun(\inj\mathcal{B},\mathcal{A})\to \Fun(\proj\mathcal{B},\mathcal{A})$ be the functor given by the composite  
\[
\mathsf{R}\colon \Fun(\inj\mathcal{B},\mathcal{A}) \xrightarrow{\widetilde{\bullet}} \Fun(\mathcal{B},\mathcal{A})\xrightarrow{\res} \Fun(\proj\mathcal{B},\mathcal{A})
\]  
where $\widetilde{\bullet}$ extends a functor $F\colon \inj\mathcal{B}\to  \mathcal{A}$ to a left exact functor $\widetilde{F}\colon \mathcal{B}\to \mathcal{A}$, and where $\res$ is the restriction functor.

\begin{theorem}[Theorem \ref{thmBmaintext}]\label{thmB}
The following holds:
\begin{enumerate}[(i)]
\item The functor $\mathsf{L}$ is left adjoint to $\mathsf{R}$;
\item If $\mathcal{B}$ is $1$-Gorenstein, then the counit of the adjunction $\mathsf{L}\dashv \mathsf{R}$ is a monomorphism. In particular, if $g$ is a right almost split morphism in $\Fun(\proj\mathcal{B},\mathcal{A})$, then $\mathsf{L}(g)$ is a split or right almost split morphism in the full subcategory of $\Fun(\inj\mathcal{B},\mathcal{A})$ whose objects are restrictions of left exact functors in $\Fun(\mathcal{B},\mathcal{A})$.  
\end{enumerate}
\end{theorem}

This approach generalises the theory of monomorphism categories to functor categories, similarly to \cite{Kva17} but without the requirement that the projectives and the injectives of $\mathcal{B}$ coincide. Theorem \ref{thmA} also applies to the context of trivial extensions of categories as studied in \cite{FGR75}, see Corollary \ref{trivialextensions}.

The article is structured as follows. In Section \ref{sec:notation} we set up our notation for this paper. In Section \ref{sec:extensionoffunctors}, we give a proof of Theorems \ref{thmA} and \ref{thmB}. In Section \ref{sec:phyla} we give a slight generalisation of pro-species of algebras called phyla and its category of modules. This will be the main source of examples for our set-up. Section \ref{sec:monads} recalls the notion of a monad with particular emphasis on the free monad. Furthermore, analogues of the top and socle functor are introduced in the abstract setting. Section \ref{sec:nakayamafunctors} is devoted to the proof of Theorem \ref{thmC} using the theory of adjunctions with Nakayama functors introduced in \cite{Kva17}. Furthermore, we show that in our setup the category of (relative) Gorenstein projective objects can be described as a generalisation of the monomorphism category. In the final section, Section \ref{sec:preprojective}, we introduce an analogue for the preprojective algebra in our setup and give a proof of Theorem \ref{thmD}.

\section{Notation} \label{sec:notation}

Throughout, for a category $\mathcal{C}$ we write $\mathcal{C}(M,N)$ for the set of morphisms from $M$ to $N$. In the case of the category $\modu A$ of finite dimensional left modules over a finite dimensional algebra $A$, as is standard, we use $\Hom_A(M,N)$ instead. Denote the category of abelian groups by $\Ab$. 

Let all functors between additive categories be additive. For abelian categories $\mathcal{A}$ and $\mathcal{B}$ we denote by $\Fun(\mathcal{B},\mathcal{A})$ the category of all (covariant) additive functors $\mathcal{B}\to \mathcal{A}$ and by $\proj \mathcal{B}$, respectively $\inj \mathcal{B}$ the full subcategory of projective, respectively injective objects in $\mathcal{B}$.  

For an endofunctor $X\colon \mathcal{C}\to \mathcal{C}$ we define $X^0=\Id_\mathcal{C}$ and  $X^i:=\underbrace{X\circ X\circ \dots \circ X}_{i \text{ times }}$ for $i>0$. 
For a category $\mathcal{C}$, we use $\coprod_{i} M_i$ for the coproduct of the objects $M_i$ and $\prod_{i} M_i$ for their product. In case the two are isomorphic via the natural map $\coprod_i M_i\to \prod_i M_i$ we refer to it as the biproduct and use the notation $\bigoplus_{i} M_i$. 

For future reference we fix the following setup:

\begin{ass}\label{standard1}
Let $\mathcal{C}$ be an abelian category. Let $X\colon \mathcal{C}\to \mathcal{C}$ be an endofunctor of $\mathcal{C}$. Assume that the left adjoint $Y$ of $X$ coincides with the right adjoint of $X$. Assume furthermore that the biproducts $\bigoplus_{i\geq 0} X^i(M)$ and $\bigoplus_{i\geq 0} Y^i(M)$ exist for all $M\in \mathcal{C}$. 
\end{ass}

Furthermore, for functors $L\colon \mathcal{A}\to \tilde{\mathcal{A}}$ and $R\colon \tilde{\mathcal{A}}\to \mathcal{A}$ we write $L\dashv R$ if $L$ is left adjoint to $R$ and denote the unit and counit of the adjunction by $\unit{L}{R}$ and $\counit{L}{R}$, respectively. The adjunction isomorphism will be denoted by $\adj^{L\dashv R}_{M,N}\colon \tilde{\mathcal{A}}(L(M),N)\to \mathcal{A}(M,R(N))$ where we omit the subscript when $M$ and $N$ are clear from the context. 

A \emphbf{quiver} $Q=(Q_0,Q_1,s,t)$ is a directed graph with vertex set $Q_0$, arrow set $Q_1$ and functions $s,t\colon Q_1\to Q_0$ specifying the source and the target of an arrow, respectively. We will denote vertices of quivers in typewriter font $\mathtt{i},\mathtt{j},\mathtt{k}$.  
A quiver is called \emphbf{locally bounded} if for every vertex $x$ the set of paths starting in $x$ as well as the set of paths ending in $x$ are finite.

\section{Extension of functors}\label{sec:extensionoffunctors}

This section provides the proof of Theorem \ref{thmA} right at the start. It is a generalisation of the result of Ringel and Schmidmeier concerning the kernel functor and the monomorphism category, see \cite[Lemma 3.1]{RS08}. Recall that a  morphism $g\colon A\to A'$ in a category $\mathcal{A}$ is called \emphbf{right almost split} if $g$ is not a split epimorphism and for every morphism $h\colon B\to A'$ which is not a split epimorphism there exists a morphism $h'\colon B\to A$ such that $gh'=h$. The notion of a \emphbf{left almost split} morphism is defined dually. Furthermore for a functor $R\colon \mathcal{A}\to \tilde{\mathcal{A}}$ we denote by $\image R$ 
the \emphbf{essential image} of $R$, considered as a full subcategory of $\tilde{\mathcal{A}}$. 

\begin{thm}\label{theoremAmaintext}
Let $\mathcal{A}$ and $\tilde{\mathcal{A}}$ be abelian categories. Let $L\colon \tilde{\mathcal{A}}\to \mathcal{A}$ be left adjoint to $R\colon \mathcal{A}\to \tilde{\mathcal{A}}$.
Assume that the counit of the adjunction $\counit{L}{R}\colon LR\to \Id_{\mathcal{A}}$ is a pointwise monomorphism. If $g\colon A\to A'$ is a right almost split morphism in $\mathcal{A}$ then $R(g)\colon R(A)\to R(A')$ is either split or right almost split in ${\image}(R)$.
A dual statement regarding left almost split morphisms being preserved holds for $L$. 
\end{thm}

\begin{proof}
Let $h\colon R(B)\to R(A')$ be a morphism which is not a split epimorphism. We show that it factors through $R(g)$. Applying $L$ and composing with the counit of the adjunction we obtain a map $\counit{L}{R}_{A'}\circ L(h)\colon (LR)(B)\to A'$. Assume that this was a split epimorphism. Since $\counit{L}{R}$ is a pointwise monomorphism by assumption, this would imply that $\counit{L}{R}_{A'}$ is an isomorphism and hence that $L(h)$ is a split epimorphism. Thus, $(RL)(h)\colon (RLR)(B)\to (RLR)(A')$ is a split epimorphism. However,  the composition of the unit and the counit $R\to RLR\to R$ is the identity natural transformation on $R$. In particular, the map $R\counit{L}{R}\colon RLR\to R$ is a pointwise epimorphism. But it is also a pointwise monomorphism as $\counit{L}{R}$ is a pointwise monomorphism and $R$ is left exact. Thus, the functors $RLR$ and $R$ are naturally isomorphic via $R\counit{L}{R}$ with inverse $\unit{L}{R}_R$. In particular, the commutative diagram
\[
\begin{tikzcd}
(RLR)(B)\arrow{r}{(RL)(h)} &(RLR)(A')\\
R(B)\arrow{u}{\unit{L}{R}_{R(B)}}\arrow{r}{h}& R(A')\arrow{u}{\unit{L}{R}_{R(A')}}
\end{tikzcd}
\]
implies that the map $h$ is a split epimorphism, which contradicts our assumption.
Thus, $\counit{L}{R}_{A'}\circ L(h)\colon (LR)(B)\to A'$ is not a split epimorphism, hence it factors through $g$, i.e. there exists $h'\colon (LR)(B)\to A$ such that $gh'=\counit{L}{R}_{A'}\circ L(h)$. Applying the adjunction yields a factorisation $h=R(g)R(h')\unit{L}{R}_B$ of $h$ through $R(g)$ by naturality of the adjunction.
\end{proof}

The condition of the counit being a pointwise monomorphism seems strange at first sight as being a pointwise epimorphism is much more common. In fact the counit being a pointwise epimorphism is equivalent to the right adjoint being faithful. But there are other instances where the condition of the counit being a pointwise monomorphism has been considered, see e.g. \cite[Proposition 2.4]{GT92}, \cite{CW11}.

As a first application we consider the notion of a trivial extension in the sense of Fossum, Griffith, and Reiten \cite{FGR75}. Let $\mathcal{A}$ be an abelian category and let $F\colon \mathcal{A}\to \mathcal{A}$ be a functor. Define the \emphbf{right trivial extension} $\mathcal{A}\ltimes F$ of $\mathcal{A}$ by $F$ as the category with objects being morphisms $h\colon F(M)\to M$ for $M\in \mathcal{A}$ such that $h\circ F(h)=0$ and morphisms $(M,h)\to (N,g)$ being given by $\varphi\colon M\to N$ such that $\varphi\circ h=g\circ F(\varphi)$. Dually, one can define the \emphbf{left trivial extension} $F\rtimes \mathcal{A}$ as the category with objects $h\colon M\to F(M)$ such that $F(h)\circ h=0$. According to \cite[Proposition 1.1]{FGR75}, right exactness of $F$ implies that $\mathcal{A}\ltimes F$ is abelian while left exactness of $F$ implies that $F\rtimes \mathcal{A}$ is abelian.  Furthermore, there is an adjunction $C\dashv Z$, where $C\colon \mathcal{A}\ltimes F\to \mathcal{A}$ is the cokernel functor $(M,h)$ to $\coker h$ and $Z\colon \mathcal{A}\to \mathcal{A}\ltimes F$ is the functor sending $M$ to $(M,0)$. Dually, the functor $Z\colon \mathcal{A}\to F\rtimes \mathcal{A}$, $M\mapsto (M,0)$,  is left adjoint to the kernel functor $K$. By \cite[Proposition 1.5]{FGR75}, the unit of the former adjunction is an epimorphism while the counit of the latter adjunction is a monomorphism. Theorem \ref{theoremAmaintext} then has the following corollary:

\begin{cor}\label{trivialextensions}
If $g$ is a left almost split morphism in $\mathcal{A}\ltimes F$, then $C(g)$ is either split or left almost split in $\mathcal{A}$. Dually, if $g$ is a right almost split morphism in $F\rtimes \mathcal{A}$, then $K(g)$ is either split or right almost split in $\mathcal{A}$. 
\end{cor}

For our needs in the sequel, we recall the notion of finitely presented functors studied by Auslander in \cite{Aus66} under the name coherent functors. Let $\mathcal{C}$ be an additive category with weak kernels, that is, for each morphism $f\colon M\to N$ in $\mathcal{C}$ there exists a morphism
$g\colon L\to M$ in $\mathcal{C}$  such that the sequence 
$\Hom_{\mathcal{C}}(-,L) \to \Hom_{\mathcal{C}}(-,M) \to \Hom_{\mathcal{C}}(-,N)$
is exact. A functor $F\colon \mathcal{C}^{\op}\to \Ab$ is called \emphbf{finitely presented} if there exists an exact sequence of functors
\[
\Hom_{\mathcal{C}}(-,M)\to \Hom_{\mathcal{C}}(-,N)\to F\to 0
\]
with $M$ and $N$ in $\mathcal{C}$. The category of finitely presented functors over $\mathcal{C}$ is denoted by $\fp(\mathcal{C}^{\op}, \Ab)$. Let $\mathbb{Y}\colon \mathcal{C}\to \fp(\mathcal{C}^{\op},\Ab)$, $C\mapsto \Hom_{\mathcal{C}}(-,C)$ 
denote the Yoneda embedding. Since $\mathcal{C}$ has weak kernels, $\fp(\mathcal{C}^{\op}, \Ab)$ is an abelian category with enough projective objects. We refer to \cite{Aus66, Kra98} for more details on finitely presented functors.

Let $\tilde{\mathcal{C}}$ be an abelian category with enough projectives. It is then known that the category $\fp((\proj \tilde{\mathcal{C}})^{\op},\Ab)$ is equivalent to $\tilde{\mathcal{C}}$, see \cite[Proposition 2.3]{Kra98}. Furthermore, it follows from \cite[Universal Property 2.1]{Kra98} that for every abelian category $\mathcal{B}'$ every additive functor $F\colon \proj\tilde{\mathcal{C}}\to \mathcal{B}'$ can be extended uniquely, up to a natural isomorphism, to a right exact functor $\widehat{F}\colon \fp((\proj\tilde{\mathcal{C}})^{\op},\Ab)\to \mathcal{B}'$ such that $\widehat{F}\circ \mathbb{Y}=F$, i.e. the following diagram commutes$\colon$
\begin{equation}
\label{universalproperty1}
\begin{tikzcd}
\fp((\proj\tilde{\mathcal{C}})^{\op},\Ab) \arrow{r}{\widehat{F}} & \mathcal{B}' \\
\proj{\tilde{\mathcal{C}}} \arrow{u}{\mathbb{Y}} \arrow{ur}[swap]{F}  &
\end{tikzcd}
\end{equation}
More precisely, given a finitely presented functor $\Phi$ in $\fp((\proj\tilde{\mathcal{C}})^{\op},\Ab)$ and a choice of projective presentation 
\[
\begin{tikzcd}
\proj\tilde{\mathcal{C}}(-,P_1) \arrow{r}{(-,f)} & \proj\tilde{\mathcal{C}}(-,P_0) \arrow{r} & \Phi \arrow{r} & 0
\end{tikzcd}
\]
then $\widehat{F}(\Phi):=\Coker{F(f)}$. Using the equivalence $\tilde{\mathcal{C}}\cong \fp((\proj\tilde{\mathcal{C}})^{\op},\Ab)$ this corresponds to choosing a projective presentation $P_1\xrightarrow{f} P_0\to A$ in $\tilde{\mathcal{C}}$ for each object $\tilde{C}\in \tilde{\mathcal{C}}$ (choosing the trivial one for each projective object), and defining $\widehat{F}(\tilde{C}):=\Coker F(f)$. It can be checked that this defines a functor 
\[\widehat{\bullet}\colon \Fun(\proj\tilde{\mathcal{C}},\mathcal{B}')\to \Fun(\tilde{\mathcal{C}},\mathcal{B}'), \quad F\mapsto \widehat{F}.\] 
It is defined on morphisms as follows: Let $\alpha\colon F\to F'$ be a natural transformation where $F$ and $F'$ lie in $\Fun(\proj \tilde{\mathcal{C}},\mathcal{B}')$. Then there is an  exact commutative diagram$\colon$
\begin{equation}
\label{exactcomdiagramff}
\begin{tikzcd}
F(P_1) \arrow{d}{\alpha_{P_1}} \arrow{r}{F(f)} & F(P_0) \arrow{d}{\alpha_{P_0}} \arrow{r} & \widehat{F}(\Phi) \arrow{r} \arrow{d}{\widehat{\alpha}_{\Phi}} & 0 \\
F'(P_1) \arrow{r}{F'(f)} & F'(P_0) \arrow{r} & \widehat{F'}(\Phi) \arrow{r} & 0
\end{tikzcd}
\end{equation}

Dually, if $\tilde{\mathcal{C}}$ is an abelian category with enough injective objects, then the category of finitely presented  functors $\fp((\inj{\tilde{\mathcal{C}}})^{\op},\Ab)$ is equivalent to $\tilde{\mathcal{C}}^{\op}$. 
This equivalence implies that any additive functor $\inj \tilde{\mathcal{C}}\to \mathcal{B}'$ can be extended uniquely up to natural isomorphism to a left exact functor $\tilde{\mathcal{C}}\to \mathcal{B}$, i.e. we can define a functor  
\[\widetilde{\bullet}\colon \Fun(\inj\tilde{\mathcal{C}},\mathcal{B}')\to \Fun(\tilde{\mathcal{C}},\mathcal{B}'), \quad F\mapsto \widetilde{F}\]
by choosing an injective copresentation $0\to \tilde{C}\to I_0\xrightarrow{f} I_1$ for each $\tilde{C}\in \tilde{\mathcal{C}}$ (choosing the trivial one for each injective object) and defining $\widetilde{F}(\tilde{C})=\ker F(f)$.
We now recall the following result due to Auslander.

\begin{prop}[{\cite[Proposition 2.1]{Aus66}}]
\label{adjoints}
Let $\tilde{\mathcal{C}}$ and $\mathcal{B}'$ be abelian categories.
\begin{enumerate}[(i)]
\item\label{adjoint:i} Assume that $\tilde{\mathcal{C}}$ has enough projectives.  Then, the functor $\widehat{\bullet}$ is left adjoint to the restriction functor $\Fun(\tilde{\mathcal{C}}, \mathcal{B}')\to \Fun(\proj{\tilde{\mathcal{C}}}, \mathcal{B}')$.
\item\label{adjoint:ii} Assume that $\tilde{\mathcal{C}}$ has enough injectives. Then, the functor $\widetilde{\bullet}$ is right adjoint to the restriction functor $\Fun(\tilde{\mathcal{C}}, \mathcal{B}')\to \Fun(\inj{\tilde{\mathcal{C}}}, \mathcal{B}')$.
\end{enumerate}
\end{prop}

Denote by $\Lex(\tilde{\mathcal{C}},\mathcal{B}')$ the full subcategory of $\Fun(\tilde{\mathcal{C}},\mathcal{B}')$ of all left exact functors and by $\Rex(\tilde{\mathcal{C}},\mathcal{B}')$ the full subcategory of all right exact functors. 

\begin{lem}
Let $\tilde{\mathcal{C}}$ and $\mathcal{B}'$ be abelian categories.
\begin{enumerate}[(i)]
\item\label{extension:ii} Assume that $\tilde{\mathcal{C}}$ has enough projectives. Then there is an equivalence of categories
\[
\widehat{\bullet}\colon \Fun(\proj \tilde{\mathcal{C}},\mathcal{B}')\cong \Rex(\tilde{\mathcal{C}},\mathcal{B}').
\]
\item\label{extension:i}
Assume that $\tilde{\mathcal{C}}$ has enough injectives. Then there is an equivalence of categories
\[
\widetilde{\bullet}\colon \Fun(\inj\tilde{\mathcal{C}},\mathcal{B}')\cong \Lex(\tilde{\mathcal{C}},\mathcal{B}').
\]

\end{enumerate}
\end{lem}
\begin{proof}
This is well-known. For the reader's convenience we sketch the proof of \eqref{extension:ii}, the proof of \eqref{extension:i} is dual. It suffices to show that the assignment $F\mapsto \widehat{F}$ induces an equivalence between $\Fun((\proj \tilde{\mathcal{C}},\mathcal{B}')$ and $\Rex(\fp((\proj\tilde{\mathcal{C}})^{\op},\Ab),\mathcal{B}')$.

We first show that $\widehat{\bullet}$ is essentially surjective. Let $H\colon \fp((\proj\tilde{\mathcal{C}})^{\op}, \Ab)\to \mathcal{B}'$ be a right exact functor. Set $K=H\circ \mathbb{Y}\colon \proj{\tilde{\mathcal{C}}}\to \mathcal{B}'$. By the universal property (\ref{universalproperty1})
\[\widehat{K}\circ \mathbb{Y}=K=H\circ \mathbb{Y}.\]
This shows that the functors $\widehat{K}$ and $H$ coincide on projective objects. Since they are both right exact, it follows that  $\widehat{K}$ is naturally isomorphic to $H$.

We now show that the functor $\widehat{\bullet}$ is full. Suppose that there is a morphism $h\colon \widehat{F}\to \widehat{F'}$. Commutativity of the diagram (\ref{universalproperty1}) ipmlies that $F(P)=\widehat{F}(P)$ and $F'(P)=\widehat{F'}(P)$ for every $P$ in $\proj{\tilde{\mathcal{C}}}$. Denote by $h'\colon F\to F'$ the restriction of $h$ to $\proj \tilde{C}$. This implies that there is an exact commutative diagram similar to $(\ref{exactcomdiagramff})$, where the vertical maps are now $h_{P_1}'$ and $h_{P_0}'$. The latter maps induce a morphism $\widehat{h'}_{\Phi}\colon \widehat{F}(\Phi)\to \widehat{F'}(\Phi)$. By the universal property of cokernels we obtain that the natural transformations $h$ and $\widehat{h'}$ are equal.

Finally, we show that the functor $\widehat{\bullet}$ is faithful. Suppose that there is a natural transformation $\alpha\colon F\to F'$ such that $\widehat{\alpha}=0$. Composing with the Yoneda functor $\mathbb{Y}$ yields that $0=\widehat{\alpha}_\mathbb{Y}\colon \widehat{F}\circ \mathbb{Y}\to \widehat{F'}\circ \mathbb{Y}$. Since $F=\widehat{F}\circ \mathbb{Y}$ and $F'=\widehat{F'}\circ \mathbb{Y}$ we infer that $\alpha=0$. This completes the proof that the functor $\widehat{\bullet}$ is an equivalence of categories.
\end{proof}

Let $\mathcal{A}$ and $\mathcal{B}$ be abelian categories. Assume that $\mathcal{B}$ has enough projectives and enough injectives. Let $\mathsf{R}$ be the functor $\Fun(\inj \mathcal{B},\mathcal{A})\to \Fun(\proj \mathcal{B},\mathcal{A})$ given by composing $\widetilde{\bullet}$ with the restriction to $\proj \mathcal{B}$. Dually, let $\mathsf{L}\colon \Fun(\proj \mathcal{B},\mathcal{A})\to \Fun(\inj \mathcal{B},\mathcal{A})$  be the functor given by composing $\widehat{\bullet}$ with restriction to $\inj\mathcal{B}$. This gives the following diagram:
\[
\begin{tikzcd}
\Lex(\mathcal{B},\mathcal{A}) \arrow{d} & \Fun(\inj\mathcal{B},\mathcal{A}) \arrow[l,swap,"\widetilde{\bullet}"] \arrow[xshift=-0.5ex]{d}[swap]{\mathsf{R}} & \Fun(\mathcal{B},\mathcal{A}) \arrow[l,swap,"\mathsf{res}"] \\
\Fun(\mathcal{B}, \mathcal{A}) \arrow["\mathsf{res}"]{r} & \Fun(\proj{\mathcal{B}},\mathcal{A}) \arrow{r}{\widehat{\bullet}} \arrow[xshift=0.5ex]{u}[swap]{\mathsf{L}} &  \Rex(\mathcal{B},\mathcal{A}) \arrow{u}
\end{tikzcd}
\]
It follows from Proposition \ref{adjoints} that $\mathsf{L}$ is left adjoint to $\mathsf{R}$.

We now recall from \cite[Chapter VII]{BR07} the notion of Gorenstein abelian categories. Let $\mathcal{A}$ be an abelian category with enough projective and injective objects. Associated to $\mathcal{A}$ we consider the following homological invariants$\colon$
\[
\mathsf{spli}{\mathcal{A}} = \supr\{\pd_{\mathcal{A}}I \ | \  I\in \inj{\mathcal{A}} \}
\ \ \text{and} \ \
\mathsf{silp}{\mathcal{A}} = \supr\{\indim_{\mathcal{A}}P \ | \  P\in \proj{\mathcal{A}} \}
\]
The category $\mathcal{A}$ is called \emphbf{Gorenstein} if
$\mathsf{spli}{\mathcal{A}}<\infty$ and $\mathsf{silp}{\mathcal{A}}<\infty$. Moreover, $\mathcal{A}$ is called \emphbf{n-Gorenstein} if the maximum of $\mathsf{spli}{\mathcal{A}}$ and $\mathsf{silp}{\mathcal{A}}$, is less than or equal to $n$.

\begin{lem}
\label{lemcounituntimonoepi}
Let $\mathcal{A}$ and $\mathcal{B}$ be abelian categories. Assume that $\mathcal{B}$ has enough projectives and enough injectives.
\begin{enumerate}[(i)]
\item\label{gorenstein:i} If $\mathsf{spli}{\mathcal{B}}\leq 1$, then the counit of the adjunction $(\mathsf{L},\mathsf{R})$ is a pointwise monomorphism.

\item\label{gorenstein:ii} If $\mathsf{silp}{\mathcal{B}}\leq 1$, then the unit of the adjunction $(\mathsf{L},\mathsf{R})$ is a pointwise epimorphism.
\end{enumerate}
\end{lem}
\begin{proof}
We only prove (i), the proof of (ii) is dual. Let $J\in \inj \mathcal{B}$. Take a projective presentation $0\to P_1\to P_0\to J\to 0$ which exists as the projective dimension of all injective objects of $\mathcal{B}$ is smaller or equal to $1$. Let $Z\in \Fun(\inj\mathcal{B}, \mathcal{A})$. As $\mathsf{R}(Z)$ is left exact, this yields the following diagram by the definition of $\mathsf{L}$:
\[
\begin{tikzcd}
&\mathsf{R} (Z)(P_1)\arrow{r}\arrow{d}{\id} &\mathsf{R} (Z)(P_0)\arrow{r}\arrow{d}{\id} &(\mathsf{L}\mathsf{R}) (Z)(J)\arrow{d}{(\counit{\mathsf{L}}{\mathsf{R}}_{Z})_J}\arrow{r} &0\\
0\arrow{r}&\mathsf{R} (Z)(P_1)\arrow{r}&\mathsf{R} (Z)(P_0)\arrow{r}&\mathsf{R} (Z)(J)=Z(J)
\end{tikzcd}
\]
Applying the snake lemma yields that $(\counit{\mathsf{L}}{\mathsf{R}})_{Z}$ is a monomorphism.
\end{proof}

As a consequence of Lemma~\ref{lemcounituntimonoepi} and Theorem~\ref{theoremAmaintext} there is the following result.

\begin{cor}\label{thmBmaintext}
Let $\mathcal{A}$ and $\mathcal{B}$ be abelian categories. Assume that $\mathcal{B}$ is $1$-Gorenstein. Then the counit, respectively the unit, of the adjunction $(\mathsf{L}, \mathsf{R})$ is a monomorphism, respectively  an epimorphism. In particular, the image of a right almost split morphism under $\mathsf{R}$ is either split or right almost split in $\image \mathsf{R}$. Dually, the image of a left almost split morphism under $\mathsf{L}$ is either split or right almost split in $\image \mathsf{L}$. 
\end{cor}

\begin{ex}
Ringel--Schmidmeier's original result is a special case of our setup: Let $\Lambda$ be a finite dimensional algebra and $Q$ be the $\mathbb{A}_2$-quiver $1\to 2$. Let $\mathcal{A}=\modu \Lambda$ and $\mathcal{B}=\modu \Bbbk Q$. Then $\Fun(\proj \mathcal{B},\mathcal{A})\cong \Fun(\inj \mathcal{B}, \mathcal{A})\cong \modu \Bbbk Q\otimes \Lambda$. Furthermore, via this equivalence, $\mathsf{L}$ and $\mathsf{R}$ can be identified with the cokernel and kernel functor, respectively. We point that using the methods in this section, we can replace $\Bbbk Q$ with an arbitrary $1$-Iwanaga--Gorenstein algebra $B$ and analogous results hold in this setup. In this case, one still has the equivalence $\proj\mathcal{B}\cong \inj\mathcal{B}$. Let us point out, that there are other categories $\mathcal{B}$, which do not fall in such a class but satisfy the assumptions for our main results in this section. One such example is obtained by letting $\mathcal{B}$ be  the category of at most countable abelian groups\footnote{The authors would like to thank Steffen Oppermann and Jeremy Rickard for independently providing this example.}.   
\end{ex}

\section{Phyla and species} \label{sec:phyla}

In this section, we introduce what will be our running `example' in later sections. It is a generalisation of Gabriel's concept of a species, \cite{Gab73}, which was studied intensively by Dlab and Ringel in a series of papers \cite{DR74a, DR74b, DR75, DR76}. The setup was recently generalised under different names in \cite{Li12, GLS16, LY15, Geu17, Kul17}. We propose a further generalisation in a similar vein as \cite{Fir16, Moz20}.

We call an adjunction of the form $G\dashv F\dashv G$ \emphbf{ambidextrous} following Baez \cite{Bae01}.  Ambidextrous adjunctions were first investigated by Morita in \cite{Mor65} in the situation where $F$ is the restriction functor along a ring homomorphism $R\to S$. In case of an embedding, he found that this is equivalent to the extension being Frobenius in the sense of \cite{Kas61}. Therefore, other authors call functors with coinciding left and right adjoint Frobenius functors, see e.g. \cite{MW13}.

\begin{defn}
Let $Q$ be a quiver. A \emphbf{phylum} $\mathfrak{A}$ on $Q$ is an assignment of an abelian category $\mathcal{A}_\mathtt{i}$ to each vertex $\mathtt{i}\in Q_0$ and to each arrow $\alpha\colon \mathtt{i}\to \mathtt{j}$ a pair of functors $(F_\alpha\colon \mathcal{A}_\mathtt{i}\to \mathcal{A}_\mathtt{j}, G_\alpha\colon \mathcal{A}_j\to \mathcal{A}_i)$ together with adjunction isomorphisms exhibiting $G_\alpha\dashv F_\alpha\dashv G_\alpha$. 
\end{defn}

The categorically inclined reader will realise that this assignment can equivalently be defined as a strict $2$-functor from $Q$, to the (large) strict 2-category $\Ambi$ whose objects are abelian categories, whose $1$-morphisms are pairs of functors $(F,G)$ with adjunction isomorphisms exhibiting $G\dashv F\dashv G$, and whose $2$-morphisms are natural transformations.

\begin{mex}\label{prospecies}
Let $\mathfrak{A}$ be a phylum on $Q$. Assume that each of the $\mathcal{A}_\mathtt{i}$ is equivalent to the category of finite dimensional modules over a finite dimensional algebra $\modu \Lambda_\mathtt{i}$. Since $F_\alpha\colon \modu \Lambda_\mathtt{i}\to \modu \Lambda_\mathtt{j}$ is right exact (as it is a left adjoint), by the Eilenberg--Watts theorem it is given by 
 \[F_\alpha \cong \Lambda_\alpha\otimes_{\Lambda_\mathtt{i}} -\colon \modu \Lambda_\mathtt{i}\to \modu \Lambda_\mathtt{j}\]
for some $\Lambda_\mathtt{j}$-$\Lambda_\mathtt{i}$-bimodule $\Lambda_\alpha$. 
Since $G_\alpha$ is right adjoint to $F_\alpha$ it is given by
\[G_\alpha\cong \Hom_{\Lambda_\mathtt{j}}(\Lambda_\alpha,-)\colon \modu \Lambda_\mathtt{j}\to \modu \Lambda_\mathtt{i}.\]
Now since $F_\alpha$ is left exact and $G_\alpha$ is right exact, it follows that $\Lambda_\alpha$ is projective as a right $\Lambda_\mathtt{i}$-module as well as a left $\Lambda_\mathtt{j}$-module. Again invoking the Eilenberg--Watts theorem we obtain that 
\[G_\alpha\cong \Hom_{\Lambda_\mathtt{j}}(\Lambda_\alpha,\Lambda_\mathtt{j})\otimes_{\Lambda_\mathtt{j}} -.\]
Therefore $F_\alpha$, being also right adjoint to $G_\alpha$, can be written as 
\[F_\alpha\cong \Lambda_\alpha\otimes_{\Lambda_\mathtt{i}}-\cong \Hom_{\Lambda_\mathtt{i}}(\Hom_{\Lambda_{\mathtt{j}}}(\Lambda_\alpha,\Lambda_{\mathtt{j}}),-)\]
It follows from the uniqueness assertion of Eilenberg--Watts that 
\[\Lambda_\alpha\cong \Hom_{\Lambda_{\mathtt{i}}}(\Hom_{\Lambda_{\mathtt{j}}}(\Lambda_\alpha,\Lambda_\mathtt{j}),\Lambda_{\mathtt{i}})\]
or equivalently that 
\[\Hom_{\Lambda_\mathtt{i}^{\op}}(\Lambda_\alpha,\Lambda_\mathtt{i})\cong \Hom_{\Lambda_\mathtt{j}}(\Lambda_\alpha,\Lambda_\mathtt{j})\]
as $\Lambda_{\mathtt{i}}$-$\Lambda_{\mathtt{j}}$-bimodules.
Thus, the assignment of $\Lambda_\mathtt{i}$ to $\mathtt{i}\in Q_0$ and $\Lambda_\alpha$ to $\alpha\in Q_1$ defines a dualisable pro-species in the sense of \cite{Kul17}. We write $\Lambda_{\alpha^*}$ for the bimodule $\Hom_{\Lambda_\mathtt{i}}(\Lambda_\alpha,\Lambda_\mathtt{i})$.  Dually, it is easy to see that a dualisable pro-species of algebras in the sense of \cite{Kul17} yields a phylum. 
\end{mex}

Other examples of ambidextrous adjunctions arise in Lie theory.

\begin{ex}
\begin{enumerate}[(i)]
\item Let $\mathfrak{g}$ be a finite dimensional semisimple complex Lie algebra. Denote by $\mathcal{O}\subseteq \Mod U(\mathfrak{g})$ the BGG category. Let $V$ be a finite dimensional $U(\mathfrak{g})$-module. Then tensoring with $V$ induces a functor $T\colon V\otimes_\mathbb{C} -\colon \mathcal{O}\to \mathcal{O}$, see e.g. \cite[Theorem 1.1 (d)]{Hum08}. Using that $V$ is finite dimensional, it is easy to see that $T$ has a left and a right adjoint which coincide and are equal to tensoring with $DV$. Since category $\mathcal{O}$ decomposes into a sum of blocks $\mathcal{O}=\bigoplus \mathcal{O}_\chi$, with each $\mathcal{O}_\chi$ being equivalent to the category of finite dimensional modules over a finite dimensional algebra $\Lambda_\chi$, this example is very close to the ambidextreous adjunction used for prospecies. The functors $T$ are the basis for the translation functors which proved to be important in Lie theory. We refer the reader to \cite{Hum08} and the references therein for more details on the BGG category $\mathcal{O}$. 
\item In the same spirit as (i) but without the corresponding category being equivalent to a sum of blocks equivalent to module categories consider the category of $G_rT$-modules where $G$ is a reductive algebraic group over a field of characteristic $p\neq 0$ with maximal torus $T$ and $G_r$ is its $r$th Frobenius kernel. Again tensoring with a finite dimensional module over the ground field provides a functor with both adjoints isomorphic to tensoring with its dual. We refer to \cite[Section 9.22]{Jan03} for details. 
\item\footnote{The second author wants to thank Greg Stevenson for bringing this example to his attention.} Let $Z$ and $Z'$ be smooth and proper Calabi--Yau varieties of the same dimension. Let $f\colon Z\to Z'$ be a proper map. Then the left adjoint of the inverse image functor $f^*\colon D^b(\Coh(Z'))\to D^b(\Coh (Z))$, i.e. the direct image functor, is equal to its right adjoint, the proper direct image $f_!$ as $f_!=S_{Z'} f^* S^{-1}_Z$ as $Z$ and $Z'$ are Calabi--Yau of the same dimension and therefore the Serre functors commute with $f^*$ and cancel. If $f$ is in addition flat, then the example restricts to the abelian categories $\Coh(Z)$ and $\Coh(Z')$. A concrete example is given by any isogeny of abelian varieties. 
\end{enumerate}
\end{ex}

Similarly to the classical case of quivers and species, one can define a category of representations for each phylum $\mathfrak{A}$.

\begin{defn}
Let $Q$ be a quiver. Let $\mathfrak{A}$ be a phylum on $Q$. An \emphbf{$\mathfrak{A}$-representation} is a collection $(M_\mathtt{i},M_\alpha)_{\mathtt{i}\in Q_0,\alpha\in Q_1}$ where $M_\mathtt{i}$ is an object of $\mathcal{A}_\mathtt{i}$ and $M_\alpha\colon F_\alpha(M_\mathtt{i})\to M_\mathtt{j}$ is a morphism in $\mathcal{A}_\mathtt{j}$. A \emphbf{morphism} of $\mathfrak{A}$-representations $(M_\mathtt{i},M_\alpha)\to (N_\mathtt{i},N_\alpha)$ is given by a collection $(\varphi_\mathtt{i}\colon M_\mathtt{i}\to N_\mathtt{i})_{\mathtt{i}\in Q_0}$ such that the following diagram commutes for every $\alpha\in Q_1$:
\[
\begin{tikzcd}
F_\alpha(M_\mathtt{i})\arrow{r}{M_\alpha}\arrow{d}{F_\alpha(\varphi_\mathtt{i})} &M_\mathtt{j}\arrow{d}{\varphi_\mathtt{j}}\\
F_\alpha(N_\mathtt{i})\arrow{r}{N_\alpha} &N_\mathtt{j}
\end{tikzcd}
\]
\end{defn}

It is straightforward to check that this forms an abelian category with componentwise composition which we denote by $\rep \mathfrak{A}$. We will provide a proof later on, realising it as the Eilenberg--Moore category of a right exact monad, see Example \ref{modulecategoryabelian} and Proposition \ref{eilenbergmooreexact}. 

To be more concrete and connect it to representation theory, we give three examples, how modules over triangular matrix rings can be thought of as representations over phyla, for details on how to obtain these equivalences, see \cite[Example 2.10 (c)]{Kul17}.  

\begin{ex} \phantomsection\label{runningexample}
\begin{enumerate}[(i)]
\item Consider the quiver $\mathbb{A}_2$ given by $\mathtt{1}\xrightarrow{\alpha} \mathtt{2}$. Then, as noted before, a phylum can be specified by a triple $(\Lambda_\mathtt{1},\Lambda_{\mathtt{2}},\Lambda_\alpha)$, where $\Lambda_\mathtt{1},\Lambda_\mathtt{2}$ are finite dimensional algebras and $\Lambda_\alpha$ is a $\Lambda_\mathtt{2}$-$\Lambda_\mathtt{1}$-bimodule. Its category of representations is equivalent to the category of modules over the triangular matrix ring $\begin{pmatrix}\Lambda_{\mathtt{1}}&0\\\Lambda_\alpha&\Lambda_{\mathtt{2}}\end{pmatrix}$. 
\item Consider the quiver $\mathbb{A}_3$ given by $\mathtt{1}\xrightarrow{\alpha} \mathtt{2}\xrightarrow{\beta} \mathtt{3}$. In a similar fashion, it is specified by a $5$-tuple $\Lambda_\mathtt{1}\xrightarrow{\Lambda_\alpha}\Lambda_\mathtt{2}\xrightarrow{\Lambda_\beta}\Lambda_\mathtt{3}$. Its category of representations is equivalent to the category of modules over the triangular matrix ring $\begin{pmatrix}\Lambda_\mathtt{1}&0&0\\\Lambda_\alpha&\Lambda_\mathtt{2}&0\\\Lambda_\beta\otimes_{\Lambda_2} \Lambda_\alpha&\Lambda_\beta&\Lambda_\mathtt{3}\end{pmatrix}$. 
\item Consider another quiver of type $\mathbb{A}_3$, namely $\mathtt{1}\xrightarrow{\alpha} \mathtt{3}\xleftarrow{\beta} \mathtt{2}$. Then, a phylum is specified by the $5$-tuple $\Lambda_{\mathtt{1}}\xrightarrow{\Lambda_\alpha} \Lambda_\mathtt{3}\xleftarrow{\Lambda_\beta} \Lambda_\mathtt{2}$. Its category of representations is equivalent to the category of modules over the triangular matrix ring $\begin{pmatrix}\Lambda_\mathtt{1}&0&0\\0&\Lambda_\mathtt{2}&0\\\Lambda_\alpha&\Lambda_\beta&\Lambda_\mathtt{3}\end{pmatrix}$.
\end{enumerate}
\end{ex}

We will illustrate the main concepts of the paper on these three concrete examples. 

\begin{rmk}
For a phylum $\mathfrak{A}$ on $Q$ we can identify the category of $\mathfrak{A}$-representations with sections of the Grothendieck construction of a certain functor as follows: Consider $Q$ as a category with objects $Q_0$ and morphisms being the paths in $Q$. Let $\CAT$ be the (large) category of all categories where the morphisms are given by functors. By definition, a phylum gives rise to a functor $F\colon Q\to \CAT$. It is given by sending a vertex $\mathtt{i}$ in the quiver $Q$ to the category $\mathcal{A}_\mathtt{i}$ and a path $p=\alpha_s\cdots \alpha_2\alpha_1$ to the functor $F_p:=F_{\alpha_s}\cdots F_{\alpha_2}F_{\alpha_1}$. The Grothendieck construction for $F$ is a category $\int F$ whose objects are pairs $(\mathtt{i}, M_\mathtt{i})$ where $\mathtt{i}\in Q_0$ and $M_{\mathtt{i}}\in \mathcal{A}_\mathtt{i}$. A morphism $(\mathtt{i}, M_\mathtt{i})\to (\mathtt{j},M_\mathtt{j})$ in $\int F$ consists of a pair $(p,f_p)$ where $p\colon \mathtt{i}\to \mathtt{j}$ is a morphism in $Q$ and $f_p\colon F_p(M)\to M'$ is a morphism in $\mathcal{A}_\mathtt{j}$. The composition of morphisms in $\int F$ is given by 
\[(q,f_q)\circ (p,f_p)=(q\circ p, f_q\circ F_q(f_p)).\]
Note that the category $\int F$ comes equipped with a forgetful functor $\int F\to Q$ 
sending an object $(\mathtt{i},M_\mathtt{i})$ to $\mathtt{i}$ and a morphism $(p,f_p)$ to $p$. A straightforward verification shows that the category of $\mathfrak{A}$-representations can be identified with the category of sections of this forgetful functor, i.e. the category of functors $Q\to \int F$ such that the composite $Q\to \int F\to Q$ is the identity functor. We refer to \cite[Section B1.3]{Joh02} and \cite[Section 2.1]{Lur09} for more details on the Grothendieck construction and to \cite{AK13, Asa13} for applications of the Grothendieck construction in representation theory. 
\end{rmk}

\section{Monads and comonads} \label{sec:monads}

In this section we recall basic facts on monads and comonads. Of particular interest for us are the free monad and comonad on endofunctors. The endofunctor  we have in mind with respect to phyla, could be called the `push-along' functor. Application of this functor gives the sum of the images of the functors on the arrows. For a general introduction to monads, see e.g. \cite[Chapter VI]{McL98} and \cite[Appendix A]{ARV11}. Throughout, let $\mathcal{C}$ be an additive category. We remind the reader of our convention that all functors are additive. 

\begin{defn}
\begin{enumerate}[(i)]
\item A \emphbf{monad} on $\mathcal{C}$ is a tuple $(T,\eta,\mu)$ where $T$ is a functor and $\eta\colon \Id\to T$ and $\mu\colon T\circ T\to T$ are natural transformations such that the diagrams
\[\begin{tikzcd}
T^3\arrow{r}[swap]{T(\mu)}\arrow{d}{\mu_T}&T^2\arrow{d}{\mu}\\
T^2\arrow{r}[swap]{\mu}&T
\end{tikzcd}
\quad \text{ and }
\begin{tikzcd}
&T\arrow[equals]{rd}\arrow[equals]{ld}\\
T\arrow{r}[swap]{T(\eta)}&T^2\arrow{u}{\mu}&T\arrow{l}{\eta_T}
\end{tikzcd}\]
commute. 
\item Dually, a \emphbf{comonad} on $\mathcal{C}$ is a tuple $(W,\varepsilon,\Delta)$ where $W$ is a functor and $\varepsilon\colon W\to \Id$ and $\Delta\colon W\to W\circ W$ are natural transformations such that $\Delta_W\circ \Delta=W(\Delta)\circ \Delta$ and $W(\varepsilon)\circ \Delta=\Id=\varepsilon_W\circ \Delta$.  
\end{enumerate}  

By slight abuse of notation, we often write just $T$ for the monad $(T,\eta,\mu)$ and just $W$ for the comonad $(W,\varepsilon, \Delta)$. 
\end{defn}

There is a close connection between adjunctions and monads, see also Proposition \ref{eilenbergmooreadjunction}.

\begin{rmk}\label{adjunctionmonad}
Let $L\colon \mathcal{C}\to \mathcal{D}$ be a functor. Let $L\dashv R$ be an adjunction. 
Then, it is well-known, see e.g. \cite[Section VI.1, p.138]{McL98} that $(LR,\counit{L}{R},L(\unit{L}{R}_{R}))$ is a comonad on $\mathcal{C}$ while $(RL,\unit{L}{R},R(\counit{L}{R}_{L}))$ is a monad on $\mathcal{D}$.
\end{rmk}

Another standard construction of monads is the association of the free monad and free comonad to  endofunctors. We only use the cases $s=1,2$ in the sequel.

\begin{defn}\label{free monad/comonad}
Let $s\in \mathbb{N}$.
\begin{enumerate}[(i)]
\item  Let $X_1,\dots X_s\colon \mathcal{C}\to \mathcal{C}$ be endofunctors preserving (countable) coproducts. For a word $w$ in $\{1,\dots,s\}$ define $X_w$ inductively as $X_{\emptyset}=\Id$ and whenever $w$ can be written as a concatenation $(i,w')$ with $i\in \{1,\dots,s\}$ and a word $w'$ then $X_w=X_i\circ X_{w'}$. Note that this implies that $X_w=X_{w'}\circ X_{w''}$ whenever $w=(w',w'')$ for $w',w''$ subwords of $w$. Assume that for all $M\in \mathcal{C}$, the coproduct $\coprod_{w\text{ word}} X_w(M)$ exists in $\mathcal{C}$. Then, for a map $f\colon M\to N$, the map 
\[\coprod_{w \text{ word}} X_w(f)\colon \coprod_{w \text{ word}} X_w(M)\to \coprod_{w \text{ word}} X_w(N)\]
can be defined on the component $X_v(M)$ as the composition of $X_v(f)$ followed by the inclusion $\iota_v\colon X_v(N)\to \coprod_{w \text{ word}} X_w(N)$, where $\iota_v$ denotes the $v$-th canonical inclusion. This way 
\[T(X_1,\dots,X_s)\colon \mathcal{C}\to \mathcal{C},\, M\mapsto \coprod_{w \text{ word}} X_w(M),\, f\mapsto \coprod_{w \text{ word}} X_w(f)\] defines a functor. Define the natural transformation $\eta\colon \Id\rightarrow T(X_1,\dots,X_s)$ as $\eta_M=\iota_{\emptyset}$, the inclusion of $M$ as the component corresponding to the empty word.  Note that $T(X_1,\dots,X_s)( T(X_1,\dots,X_s)(M))=\coprod_{w' \text{ word}} X_{w'}(\coprod_{w'' \text{ word}} X_{w''}(M))\cong \coprod_{w',w''} X_{w'}(X_{w''}(M))$. Thus, we can define 
\[\mu\colon T(X_1,\dots,X_s)\circ T(X_1,\dots,X_s)\rightarrow T(X_1,\dots,X_s)\] componentwise as the identification $X_{w'}(X_{w''}(M))\to X_{(w',w'')}(M)$. This way 
\[(T(X_1,\dots,X_s),\eta,\mu)\] defines a monad, called the \emphbf{free monad} on $X_1,\dots,X_s$.
\item Dually,  let $Y_1,\dots,Y_s\colon \mathcal{C}\to \mathcal{C}$ be endofunctors preserving (countable) products. Assume that for all $M\in \mathcal{C}$, the product $\prod_{w \text{ word}} Y_w(M)$ exists. Then, similarly as before, $M\mapsto \prod_{w \text{ word}} Y_w(M)$ defines the \emphbf{free comonad} $W(Y_1,\dots,Y_s)$ on $Y_1,\dots,Y_s$. 
\end{enumerate}
\end{defn}

Our terminology comes from the analogy with the tensor algebra of e.g. a vector space. We will not use this fact, but (provided $\mathcal{C}$ is essentially small), the association $(X_1,\dots,X_s)\mapsto T(X_1,\dots,X_s)$ (resp. $(Y_1,\dots,Y_s)\to W(Y_1,\dots,Y_s))$ defines a functor from the category of $s$-tuples of endofunctors on $\mathcal{C}$ to monads on $\mathcal{C}$. The functor $T(-)$ is left adjoint to the `forgetful diagonal' functor which sends a monad $(T,\eta,\mu)$ to the $s$-tuple $(T,\dots,T)$. Dually, the functor $W(-)$ is right adjoint to the functor sending a comonad $(W,\varepsilon, \Delta)$ to the $s$-tuple $(W,\dots,W)$.

\begin{rmk}\label{justonefunctor}
Let $s\in \mathbb{N}$. Let $X_1,\dots,X_s\colon \mathcal{C}\to \mathcal{C}$ be endofunctors preserving (countable) coproducts. Assume that for all $M\in \mathcal{C}$ the coproduct $\coprod_{w \text{ word}}X_w(M)$ exists in $\mathcal{C}$. Then the functor $X=\coprod_{i=1}^s X_i$ preserves countable coproducts and for all $M\in \mathcal{C}$, the coproduct $\coprod_{j\geq 0} X^j(M)$ exists and there is a natural isomorphism of monads
\[T(X_1,\dots,X_s)\rightarrow T(X).\]
A similar remark applies to free comonads. For this reason, we will mostly formulate our results in this section only in the case of $s=1$. 
\end{rmk}

\begin{mex}\label{runningexample1}
Let $\mathfrak{A}$ be a phylum on a locally bounded quiver $Q$. Let $\mathcal{C}:=\prod_{\mathtt{i}\in Q_0} \mathcal{A}_\mathtt{i}$. Define $X\colon \mathcal{C}\to \mathcal{C}$ and $Y\colon \mathcal{C}\to \mathcal{C}$ via 
\[X((M_\mathtt{i})_{\mathtt{i}\in Q_0})=\left(\coprod_{\substack{\alpha\in Q_1\\t(\alpha)=\mathtt{j}}} F_\alpha(M_{s(\alpha)})\right)_{\mathtt{j}\in Q_0} \text{ and } Y((M_\mathtt{i})_{\mathtt{i}})=\left(\prod_{\substack{\alpha\in Q_1\\s(\alpha)=\mathtt{j}}} G_\alpha(M_{t(\alpha)})\right)_{\mathtt{j}\in Q_0} \]
on objects and extend them to functors in the obvious way. As $Q$ is locally bounded, the sum on each component is finite, thus the functor is well-defined and the coproducts and products are actually biproducts. As the $F_\alpha$ are left adjoints, they preserve (countable) coproducts. As in each component, $X$ is defined as a finite sum of $F_\alpha$, $X$ also preserves (countable) coproducts. Locally boundedness of $Q$ also guarantees that $\coprod_{i\geq 0} X^i(M)$ exists. Thus, we obtain a monad $T(X)$. The argument for $W(Y)$ is similar. 
\end{mex}

\begin{ex}\label{runningexample6}
In the terminology of Meta-Example \ref{prospecies} and Example \ref{runningexample}:
\begin{enumerate}[(i)]
\item Let $M=(M_\mathtt{1},M_\mathtt{2})\in \modu(\Lambda_\mathtt{1}\times \Lambda_\mathtt{2})$, then $X(M)=(0,\Lambda_\alpha\otimes_{\Lambda_\mathtt{1}} M_\mathtt{1})$ and $Y(M)=(\Lambda_{\alpha^*}\otimes_{\Lambda_{\mathtt{2}}} M_{\mathtt{2}},0)$. 
\item Let $M=(M_\mathtt{1},M_\mathtt{2},M_\mathtt{3})\in \modu (\Lambda_\mathtt{1}\times \Lambda_\mathtt{2}\times \Lambda_{\mathtt{3}})$. Then $X(M)=(0,\Lambda_\alpha\otimes_{\Lambda_\mathtt{1}} M_{\mathtt{1}}, \Lambda_\beta\otimes_{\Lambda_{\mathtt{2}}} M_\mathtt{2})$ and $Y(M)=(\Lambda_{\alpha^*}\otimes_{\Lambda_\mathtt{2}} M_\mathtt{2}, \Lambda_{\beta^*}\otimes_{\Lambda_{\mathtt{3}}} M_{\mathtt{3}},0)$. 
\item Let $M=(M_\mathtt{1},M_\mathtt{2},M_\mathtt{3})\in \modu (\Lambda_\mathtt{1}\times \Lambda_\mathtt{2}\times \Lambda_\mathtt{3})$. Then $X(M)=(0,0,(\Lambda_\alpha\otimes_{\Lambda_{\mathtt{1}}} M_\mathtt{1})\oplus (\Lambda_\beta \otimes_{\Lambda_{\mathtt{2}}} M_\mathtt{2}))$ and $Y(M)=(\Lambda_{\alpha^*}\otimes_{\Lambda_\mathtt{3}} M_3,\Lambda_{\beta^*}\otimes_{\Lambda_{\mathtt{3}}} M_3,0)$. 
\end{enumerate}
\end{ex}

\begin{defn}
Let $L_1,L_2\colon \mathcal{C}\to \mathcal{D}$ and $R_1,R_2\colon \mathcal{D}\to \mathcal{C}$ be functors. Assume $L_1\dashv R_1$ and $L_2\dashv R_2$ are adjunctions. Following \cite{McL98}, two natural transformations $\phi\colon L_1\to L_2$ and $\phi'\colon R_2\to R_1$ are \emphbf{conjugate} if the square
\[
\begin{tikzcd}
\mathcal{D}(L_2(M),N)\arrow{d}{\mathcal{D}(\phi_M,N)}\arrow{r}{\adj^{L_2\dashv R_2}} & \mathcal{C}(M,R_2(N))\arrow{d}{\mathcal{C}(M,\phi'_N)} \\
\mathcal{D}(L_1(M),N)\arrow{r}{\adj^{L_1\dashv R_1}} & \mathcal{C}(M,R_1(N))
\end{tikzcd}
\]
commutes for all $M\in \mathcal{C}$ and all $N\in \mathcal{D}$. We say that $\phi'$ is the \emphbf{right conjugate} of $\phi$ or $\phi$ is the \emphbf{left conjugate} of $\phi'$ in this case. It follows from \cite[Theorem IV.7.2]{McL98} that the left and right conjugate of a given natural transformation is unique. 
\end{defn}

From the proof of \cite[Theorem 2.15]{Lau06}, the following result can be extracted:

\begin{prop}\label{Conjugate units and counits}
Let $L, R\colon \mathcal{C}\to \mathcal{D}$ and $G\colon \mathcal{D}\to \mathcal{C}$ be functors. Assume that there are adjunctions $L\dashv G$ and $G\dashv R$. 
Then $\unit{L}{G}$ and $\counit{G}{R}$ are conjugate, and $\counit{L}{G}$ and $\unit{G}{R}$ are conjugate.
\end{prop}


For monads and comonads, there is a stronger sense in which they can be adjoint. This forces the corresponding Eilenberg--Moore categories to be equivalent, see Proposition \ref{eilenbergmooreequivalent}. 

\begin{defn}
Let $T$ (resp. $W$) be a monad (resp. comonad) on $\mathcal{C}$. Then $T$ is said to be \emphbf{right adjoint} to $W$ \emphbf{(in the monadic sense)} if there exists an adjunction $\adj^{W\dashv T}_{M,N}\colon \begin{tikzcd}\mathcal{C}(W(M),N)\arrow{r}{\sim}  &\mathcal{C}(M,T(N))\end{tikzcd}$ such that 
$\mu$ is right conjugate to $\Delta$ and $\eta$ is right conjugate to $\varepsilon$, i.e. the following two diagrams commute for all $M,N\in \mathcal{C}$:
\[
\begin{tikzcd}
\mathcal{C}(W(M),N)\arrow{r}{\adj^{W\dashv T}_{M,N}} &\mathcal{C}(M,T(N))&\mathcal{C}(W(M),N)\arrow{rr}{\adj^{W\dashv T}_{M,N}} &&\mathcal{C}(M,T(N))\\
\mathcal{C}(W^2(M),N)\arrow{r}{\adj^{W^2\dashv T^2}_{M,N}}\arrow{u}{\mathcal{C}(\Delta_M,N)} &\mathcal{C}(M,T^2(N))\arrow{u}[swap]{\mathcal{C}(M,\mu_N)}&&\mathcal{C}(M,N)\arrow{lu}{\mathcal{C}(\varepsilon_M,N)}\arrow{ru}[swap]{\mathcal{C}(M,\eta_N)}
\end{tikzcd}
\]
We say that $T$ is \emphbf{left adjoint} to $W$ \emphbf{(in the monadic sense)}, if there exists an adjunction $\adj^{T\dashv W}_{M,N}\colon \begin{tikzcd}\mathcal{C}(T(M),N)\arrow{r}{\sim}  &\mathcal{C}(M,W(N))\end{tikzcd}$ such that 
$\mu$ is left conjugate to $\Delta$ and $\eta$ is left conjugate to $\varepsilon$.
\end{defn}

Given an adjunction between endofunctors, the corresponding free monad and cofree comonad are adjoint in the monadic sense:

\begin{lem}\label{naturalmonadicadjunction}
Let $X,Y\colon \mathcal{C}\to \mathcal{C}$ be endofunctors with $X\dashv Y$.  Assume further that $\coprod_{i\geq 0} X^i(M)$ and $\prod_{i\geq 0} Y^i(M)$ exist for all $M\in \mathcal{C}$. Then $(T(X),\eta,\mu)$ is left adjoint to $(W(Y),\varepsilon,\Delta)$ in the monadic sense. 
\end{lem}

\begin{proof}
Firstly observe that since $X$ is a left adjoint, it commutes with (countable) coproducts and dually $Y$ commutes with (countable) products as it is a right  adjoint.

Secondly, note that as $X\dashv Y$, it follows that $X^i\dashv Y^i$ for all $i\geq 0$. To show that $T(X)$ and $W(Y)$ are adjoint as functors, note that 
\[\mathcal{C}\left(\coprod_{i\geq 0} X^i(M),M'\right)\cong \prod_{i\geq 0} \mathcal{C}(X^i(M),M')\cong \prod_{i\geq 0} \mathcal{C}(M,Y^i(M')) \cong \mathcal{C}\left(M,\prod_{i\geq 0} Y^i(M')\right).\]
It is straightforward to show that $\eta$ and $\varepsilon$, resp. $\mu$ and $\Delta$ are conjugate.
\end{proof}

\begin{mex}
Let $\mathfrak{A}$ be a phylum on a locally bounded quiver $Q$. Let $\mathcal{C}$, $X$ and $Y$ be as in Meta-Example \ref{runningexample1}. Then it is straightforward to check that $X$ is left adjoint to $Y$.
\end{mex}

Recall that for objects $M_\mathtt{i}$ in some additive category, we write $\bigoplus_{i} M_i$ if both the coproduct $\coprod_{i} M_i$ and $\prod_i M_i$ exist and are isomorphic via the natural homomorphism. By imposing the strong assumption that the corresponding biproducts exist, we can also force an adjunction in the other direction. The condition might seem quite strong but has recently been studied in several contexts, see e.g. \cite{Iov06} for the case of an arbitrary collection and \cite{EH17} for the case of tensor powers of an object in a monoidal category. 

\begin{lem}\label{unnaturalmonadicadjunction}
Let $\mathcal{C}$, $X$, and $Y$ be as in Assumption \ref{standard1}. 
Then, $(W(Y),\varepsilon,\Delta)$ is left adjoint to $(T(X),\eta,\mu)$ in the monadic sense.
\end{lem}

\begin{proof}
As the products and coproducts which need to be preserved are actually biproducts, and $X$ and $Y$ are adjoint and therefore preserve coproducts resp.  products, it follows that the free monad $T(X)$ and the free comonad $W(Y)$ are well-defined. In a similar fashion as before, using that $\coprod X^i$ is naturally isomorphic to $\prod X^i$, and similarly for $Y$, it follows that $W(Y)=T(Y)\dashv W(X)=T(X)$ as functors. One can easily verify that $\eta, \varepsilon$ and $\Delta,\mu$ are conjugate. 
\end{proof}

This quite strong assumption is satisfied in our example:

\begin{mex}\label{assumptionssatisfiedforphyla}
Let $\mathfrak{A}$ be a phylum on a locally bounded quiver. Let $X,Y$ be as in Example \ref{runningexample1}. As $\mathcal{C}=\prod_{\mathtt{i}\in Q_0} \mathcal{A}_\mathtt{i}$, sums and products are computed componentwise and since  $Q$ is locally bounded, for $M\in \mathcal{C}$ the coproduct $\coprod_{i\geq 0} X^i(M)$ on each component is finite, thus isomorphic to the product. Therefore, the condition that $\coprod X^i\cong \prod X^i$ is satisfied. The same holds for $Y$.
\end{mex}

\begin{rmk} 
\begin{enumerate}[(i)]
\item Note that the condition that the biproduct exists does not necessarily imply that $X$ is nilpotent. For example, take $Q$ to be the union of linearly oriented $\mathbb{A}_n$-quivers, one for every $n$. Set $\mathcal{A}_\mathtt{i}=\modu \mathbbm{k}$ and set $F_\alpha=\Id$.  Then, $X^n(M)$ does not vanish for any $n$ with $M$ defined by $M_{\mathtt{i}}=\mathbbm{k}$ for all $\mathtt{i}\in Q_0$. 
\item A more extreme example is given by $\mathcal{C}$ being the module category of a quiver of the form $\mathbb{A}_\infty$ obtained from quivers of the form 
$\begin{tikzcd}\bullet\arrow{r}&\dots\arrow{r}&\bullet&\dots\arrow{l}&\bullet\arrow{l}\end{tikzcd}$
with $n$ arrows going left and right, by gluing them together at the left and rightmost vertex while varying $n$ from $1$ to $+ \infty$. The $X$ as before can be lifted to that setting. Here, the module $M$ as before gives an example of an indecomposable such that $X^n(M)\neq 0$ for every $n$. But as $Q$ is locally bounded, still $\coprod X^i\cong \prod X^i$. 
\item Let $Q$ be the $\mathbb{A}_\infty$-quiver with all arrows  pointing towards the unique sink, resp. originating from a unique source,, i.e. 
\[
\begin{tikzcd}[row sep=1ex]
Q=\bullet&\bullet\arrow{l}&\bullet\arrow{l}&\dots\arrow{l}
\end{tikzcd}
\]
Let $\Lambda_{\mathtt{i}}=\Bbbk$ for all $\mathtt{i}\in Q_0$ and $\Lambda_\alpha=\Bbbk$ for all $\alpha\in Q_1$. Let $M$ be the representation of $\Lambda$ defined by $M_\mathtt{i}=\Bbbk$ for all $\mathtt{i}\in Q_0$  and $M_\alpha=\id$ for all $\alpha\in Q_1$. Then it is easy to see that $\coprod Y^i(M)\cong \prod Y^i(M)$ but $\coprod X^i(M)\ncong \prod X^i(M)$  An analogous example for $Q^{\op}$ gives $\coprod X^i(M)\cong \prod X^i(M)$ but $\coprod Y^i(M)\ncong \prod Y^i(M)$. 
\end{enumerate}
\end{rmk}

Attached to a monad $T$ is another category, called the Eilenberg--Moore category or the category of $T$-algebras. We will use the former terminology, since in our context the standard example will be representations, not rings. 

\begin{defn}\label{eilenbergmoore}
Let $(T,\eta,\mu)$ be a monad on $\mathcal{C}$. Then, the \emphbf{Eilenberg--Moore category} $\mathcal{C}^T$ is defined to be the category with:
\begin{description}
\item[objects] pairs $(M,h)$, where $M\in \mathcal{C}$ and $h\in \mathcal{C}(T(M),M)$, such that the following diagrams commute:
\[\begin{tikzcd}
T^2(M)\arrow{r}{\mu_M}\arrow{d}{T(h)} &T(M)\arrow{d}{h}\\
T(M)\arrow{r}{h}&M
\end{tikzcd}\qquad
\begin{tikzcd}
T(M)\arrow{r}{h}&M\arrow[equals]{ld}\\
M\arrow{u}{\eta_M}
\end{tikzcd}\]
\item[morphisms] $\mathcal{C}^T((M,h),(M',h'))$ is the set given by those $\varphi\in \mathcal{C}(M,M')$ such that the following diagram commutes:
\[
\begin{tikzcd}
M\arrow{r}{\varphi}&M'\\
T(M)\arrow{u}{h}\arrow{r}{T(\varphi)}&T(M')\arrow{u}{h'}
\end{tikzcd}\]
\item[composition] Composition is given by the composition in $\mathcal{C}$.
\item[unit] The unit is given by the identity map $1_M\in \mathcal{C}(M,M)$.
\end{description} 
Dually, for $(W,\varepsilon,\Delta)$ a comonad on $\mathcal{C}$ one can define the \emphbf{(co-)Eilenberg--Moore category} $\mathcal{C}^W$ as the category of certain morphisms $N\to W(N)$. 
\end{defn}

A similar construction to the Eilenberg--Moore category for two endofunctors is given by the following definition. The notation is inspired by the standard notation $(F\downarrow G)$ for the comma category of morphisms $F(M)\to G(N)$.

\begin{defn}
Let $\mathcal{C}$, $\mathcal{D}$ be additive categories.
Let $F_1,\dots,F_s,G\colon \mathcal{C}\to \mathcal{D}$ be functors. Then the category $(F_1,\dots,F_s\Downarrow G)$ is given by objects being tuples $(M,h_1,\dots,h_s)$ where $h_i\colon F_i(M)\to G(M)$ are morphisms in $\mathcal{D}$. Morphisms $(M,h_1,\dots,h_s)\to (N,g_1,\dots,g_s)$ are given by morphisms $\varphi\colon M\to N$ in $\mathcal{C}$ such that the following diagram commutes for each $i=1,\dots,s$:
\[
\begin{tikzcd}
F_i(M)\arrow{r}{h_i}\arrow{d}{F_i(\varphi)} &G(M)\arrow{d}{G(\varphi)}\\
F_i(N)\arrow{r}{g_i} &G(N)
\end{tikzcd}
\]
Composition and identities are induced from composition and identities in $\mathcal{C}$. 

Dually for functors $F,G_1,\dots,G_t\colon \mathcal{C}\to \mathcal{D}$ one defines the category $(F\Downarrow G_1,\dots,G_t)$. 
\end{defn}

\begin{rmk}
A similar remark to Remark \ref{justonefunctor} applies: If $\mathcal{D}$ is additive, then $(F_1,\dots,F_s\Downarrow G)\cong (\coprod_{i=1}^s F_i\Downarrow G)$ and $(F\Downarrow G_1,\dots,G_t)\cong (F\Downarrow \prod_{i=1}^t G_i)$. Therefore in the sequel we mostly consider the case $s=t=1$. 
\end{rmk}

The following lemma is well-known, but the authors were not able to find a published reference in the stated setting for it. 

\begin{lem}\label{comma=EilenbergMoore}
Let $\mathcal{C}$ be an additive category. 
Let $X\colon \mathcal{C}\to \mathcal{C}$ be an endofunctor preserving (countable) coproducts. Assume that for all $M\in \mathcal{C}$ the coproduct $\coprod X^i(M)$ exists in $\mathcal{C}$. Let $T(X)$ be the associated monad. Then, the functor $\mathcal{C}^{T(X)}\to (X\Downarrow \Id)$ given by $(M,h)\mapsto (M,h_1)$ where $h_1\colon X(M)\to M$ is the composition of the canonical inclusion $X(M)\to T(X)(M)$ followed by $h$, is an isomorphism of categories. 
Dually, the categories $\mathcal{C}^{W(Y)}$ and $(\Id\Downarrow Y)$ are isomorphic for an endofunctor $Y$ preserving countable products. 
\end{lem}

\begin{proof}
In order for $h\colon T(X)(M)\to M$ to be in the Eilenberg--Moore category, we need the following diagram to commute:
\[
\begin{tikzcd}
T(X)^2(M)\arrow{r}{\mu_M}\arrow{d}{T(X)(h)}&T(X)(M)\arrow{d}{h}\\
T(X)(M)\arrow{r}{h}&M
\end{tikzcd}
\]
Restricting to the component starting in $X^i(X(M))$ in the upper left corner, we get that the following diagram needs to commute where $h_i$ denotes the composition of the canonical inclusion $X^i(M)\to T(X)(M)$ followed by $h$:
\[
\begin{tikzcd}[column sep = 2cm]
X^i(X(M))\arrow{r}{\mu_T|_{X^i(X(M))}=\id}\arrow{d}{X^i(h_1)} &X^{i+1}(M)\arrow{d}{h_{i+1}}\\
X^i(M)\arrow{r}{h_i}&M
\end{tikzcd}
\]
It follows by induction that all $h_i$ are determined by $h_1$. In the other direction a map $T(X)(M)\to M$ is given by $h_0=\id_M$, the given $h_1$ and $h_{i+1}=h_i\circ X^i(h_1)$.
The proof for morphisms is similar. 
\end{proof}

From now on, in the situation of the preceding lemma we will identify the Eilenberg--Moore category $\mathcal{C}^{T(X)}$ with the comma category $(X\Downarrow \Id)$. Also, we write $(M,h)$ where $h\colon T(X)(M)\to M$ and $(M,h_1)$ where $h_1\colon X(M)\to M$ for an object in this category interchangeably.

As mentioned before, categories of representations of phyla can be realised as Eilenberg--Moore categories over the free monad associated to the `push-along' functor. 

\begin{mex}\label{modulecategoryabelian}
Let $\mathfrak{A}$ be a phylum on a locally bounded quiver $Q$. Let $\mathcal{C}$, $X$, $Y$ be as in Meta-Example \ref{runningexample1}. Then, $\mathcal{C}^{T(X)}=\rep \mathfrak{A}$. 
\end{mex}

The following proposition gives the well-known correspondence between monads and adjunctions, see e.g. \cite[Theorem VI.2.1]{McL98}.

\begin{prop}\label{eilenbergmooreadjunction}
Let $\mathcal{C}$ be a category and $(T, \eta, \mu)$ be a monad on $\mathcal{C}$. Then, there is an adjunction $F_T\dashv U_T$ where $F_T\colon \mathcal{C}\to \mathcal{C}^T$ and $U_T\colon \mathcal{C}^T\to \mathcal{C}$ are given by $U_T(M,h)=M$ and $F_T(M)=(T(M),\mu_M)$. 
The monad associated to the adjunction $(F_T,U_T)$ as in Remark \ref{adjunctionmonad} is $T$. Dually, for a comonad $(W, \varepsilon, \Delta)$ on $\mathcal{C}$, there is an adjunction $U_W\dashv F_W$ where $U_W\colon \mathcal{C}^W\to \mathcal{C}$ and $F_W\colon \mathcal{C}\to \mathcal{C}^W$given by $U_W(M,h)=M$ and $F_W(M)=(W(M),\Delta_M)$. The comonad associated to the adjunction $(U_W,F_W)$ as in Remark \ref{adjunctionmonad} is $W$. 
\end{prop}

\begin{rmk}\label{eilenbergmooreforfreemonad}
In the case of the free monad $(T(X),\eta,\mu)$ the adjunction 
\[\mathcal{C}(M',U_{T(X)}(M,h))\xrightarrow{\cong} \mathcal{C}^{T(X)}(F_{T(X)}(M'),(M,h))\]
sends a morphism $g\colon M'\to M$ to a morphism $F_{T(X)}(M')\to (M,h)$ whose underlying map in $\mathcal{C}$ is given by $(g_i)_{i\geq 0}\colon \coprod_{i\geq 0} X^i(M')\to M$ where
\[g_i=h_1\circ X(h_1)\circ \dots\circ X^{i-1}(h_1)\circ X^i(g)\colon X^i(M')\to M.\]
A dual description holds for $\mathcal{C}(U_{W(Y)}(M,h), M')\xrightarrow{\cong} \mathcal{C}^{W(Y)}((M,h),F_{W(Y)}(M'))$.
\end{rmk}

The Eilenberg--Moore category is particularly well-behaved if $T$ is in addition right exact. In particular, this holds in the examples we have in mind, as the monads we consider come with an ambidextrous adjunction.

\begin{prop}[{\cite[Proposition 5.3]{EM65}}]\label{eilenbergmooreexact}
Let $\mathcal{C}$ be an abelian category and let $T$ be a right exact monad. Then the Eilenberg--Moore category $\mathcal{C}^T$ is abelian and the forgetful functor $U_T\colon \mathcal{C}^T\to \mathcal{C}$ preserves and reflects exact sequences. Dually, if $W$ is a left exact comonad, then the Eilenberg--Moore category $\mathcal{C}^W$ is abelian and the forgetful functor $U_W\colon \mathcal{C}^W\to \mathcal{C}$ preserves and reflects exact sequences. 
\end{prop}

We define four functors, which are generalisations of the functor associating to a vector space the corresponding simple module and to a module its top or socle in the case of bound quiver algebras. 

\begin{defn}
Let $\mathcal{C}$ be an abelian category and $X, Y\colon \mathcal{C}\to \mathcal{C}$ be endofunctors. 
\begin{enumerate}[(i)]
\item The \emphbf{(left) simple functor} $S\colon \mathcal{C}\to (X\Downarrow \Id)$ is defined as $M\mapsto (X(M)\xrightarrow{0} M)$.
\item The \emphbf{(right) simple functor} $S'\colon \mathcal{C}\to (\Id \Downarrow Y)$ is defined by $M\mapsto (M\xrightarrow{0} Y(M))$.
\item The \emphbf{top functor} $\Kopf_X\colon (X\Downarrow \Id)\to \mathcal{C}$ is defined by $(M,h_1)\mapsto \coker h_1$.
\item The \emphbf{socle functor} $\soc_Y\colon (\Id\Downarrow Y)\to \mathcal{C}$ is defined by $(M,h_1)\mapsto \ker h_1$. 
\end{enumerate}
\end{defn}

\begin{rmk}
In case that $X\dashv Y$ it is easy to see that the categories $(X\Downarrow \Id)$ and $(\Id\Downarrow Y)$ are equivalent via the adjunction isomorphism. Under the equivalence, $S$ is sent to $S'$. 
\end{rmk}

\begin{ex}
We use the notation of Examples \ref{runningexample} and \ref{runningexample6}. For an arrow $\alpha\in Q_1$ denote by $M_{\alpha^*}\colon M_{s(\alpha)}\to \Lambda_{\alpha^*}\otimes_{\Lambda_{t(\alpha)}} M_{t(\alpha)}$ the image of $M_\alpha\colon \Lambda_\alpha\otimes_{\Lambda_{s(\alpha)}} M_{s(\alpha)}$ under the adjunction isomorphism corresponding to the adjuntion $\Lambda_\alpha\otimes_{\Lambda_{s(\alpha)}}\dashv \Lambda_{\alpha^*}\otimes_{t(\alpha)}-$. We give an explicit description of $\Kopf_X$ and $\soc_Y$ for the phyla in Example \ref{runningexample6}.
\begin{enumerate}[(i)]
\item For $(M_\mathtt{1},M_\mathtt{2},M_\alpha)\in \rep(\Lambda_\mathtt{1}\xrightarrow{\Lambda_\alpha} \Lambda_\mathtt{2})$ we obtain $\Kopf_X(M_\mathtt{1},M_\mathtt{2},M_\alpha)=(M_\mathtt{1},\coker M_\alpha)$ and $\soc_Y(M_\mathtt{1},M_\mathtt{2},M_\alpha)=(\ker M_{\alpha^*}, M_\mathtt{2})$.
\item For $(M_\mathtt{1},M_\mathtt{2},M_\mathtt{3},M_\alpha,M_\beta)\in \rep(\Lambda_\mathtt{1}\xrightarrow{\Lambda_\alpha} \Lambda_\mathtt{2}\xrightarrow{\Lambda_\beta} \Lambda_\mathtt{3})$ we obtain 
\begin{equation*}
\begin{split}\Kopf_X (M_\mathtt{1},M_\mathtt{2},M_\mathtt{3}, M_\alpha,M_\beta)&=(M_\mathtt{1},\coker M_\alpha, \coker M_\beta)\\\text{ and }\soc_Y(M_\mathtt{1},M_\mathtt{2},M_\mathtt{3}, M_\alpha,M_\beta)&=(\ker M_{\alpha^*}, \ker M_{\beta^*}, M_\mathtt{3}).
\end{split}
\end{equation*} 
\item For $(M_\mathtt{1},M_\mathtt{2},M_\mathtt{3},M_\alpha,M_\beta)\in \rep(\Lambda_\mathtt{1}\xrightarrow{\Lambda_\alpha} \Lambda_\mathtt{3}\xleftarrow{\Lambda_\beta} \Lambda_\mathtt{2})$ we obtain 
\begin{equation*}
\begin{split}\Kopf_X (M_\mathtt{1},M_\mathtt{2},M_\mathtt{3}, M_\alpha,M_\beta)&=(M_\mathtt{1}, M_\mathtt{2},\coker (M_\alpha, M_\beta))\\
\text{ and }\soc_Y(M_\mathtt{1},M_\mathtt{2},M_\mathtt{3}, M_\alpha,M_\beta)&=(\ker M_{\alpha^*}, \ker M_{\beta^*}, M_\mathtt{3}).
\end{split}
\end{equation*}
\end{enumerate}
\end{ex}

\begin{lem}\label{topisaleftadjoint}
Let $\mathcal{C}$ be an abelian category. Let $X, Y\colon \mathcal{C}\to \mathcal{C}$ be endofunctors. Then, $(\Kopf_X,S)$ and $(S',\soc_Y)$ form adjoint pairs. 
\end{lem}

\begin{proof}
We only prove the first adjunction, the proof of the second is dual. We need to prove that $\mathcal{C}(\Kopf_X(M,h_1),N)\cong \mathcal{C}^{T(X)}((M,h_1),S(N))$. Denote the canonical projection $M\to \coker h_1=\Kopf_X(M,h_1)$ by $\pi$. Suppose $f\colon \Kopf_X(M,h_1)\to N$ is a morphism. Define $g\colon M\to N$ by $g=f\pi$. We have to check that the following diagram commutes:
\[
\begin{tikzcd}
X(M)\arrow{r}{h_1}\arrow{d}{X(g)} &M\arrow{r}{\pi}\arrow{d}{g} &\Kopf_X(M,h_1)\arrow{ld}{f}\arrow{r}&0\\
X(N)\arrow{r}{0} &N
\end{tikzcd}
\]
i.e. that $gh_1=0$, but this is clear since $g$ factors through the cokernel. Dually let $g\colon (M,h_1)\to S(N)$ be a morphism in $\mathcal{C}^{T(X)}$. This implies that $gh_1=0\circ X(g)$, i.e. $g$ induces a unique map $f\colon \Kopf_X(M,h_1)\to N$. It is straightforward to check that the isomorphism is natural.
\end{proof}

The following lemma is obvious from the definition and the description of the Eilenberg--Moore category in Lemma \ref{comma=EilenbergMoore}.

\begin{lem}\label{topofprojective}
Let $\mathcal{C}$ be an abelian category and let $X\colon \mathcal{C}\to \mathcal{C}$ be an endofunctor preserving (countable) coproducts such that $\coprod_{i\geq 0} X^i(M)$ exists for all $M$. Then, $U_{T(X)}\circ S=1_\mathcal{C}$ and $\Kopf_{X}\circ F_{T(X)}=1_\mathcal{C}$. Dually, for an endofunctor $Y\colon \mathcal{C}\to \mathcal{C}$ preserving (countable) products such that $\prod_{i\geq 0} Y^i(M)$ exists for all $M\in \mathcal{C}$, we get $U_{W(Y)}\circ S'=1_\mathcal{C}$ and $\soc_Y \circ F_{W(Y)}=1_\mathcal{C}$. 
\end{lem}

We finish this section with another well-known proposition regarding the Eilenberg--Moore categories of adjoint monad-comonad pairs, see e.g. \cite[Proposition 3.3]{EM65}. 

\begin{prop}\label{eilenbergmooreequivalent}
Let $T$ be a monad on $\mathcal{C}$ which is left adjoint to a comonad $W$ on $\mathcal{C}$ in the monadic sense. Then the respective Eilenberg--Moore categories $\mathcal{C}^T$ and $\mathcal{C}^W$ are isomorphic using the adjunction homomorphism.
\end{prop}

\section{Adjunctions with Nakayama functors} \label{sec:nakayamafunctors}

In this section, we recall results from the third author's paper \cite{Kva16} on Nakayama functors for adjunctions and apply them to the Eilenberg--Moore adjunction associated to a free monad. 

\begin{defn}[{\cite[Definition 1.1.1]{Kva16}}]
Let $\mathcal{A}$ and $\mathcal{C}$ be abelian categories and let $f^*\colon \mathcal{A}\to \mathcal{C}$ be a faithful functor with left adjoint $f_!\colon \mathcal{C}\to \mathcal{A}$. A \emphbf{Nakayama functor} with respect to $f_!\dashv f^*$ is a functor $\nu\colon \mathcal{A}\to \mathcal{A}$ with right adjoint $\nu^-\colon \mathcal{A}\to \mathcal{A}$ satisfying:
\begin{enumerate}[(i)]
\item $\nu\circ f_!$ is right adjoint to $f^*$.
\item The unit of the adjunction $\nu\dashv \nu^-$ induces a natural isomorphism $f_!\to \nu^-\circ \nu\circ f_!$. 
\end{enumerate}
\end{defn}

It follows from \cite[Theorem 3.3.4]{Kva16} that a Nakayama functor relative to $f_!\dashv f^*$ is unique up to an equivalence $\Phi\colon \mathcal{A}\to \mathcal{A}$ satisfying $\Phi\circ f_!=f_!$. From now on, we fix the notation $f_*:=\nu\circ f_!\colon \mathcal{C}\to \mathcal{A}$, so that we have adjunctions $f_!\dashv f^*\dashv f_*$. 
According to \cite[Theorem 3.3.5]{Kva16}, ambidextrous adjunctions (in the monadic sense) are intrically related to the existence of a Nakayama functor. Under our strong assumptions, such a situation arises. 

\begin{lem}\label{XadmitsNakayamafunctor}
Let $\mathcal{C}$, $X$, and $Y$ be as in Assumption \ref{standard1}. 
Suppose that there is an ambidextrous adjunction $Y\dashv X\dashv Y$. Then, there is an ambidextrous adjunction $W(Y)\dashv T(X)\dashv W(Y)$ in the monadic sense. In particular, the adjunction $F_{T(X)}\dashv U_{T(X)}$ has a Nakayama functor.  
\end{lem}

\begin{proof}
The adjointness of the functors follows from Lemmas \ref{naturalmonadicadjunction} and \ref{unnaturalmonadicadjunction}. The statement about the existence of the Nakayama functor then follows from \cite[Theorem 3.3.5]{Kva16}.
\end{proof}

In the situation of the above lemma, we write 
\[f_!:=F_{T(X)}\colon \mathcal{C}\to \mathcal{C}^{T(X)};\quad f^*:=U_{T(X)}\colon \mathcal{C}^{T(X)}\to \mathcal{C};\quad	 f_*:=F_{W(Y)}\colon \mathcal{C}\to \mathcal{C}^{W(Y)}\]
following the notation in Proposition \ref{eilenbergmooreadjunction}. Identifying the categories $\mathcal{C}^{W(Y)}$ and $\mathcal{C}^{T(X)}$ via the isomorphism in Proposition \ref{eilenbergmooreequivalent}, we get adjunctions $f_!\dashv  f^*\dashv f_*$. We recall from \cite{Kva16}:

\begin{defn}[{\cite[Section 4.1]{Kva16}}]
Let $f_!\dashv f^*$ be an adjunction with a Nakayama functor. Then an object is called \emphbf{relative projective} with respect to $f_!\dashv f^*$ if it is a direct summand of an object in the essential image of $f_!$. Dually, an object is called \emphbf{relative injective} with respect to $f^*\dashv f_*$ if it is a direct summand of an object in the essential image of $f_*$. 
\end{defn}


\begin{defn}[{\cite[Definition 4.1.1]{Kva16}}]\label{Gorenstein projectives}
Let $f_!\dashv f^*$ be an adjunction with a Nakayama functor $\nu$. 
\begin{enumerate}[(i)]
\item An object $G\in \mathcal{A}$ is \emphbf{relative Gorenstein projective} with respect to $f_!\dashv f^*$ if there exists an exact sequence
\[
Q_{\bullet}=\cdots \xrightarrow{f_2} Q_{1} \xrightarrow{f_1} Q_0\xrightarrow{f_0} Q_{-1} \xrightarrow{f_{-1}} \cdots 
\]
in $\mathcal{A}$, where $Q_i$ is relative projective for all $i\in \mathbb{Z}$, such that the complex $\nu(Q_{\bullet})$ is exact, and with $Z_0(Q_{\bullet})=\ker f_0=G$. We denote the full subcategory of $\mathcal{A}$ consisting of relative Gorenstein projective objects by $\relGproj{f_!}{}$. 
\item An object $G\in \mathcal{A}$ is \emphbf{relative Gorenstein injective} with respect to $f^*\dashv f_*$ if there exists an exact sequence
\[
J_{\bullet}=\cdots \xrightarrow{g_2} J_{1} \xrightarrow{g_1} J_0\xrightarrow{g_0} J_{-1} \xrightarrow{g_{-1}} \cdots 
\]
in $\mathcal{A}$, where $J_i$ is relative injective for all $i\in \mathbb{Z}$, such that the complex $\nu^-(J_{\bullet})$ is exact, and with $Z_0(J_{\bullet})=\ker g_0=G$. We denote the full subcategory of $\mathcal{A}$ consisting of relative Gorenstein injective objects by $\relGinj{I}{\mathcal{A}}$.
\end{enumerate} 
\end{defn} 

 It is easy to see that $\relGproj{P}{\mathcal{A}}$ and $\relGinj{I}{\mathcal{A}}$ contain the relative projective and the relative injective objects, respectively. Note that by \cite[Proposition 4.1.5]{Kva16} the categories $\relGproj{P}{\mathcal{A}}$ and $\relGinj{I}{\mathcal{A}}$ are closed under extensions, and are therefore exact categories. Furthermore, $\relGproj{P}{\mathcal{A}}$ has the special properties that it is generating (i.e. for every $M\in \mathcal{A}$ there exists $G\in \relGproj{P}{\mathcal{A}}$ and an epimorphism $G\twoheadrightarrow M$) and that it is closed under direct summands and kernels of epimorphisms, i.e. it is a \emphbf{resolving} subcategory of $\mathcal{A}$, see \cite[Proposition 4.1.5]{Kva16}. Dually, $\relGinj{I}{\mathcal{A}}$ is a \emphbf{coresolving} subcategory of $\mathcal{A}$. 

The following lemma shows that $\relGproj{P}{\mathcal{A}}$ has enough injectives and projectives as an exact category if $\mathcal{C}$ has enough injectives and projectives. Furthermore, the injective and projective objects in $\relGproj{P}{\mathcal{A}}$ are precisely the summands of objects of the form $f_!(J)$ and $f_!(Q)$ where $J$ is an injective and $Q$ is a projective object in $\mathcal{C}$, respectively.

\begin{lem}\label{injectives and projectives for Gorenstein P-projectives}
Let $f_!\dashv f^*$ be an adjunction with a Nakayama functor $\nu$. Furthermore, let $A\in \relGproj{P}{\mathcal{A}}$. The following holds:
\begin{enumerate}[(i)]
\item\label{Injectives for Gorenstein P-projectives} If $J\in \inj \mathcal{C}$, then $\Ext^i_{\mathcal{A}}(A,f_!(J))=0$ for all $i>0$;
\item\label{Enough injectives for Gorenstein P-projectives} Assume $\mathcal{C}$ has enough injectives. Then there exists an exact sequence
\[
0\to A\to f_!(J)\to A'\to 0
\]
where $J\in \inj \mathcal{C}$ and $A'\in \relGproj{P}{\mathcal{A}}$;
\item\label{Projectives for Gorenstein P-projectives} If $Q\in \proj \mathcal{C}$, then $\Ext^i_{\mathcal{A}}(f_!(Q),A)=0$ for all $i>0$;
\item\label{Enough projectives for Gorenstein P-projectives} Assume $\mathcal{C}$ has enough projectives. Then there exists an exact sequence
\[
0\to A'\to f_!(Q)\to A\to 0
\]
where $Q\in \proj \mathcal{C}$ and $A'\in \relGproj{P}{\mathcal{A}}$.
\end{enumerate}
\end{lem}

\begin{proof}
We only prove parts \eqref{Injectives for Gorenstein P-projectives} and \eqref{Enough injectives for Gorenstein P-projectives}, parts \eqref{Projectives for Gorenstein P-projectives} and \eqref{Enough projectives for Gorenstein P-projectives} are proved dually. 
\begin{enumerate}[(i)]
\item[(i)] The restrictions of the functors $f^*\nu\colon \mathcal{G}_P\proj(\mathcal{A})\to \mathcal{C}$ and $f_!\colon \mathcal{C}\to \relGproj{P}{\mathcal{A}}$ are exact and adjoint. It follows from the well-known \cite[Lemma 3.2]{LO17} that for  $A\in \relGproj{P}{\mathcal{A}}$ and  $C\in \mathcal{C}$ there are isomorphisms 
\[\Ext^n_{\mathcal{C}}((f^*\nu)(A),C)\cong \Ext^n_{\relGproj{P}{\mathcal{A}}}(A,f_!(C))\]
for all $n\geq 0$ where $\Ext^n_{\relGproj{P}{\mathcal{A}}}(-,-)$ and $\Ext^n_\mathcal{C}(-,-)$ denote the Yoneda $\Ext$-groups in the exact category $\relGproj{P}{\mathcal{A}}$ and the abelian category $\mathcal{C}$, respectively. Since $\relGproj{P}{\mathcal{A}}$ is resolving, it follows that $\relGproj{P}{\mathcal{A}}$ satisfies the dual of condition (C2) in \cite[Section 12]{Kel96}, and hence by the dual of \cite[Theorem 12.1]{Kel96}, the induced functor 
\[D^-(\relGproj{P}{\mathcal{A}})\to D^-(\mathcal{A})\]
between the derived categories of the exact category $\relGproj{P}{\mathcal{A}}$ and the abelian category $\mathcal{A}$  is fully-faithful. Hence, we obtain that
\begin{align*}\Ext^n_{\relGproj{P}{\mathcal{A}}}(A,f_!(C))&\cong \Hom_{D^-(\relGproj{P}{\mathcal{A}})}(A,f_!(C)[n])\\
&\cong \Hom_{D^-(\mathcal{A})}(A, f_!(C)[n])\cong \Ext^n_{\mathcal{A}}(A,f_!(C))
\end{align*}
where the first isomorphism for exact categories is well-known, see e.g. \cite[Proposition A.13]{Pos11}. In particular, combining these isomorphisms, we obtain that 
\[\Ext^i_\mathcal{A}(A,f_!(J))\cong \Ext^i_\mathcal{C}((f^*\nu)(A),J)=0\]
for all $i>0$ since $J$ is injective in $\mathcal{C}$. This proves part (i).
\item[(ii)] Now assume $\mathcal{C}$ has enough injectives. Since $A\in \relGproj{P}{\mathcal{A}}$, by definition there exists an exact sequence 
\[
0\to A\xrightarrow{i} f_!(B)\xrightarrow{p} A''\to 0
\] 
with $A''\in \relGproj{P}{\mathcal{A}}$. Choose a monomorphism $B\xrightarrow{j}J$ in $\mathcal{C}$ with $J\in \inj \mathcal{C}$, and set $A':= \coker(f_!(j)\circ i)$. Then there is a commutative diagram
\[
\begin{tikzcd}
0\arrow{r} &A\arrow{r}{i}\arrow{d}{1_A}&f_!(B)\arrow{d}{f_!(j)}\arrow{r}&A''\arrow{d}\arrow{r}&0\\
0\arrow{r} &A\arrow{r}{f_!(j)\circ i}&f_!(J)\arrow{r}&A'\arrow{r}&0
\end{tikzcd}
\]
with exact rows. It follows by the snake lemma that $A''\to A'$ is a monomorphism with cokernel isomorphic to $\coker f_!(j)\cong f_!(\coker j)$. Since $A''\in \relGproj{P}{\mathcal{A}}$ and $f_!(\Coker j)\in \relGproj{P}{\mathcal{A}}$ and $\relGproj{P}{\mathcal{A}}$ is closed under extensions, it follows that $A'\in \relGproj{P}{\mathcal{A}}$. This proves part (\ref{Enough injectives for Gorenstein P-projectives}). \qedhere
\end{enumerate}
\end{proof}

A dual result holds for relative Gorenstein injectives.
Let $f_!\dashv f^*$ be an adjunction with a Nakayama functor.  This guarantees that the derived functors $L_j\nu$ and $R^j\nu^-$ exist, independently of whether $\mathcal{A}$ or $\mathcal{B}$ have enough projectives or injectives, see \cite[Definition 4.2.1]{Kva16}.

\begin{defn}
Let $f_!\dashv f^*$ be an adjunction with a Nakayama functor $\nu$. Then, the adjunction $f_!\dashv f^*$ is called \emphbf{Iwanaga--Gorenstein} if there exists an integer $n$ such that $L_j\nu=0$ for all $j>n$ and $R^j\nu^-=0$ for all $j>n$.  
\end{defn}

In this case, the categories of relative Gorenstein projectives and injectives admit another description:

\begin{prop}[{\cite[Theorem 4.2.2]{Kva16}}]
Let $f_!\dashv f^*$ be an Iwanaga--Gorenstein adjunction on $\mathcal{A}$ with a Nakayama functor $\nu$. The following holds:
\begin{enumerate}[(i)]
\item $A\in \mathcal{A}$ is relative Gorenstein projective if and only if $L_j\nu(A)=0$ for all $j>0$;
\item $A\in \mathcal{A}$ is relative Gorenstein injective if and only if $R^j\nu^-(A)=0$ for all $j>0$; 
\end{enumerate}
\end{prop}

The dimension of $\mathcal{A}$ with respect to the resolving subcategory $\relGproj{P}{\mathcal{A}}$, denoted $\dim_{\relGproj{P}{\mathcal{A}}} (\mathcal{A})$, is the smallest integer $n$ such that for any object $A\in \mathcal{A}$ there exists an exact sequence $0\to G_n\to \dots\to G_1\to G_0\to A\to 0$ with $G_i\in \relGproj{P}{\mathcal{A}}$ for $0\leq i\leq n$. Dually, the dimension of $\mathcal{A}$ with respect to the coresolving subcategory $\relGinj{I}{\mathcal{A}}$, denoted $\dim_{\relGinj{I}{\mathcal{A}}}(\mathcal{A})$, is the smallest integer $n$ such that for any object $A\in \mathcal{A}$ there exists an exact sequence $0\to A\to G_0'\to\dots\to G_n'\to 0$ with $G_i'\in \relGinj{I}{\mathcal{A}}$. for $0\leq i\leq n$.

\begin{thm}[{\cite[Theorem 4.2.6]{Kva16}}]\label{resolution_dimension}
Let $f_!\dashv f^*$ be an adjunction with Nakayma functor $\nu$. Then, the following are equivalent:
\begin{enumerate}[(1)]
	\item $f_!\dashv f^*$ is Iwanaga--Gorenstein;
	\item $\dim_{\relGproj{P}{\mathcal{A}}}(\mathcal{A})< \infty$;
	\item $\dim_{\relGinj{I}{\mathcal{A}}}(\mathcal{A})< \infty$.
\end{enumerate}
Moreover, if this holds, then the following numbers coincide:
\begin{enumerate}[(1)]
\item $\dim_{\relGproj{P}{\mathcal{A}}}(\mathcal{A})$;
\item $\dim_{\relGinj{I}{\mathcal{A}}}(\mathcal{A})$;
\item The smallest integer $s$ such that $L_i\nu(A)=0$ for all $i>s$ and $A\in \mathcal{A}$;
\item The smallest integer $t$ such that $R^i\nu^-(A)=0$ for all $i>t$ and $A\in \mathcal{A}$.
\end{enumerate}
We say that $f_!\dashv f^*$ is $n$\emphbf{-Iwanaga--Gorenstein} or \emphbf{Iwanaga--Gorenstein of dimension $n$} if this common number is $n$.
\end{thm}

For any object $A\in \mathcal{A}$ choose an exact sequence $Q_1\xrightarrow{g'} Q_0\to A\to 0$ and an exact sequence $0\to A\to J_0\xrightarrow {h'} J_1$ with $Q_1$ and $Q_0$ being relative projective and $J_0$ and $J_1$ being relative injective. Let \[\tau(A):=\ker \nu(g')\text{ and }\tau^-:=\coker \nu^-(h').\] 
Note that the isomorphism classes of $\tau(A)$ and $\tau^-(A)$ are not determined uniquely by $A$, as they depend on the choice of $g'$ and $h'$. In particular, $\tau$ and $\tau^-$ are not functors on $\mathcal{A}$ in general. We use the notation $\tau(A)$ to indicate the resemblance with the definition of the Auslander--Reiten translation for a finite-dimensional algebra.

\begin{prop}\label{kernel and cokernel of unit/counit}
Let $f_!\dashv f^*$ be an adjunction with a Nakayama functor $\nu$.
Then, for all objects $A\in \mathcal{A}$ there is an exact sequence
\[
0\to R^1\nu^-(\tau(A))\to A\xrightarrow{\unit{\nu}{\nu^-}_A} \nu^-\nu(A) \to R^2\nu^-(\tau(A))\to 0.
\] 
\end{prop}

\begin{proof}
By definition of $\tau$ there are exact sequences
\begin{equation*}
0\to \tau(A)\to \nu(Q_1)\xrightarrow{h} \im \nu(g')\to 0
\text{ and }
0\to \im \nu(g')\to \nu(Q_0)\xrightarrow{\nu(g)} \nu(A)\to 0.
\end{equation*}
for $h$ the range restriction of $\nu(g')$ to its image. Since $\nu\circ f_!=f_*$, it follows that $\nu$ sends relative projective objects to relative injective objects. Hence, the objects $\nu(Q_i)$ are relative injective for $i=0,1$. Thus, applying $\nu^-$ and considering the long exact sequence in cohomology gives $\Coker \nu^-(h) = R^1\nu^-(\tau(A))$ and $R^2\nu^-(\tau(A))\cong R^1\nu^-(\im \nu(g'))\cong \Coker \nu^-\nu(g)$ as $R^i\nu^-$ vanishes on relative injectives for $i\geq 1$. Since $Q_0$ is relative projective, $\unit{\nu}{\nu^-}_{Q_0}$ is an isomorphism. Hence, $\nu^-\nu(g) = \unit{\nu}{\nu^-}_A\circ g\circ (\unit{\nu}{\nu^-}_{Q_0})^{-1}$ by naturality of $\unit{\nu}{\nu^-}$. Since $(\unit{\nu}{\nu^-}_{Q_0})^{-1}$ is an isomorphism and $g$ is an epimorphism, it follows that $\Coker \unit{\nu}{\nu^-}_A\cong \Coker \nu^-\nu(g) \cong R^2\nu^-(\tau(A))$. Also, we have a commutative diagram

\[
\begin{tikzcd}
0 \arrow{r} &\im g'\arrow{r}\arrow{d}{k} &Q_0 \arrow{r}{g}\arrow{d}{\unit{\nu}{\nu^-}_{Q_0}} &A \arrow{d}{\unit{\nu}{\nu^-}_A}\arrow{r} &0\\
0 \arrow{r}&\nu^-(\im \nu(g'))\arrow{r}&(\nu^-\nu)(Q_0)\arrow{r}{(\nu^-\nu)(g)}&(\nu^-\nu)(A)
\end{tikzcd}
\]
with exact rows, where $k$ is induced from the commutativity of the right square. Here, exactness of the lower row follows from left exactness of $\nu^-$. The snake lemma implies that $\ker \unit{\nu}{\nu^-}_A \cong \coker k$. If $p\colon Q_1\to \im g'$ denotes the range restriction of $g'$ to its image, then we have $\nu^-(h)\circ \unit{\nu}{\nu^-}_{Q_1}= k\circ p$. Since $p$ is an epimorphism and $\unit{\nu}{\nu^-}_{Q_1}$ is an isomorphism, we get that $\coker k\cong \coker \nu^-(h)$. Since $\coker \nu^-(h)\cong R^1\nu^-(\tau(A))$, $\ker \unit{\nu}{\nu^-}_A \cong \coker k$, and $\Coker \unit{\nu}{\nu^-}\cong R^2\nu^-(\tau(A))$ the claim follows.
\end{proof}

A dual statement holds involving $\nu\nu^-(A)$ as middle term and the left derived functors of the Nakayama functor. As in the classical case, a `hereditary' situation allows to conclude that $\tau$ is a functor on the whole category: 

\begin{rmk}\label{hereditaryimpliestauisafunctor}
If the dimension of $\mathcal{A}$ with respect to the subcategory of relative projectives is equal to $1$, then we obtain an exact sequence $0\to Q_1\to Q_0\to A\to 0$ for every $A\in \mathcal{A}$ and therefore the exact sequence $0\to \tau(A)\to \nu(Q_1)\to \nu(Q_0)\to \nu(A)\to 0$ shows that $\tau$ can be identified with the functor $L_1\nu$. Dually, if the dimension of $\mathcal{A}$ with respect to the subcategory of relative injectives is equal to $1$, then $\tau^-$ can be identified with the functor $R^1\nu^-$. 
\end{rmk}

\begin{cor} \label{Gorensteinprojectiveimagenu}
Let $f_!\dashv f^*$ be an adjunction with Nakayama functor $\nu$. If $f_!\dashv f^*$ is $n$-Iwanaga--Gorenstein with $n\leq 2$, then $\im \nu^-=\relGproj{P}{\mathcal{A}}$ and $\im \nu = \relGinj{I}{\mathcal{A}}$. 
\end{cor}

\begin{proof}
We only prove the first claim. The proof of the second claim is dual. The inclusion $\mathcal{G}_P\proj(\mathcal{A})\subseteq \image\nu^-$ follows from \cite[Proposition 4.1.2]{Kva16}. For the converse, note that by \cite[Lemma 4.2.3]{Kva16}, we have that $A\in \image \nu^-$ if and only if there exists a sequence $0\to A\to Q_0\to Q_1$ with $Q_0$ and $Q_1$ being relative projective. The claim now follows from the fact that  $\mathcal{G}_P\proj(\mathcal{A})$ is a resolving subcategory with $\dim_{\relGproj{P}{\mathcal{A}}}\mathcal{A}\leq 2$ as $f_!\dashv f^*$ is Iwanaga--Gorenstein of dimension $n\leq 2$, see \cite[Proposition 2.3]{Sto14}. 
\end{proof}

\begin{prop}\label{1gorensteinunitepi}
Let $f_!\dashv f^*$ be an adjunction with Nakayama functor $\nu$. Assume that $f_!\dashv f^*$ is Iwanaga--Gorenstein of dimension $n\leq 1$. Then the unit $\unit{\nu}{\nu^-}$ is a pointwise epimorphism while the counit $\counit{\nu}{\nu^-}$ is a pointwise monomorphism. 
\end{prop}

\begin{proof}
By Theorem \ref{resolution_dimension}, $R^2\nu^-(\tau(A))=L_2\nu( \tau^-(A))=0$ for all $A$. Thus, by Proposition \ref{kernel and cokernel of unit/counit}, it follows that $\unit{\nu}{\nu^-}$ is a pointwise epimorphism and $\counit{\nu}{\nu^-}$ is a pointwise monomorphism. It is well-known and straightforward to check that a left approximation which is an epimorphism (resp. a right approximation which is a monomorphism) is minimal.
\end{proof}

\begin{cor}\label{thmCgeneralform}
Let $f_!\dashv f^*$ be an adjunction with Nakayama functor $\nu$. Assume that $f_!\dashv f^*$ is Iwanaga--Gorenstein of dimension $n\leq 1$. Let $g$ be a morphism in $\mathcal{A}$.
\begin{enumerate}[(i)]
\item If $g$ is right almost split in $\mathcal{A}$, then $\nu^-(g)$ is either split or right almost split in $\relGproj{P}{\mathcal{A}}$. 
\item If $g$ is left almost split in $\mathcal{A}$, then $\nu(g)$ is either split or left almost split in $\relGinj{I}{\mathcal{A}}$. 
\end{enumerate}
\end{cor}

\begin{proof}
It follows from Proposition \ref{1gorensteinunitepi} that the unit $\unit{\nu}{\nu^-}$ is a pointwise epimorphism and the counit $\counit{\nu}{\nu^-}$ is a pointwise monomorphism. The claim now follows from Theorem \ref{theoremAmaintext} and Corollary \ref{Gorensteinprojectiveimagenu}. 
\end{proof}

Recall that an additive category $\mathcal{B}$ is called \emphbf{Krull--Schmidt} if every object decomposes as a finite direct sum of objects having local endomorphism ring. We refer the reader to \cite{Kra15} for more details on Krull--Schmidt categories. A morphism $f\colon M\to N$ in a Krull--Schmidt category is called \emphbf{left minimal} if each $g\colon N\to N$ such that $gf=f$ is an automorphism. Dually a morphism $f\colon M\to N$ in a Krull--Schmidt category is called \emphbf{right minimal} if each $g\colon M\to M$ such that $fg=f$ is an automorphism. 

Assume $\mathcal{B}$ is a Krull-Schmidt extension closed subcategory of an abelian category. An exact sequence $0\to L\xrightarrow{f} M\xrightarrow{g} N\to 0$ in $\mathcal{B}$ is called an \emphbf{almost split sequence} if $f$ is left almost split and $g$ is right almost split in $\mathcal{B}$.  A standard argument, see e.g. \cite[Proposition 1.14]{ARS95} shows that in this case $L$ and $N$ are indecomposable and $f$ and $g$ are left and right minimal, respectively. 

\begin{lem}\label{right minimal gives almost split seq}
Assume $\mathcal{B}$ is a Krull-Schmidt extension closed subcategory of an abelian category. Let $0\to L\xrightarrow{f} M\xrightarrow{g} N\to 0$ be an exact sequence in $\mathcal{B}$. If $g$ is minimal right almost split, then the sequence is almost split.
\end{lem}

\begin{proof}
We need to show that $f$ is left almost split. Let $h\colon L\to L'$ be a morphism in $\mathcal{B}$ which is not a split monomorphism. Taking the pushout of $f$ along $h$, we get a commutative diagram
\[
\begin{tikzcd}[ampersand replacement=\&]
0\arrow{r} \&L\arrow{r}{f}\arrow{d}{h} \&M\arrow{r}{g}\arrow{d}{k} \&N\arrow{r}\arrow{d}{1}\& 0\\
0\arrow{r} \&L'\arrow{r}{f'} \&M'\arrow{r}{g'} \&N\arrow{r} \&0
\end{tikzcd}
\]
Assume $g'$ is not a split epimorphism. Then $g'$ factors through $g$ via a map $k'\colon M'\to M$, since $g$ is right almost split. Also, the equality $gk'=g'$ implies that there exists induced map $h'\colon L'\to L$. Also, since $g$ is minimal, the composite $k'\circ k$ is an isomorphism, Therefore, $h'\circ h$ is an isomorphism, which implies that $h$ is a split monomorphism. This is a contradiction, and hence $g'$ must be a split epimorphism. Therefore $f'$ is a split monomorphism, so $h$ factors through $f$, which proves the claim.
\end{proof}

\begin{cor}\label{imagesofseveralalmostsplitsequences}
Let $\mathcal{A}$ be a Krull--Schmidt abelian category and $f_!\dashv f^*$ an adjunction with Nakayama functor $\nu$. Assume $f_!\dashv f^*$ is Iwanaga--Gorenstein of dimension $n\leq 1$. Let $0\to A''\xrightarrow{g'} A\xrightarrow{g} A'\to 0$ be an almost split sequence in $\mathcal{A}$.
\begin{enumerate}[(i)]
\item \label{imagesofseveralalmostsplitsequences:ii} If $A'$ is an indecomposable non-projective object in $\relGinj{I}{\mathcal{A}}$, then the sequence
\[
0\to \nu^-(A'')\xrightarrow{\nu^-(g')}\nu^-(A)\xrightarrow{\nu^-(g)}\nu^-(A')\to 0
\]
is exact, and is a sum of an almost split sequence in $\relGproj{P}{\mathcal{A}}$ and a sequence of the form $0\to A'''\xrightarrow{1}A'''\to 0\to 0$ in $\relGproj{P}{\mathcal{A}}$.
\item \label{imagesofseveralalmostsplitsequences:i} If $A'$ is an indecomposable non-projective object in $\relGinj{I}{\mathcal{A}}$, then the sequence
\[
0\to \nu\nu^-(A'')\xrightarrow{\nu\nu^-(g')}\nu\nu^-(A)\xrightarrow{g\circ \counit{\nu}{\nu^-}_A} A'\to 0
\]
is exact, and is a sum of an almost split sequence in $\relGinj{I}{\mathcal{A}}$ and a sequence of the form $0\to A'''\xrightarrow{1}A'''\to 0\to 0$ in $\relGinj{I}{\mathcal{A}}$.  
\end{enumerate}
Dual results hold for the relative Nakayama functor $\nu$. 
\end{cor}

\begin{proof}
Since $f_!\dashv f^*$ is Iwanaga--Gorenstein of dimension $n\leq 1$ it follows from Proposition \ref{kernel and cokernel of unit/counit} that the unit $\unit{\nu}{\nu^-}$ is an epimorphism while the counit $\counit{\nu}{\nu^-}$ is a monomorphism. Therefore Theorem \ref{theoremAmaintext} implies that $\nu^-(g)$ is right almost split or split in $\relGproj{P}{\mathcal{A}}$. Note that since $\nu^-$ and $\nu$ are exact equivalences when restricted to $\relGinj{I}{\mathcal{A}}$ and $\relGproj{P}{\mathcal{A}}$ it follows that if $A'$ is non-projective in $\relGinj{I}{\mathcal{A}}$ then $\nu^-(A')$ is non-projective in $\relGproj{P}{\mathcal{A}}$. Therefore, $\nu^-(g)$ is an epimorphism since $\nu^-(A')$ not being projective in $\relGproj{P}{\mathcal{A}}$ implies that there exists a non-split deflation ending in it which must factor through $\nu^-(g')$. It remains to show that $\nu^-(g')$ is not a split epimorphism since then the claim follows from Lemma \ref{right minimal gives almost split seq} and the fact that in a Krull--Schmidt category there are minimal versions of morphisms, cf. \cite[Corollary 1.4]{KS98}. To this end note that the composite $g\circ \counit{\nu}{\nu^-}_A=\counit{\nu}{\nu^-}_{A'} \circ \nu\nu^-(g)$ is not a split epimorphism since $g$ is not a split epimorphism.  Also since $A'\in \relGinj{I}{\mathcal{A}}$ and $\counit{\nu}{\nu^-}$ is an isomorphism when restricted to $\relGinj{I}{\mathcal{A}}$, it follows that $\nu\nu^-(g)$ is not a split epimorphism, so $\nu^-(g)$ is not a split epimorphism and therefore the sequence is not split exact. This finishes the proof of \eqref{imagesofseveralalmostsplitsequences:ii}.

To prove \eqref{imagesofseveralalmostsplitsequences:i} note that $\nu^-$ and $\nu$ are exact equivalences, when restricted to $\relGinj{I}{\mathcal{A}}$ and $\relGproj{P}{\mathcal{A}}$. It follows that they send almost split sequences to almost split sequences. Thus, \eqref{imagesofseveralalmostsplitsequences:i} immediately follows from \eqref{imagesofseveralalmostsplitsequences:ii}. 
\end{proof}

Recall that a Krull--Schmidt extension closed subcategory of an abelian category is said to \emphbf{have almost split sequences} if for every indecomposable non-projective object there exists an almost split sequence ending in it and for every indecomposable non-injective object there exists an almost split sequence starting in it. 

\begin{cor}\label{Existence of almost split sequences}
Let $\mathcal{A}$ be a Krull--Schmidt abelian category with almost split sequences and $f_!\dashv f^*$ an adjunction with Nakayama functor $\nu$. Assume $f_!\dashv f^*$ is Iwanaga--Gorenstein of dimension $n\leq 1$. Then $\relGproj{P}{\mathcal{A}}$ and $\relGinj{I}{\mathcal{A}}$ have almost split sequences.
\end{cor} 

\begin{proof}
Since $\nu^-\colon \relGinj{I}{\mathcal{A}}\to \relGproj{P}{\mathcal{A}}$ is an exact equivalence,  any indecomposable object $A\in \relGproj{P}{\mathcal{A}}$ which is not projective in $\relGproj{P}{\mathcal{A}}$ is isomorphic to an indecomposable object $\nu^-(A')$ where $A'\in \relGinj{I}{\mathcal{A}}$ is not projective in $\relGinj{I}{\mathcal{A}}$. Since $A'$ is not projective, there exists an almost split sequence ending in $A'$ in $\mathcal{A}$. Applying Corollary \ref{imagesofseveralalmostsplitsequences} \eqref{imagesofseveralalmostsplitsequences:ii}, we get an almost split sequence in $\relGproj{P}{\mathcal{A}}$ ending in $A\cong \nu^-(A')$. This shows that any indecomposable non-projective object in $\relGproj{P}{\mathcal{A}}$ is the rightmost term of an almost split sequence. The fact that any  indecomposable non-injective object in $\relGproj{P}{\mathcal{A}}$ is the leftmost term of an almost split sequence follows immediately from the dual of Corollary \ref{imagesofseveralalmostsplitsequences} \eqref{imagesofseveralalmostsplitsequences:i}. This shows that $\relGproj{P}{\mathcal{A}}$ has almost split sequences. The fact that $\relGinj{I}{\mathcal{A}}$ has almost split sequences is proved dually.
\end{proof}

The goal of the remainder of this section is to show that the adjunction with Nakayama functor coming from a phylum is $1$-Iwanaga--Gorenstein and to give explicit descriptions of the functors $\nu$ and $\nu^-$ and the subcategories $\relGproj{P}{\mathcal{A}}$ and $\relGinj{I}{\mathcal{A}}$ in this case. 

Let $\mathcal{C}$, $X$ and $Y$ be as in Assumption \ref{standard1}. Note that the canonical inclusion
\[\bigoplus_{i\geq 1} X^i\to \bigoplus_{i\geq 0} X^i\]
induces a natural transformation $\iota\colon f_!\circ X\to f_!$.  

In our set-up of the Eilenberg--Moore adjunction associated to a free monad, we obtain a `relative global dimension $1$'-situation as the following lemma shows. It generalises the standard projective resolution for hereditary algebras.

\begin{lem}\label{globaldimension1}
Let $\mathcal{C}$, $X$, and $Y$ be as in Assumption \ref{standard1}. 
Then, for every $(M,h)\in \mathcal{C}^{T(X)}$ the sequence 
\[\begin{tikzcd}0\arrow{r} &f_!(X (f^*(M,h)))\arrow{rrr}{\iota_{f^*(M,h)}-f_!(h_1)} &&&(f_!f^*)(M,h)\arrow{r}{\counit{f_!}{f^*}_{(M,h)}} &(M,h)\arrow{r}&0\end{tikzcd}\]
in $\mathcal{C}^{T(X)}$ is exact. In particular, the adjunction $f_!\dashv f^*$ is Iwanaga--Gorenstein of dimension $n\leq 1$. 
\end{lem}

\begin{proof}
By Proposition \ref{eilenbergmooreexact}, it is enough to show that applying $f^*$ to the claimed sequence gives an exact sequence. First note that this yields the sequence 
\[\begin{tikzcd}[column sep=10ex]0\arrow{r} &\bigoplus_{i\geq 1} X^i(M)\arrow{r}{\incl-f^*f_!(h_1)} &\bigoplus_{i\geq 0} X^i(M)\arrow{r}{h}& M\arrow{r}&0.\end{tikzcd}\]
as for $(M,h)\in \mathcal{C}^{T(X)}$ we have $f^*(\counit{f_!}{f^*}_{(M,h)})=h$. Now consider the composite 
\[s:=f^*f_!(h_1)\circ p\colon \bigoplus_{i\geq 0} X^i(M)\to \bigoplus_{i\geq 0}X^i(M)\]
where $p\colon \bigoplus_{i\geq 0} X^i(M)\to \bigoplus_{i\geq 1} X^i(M)$ is the canonical projection. Note that the infinite sum $1+s+s^2+\dots$ makes sense as a morphism since $s^j$ restricted to $X^i(M)$ vanishes for $j\geq i+1$ and therefore the sum is finite when restricting to each component $X^i(M)$. Furthermore  it is an isomorphism with inverse given by $1-s$. Since 
\[\left(\sum_{i=0}^\infty s^i\right)\circ (\incl-f^*f_!(h_1))=\incl\colon \bigoplus_{i\geq 1} X^i(M)\to \bigoplus_{i\geq 0} X^i(M)\]
and 
\[\proj\circ \left(\sum_{i=0}^\infty s^i\right) = h\colon \bigoplus_{i\geq 0}X^i(M)\to M\]
where $\proj$ denotes the canonical projection, it follows that 
\[
\begin{tikzcd}[ampersand replacement=\&, column sep=10ex]
0\arrow{r} \&\bigoplus_{i\geq 1} X^i(M)\arrow{r}{\incl-f^*f_!(h_1)}\arrow[equals]{d} \&\bigoplus_{i\geq 1} X^i(M)\oplus M\arrow{r}{h}\arrow{d}{\sum_{i=0}^\infty s^i} \& M\arrow{r}\arrow[equals]{d} \& 0\\
0\arrow{r} \&\bigoplus_{i\geq 1} X^i(M) \arrow{r}{\begin{pmatrix}1\\0\end{pmatrix}} \&\bigoplus_{i\geq 1} X^i(M)\oplus M\arrow{r}{(0,1)} \& M\arrow{r} \& 0
\end{tikzcd}
\]
commutes. Since the lower sequence is obviously (split) exact, it follows that also the upper sequence is exact. This proves the first claim. To see that $f_!\dashv f^*$ is Iwanaga--Gorenstein of dimension $n\leq 1$ it suffices to use Theorem \ref{resolution_dimension} and note that $\dim_{\relGproj{P}{\mathcal{C}^{T(X)}}} \mathcal{C}^{T(X)}\leq 1$ by the existence of the exact sequence in the lemma and the fact that the image of $f_!$ is contained in $\relGproj{P}{\mathcal{C}^{T(X)}}$. 
\end{proof}

The following lemma generalises that a map between projectives over a finite dimensional algebra is determined by the image of its top.

\begin{lem}\label{morphismsbetweenPprojectives}
Let $\mathcal{C}$, $X$, and $Y$ be as in Assumption \ref{standard1}. 
Let 
\[(\varphi_{i,j})_{i,j}\colon \bigoplus_{i\geq 0} X^i(M)\to \bigoplus_{j\geq 0} X^j(N)\]
be a morphism in $\mathcal{C}$ with components $\varphi_{i,j}\colon X^i(M)\to X^j(N)$. Then there exists a morphism $\varphi\colon f_!(M)\to f_!(N)$ satisfying $f^*(\varphi)=(\varphi_{i,j})_{i,j}$ if and only if $\varphi_{i,j}=0$ for $i>j$ and $\varphi_{i,j}=X^i(\varphi_{0,j-i})$ for $i\leq j$. In particular, $\varphi$ is an isomorphism if and only if $\Kopf_X\varphi=\varphi_{0,0}$ is an isomorphism. 
\end{lem}

\begin{proof}
Using Lemma \ref{comma=EilenbergMoore} it follows that $(\varphi_{i,j})_{i,j}$ induces a morphism $\varphi\colon f_!(M)\to f_!(N)$ in $\mathcal{C}^{T(X)}$ if and only if the following diagram commutes 
\[
\begin{tikzcd}
X(\bigoplus_{i\geq 0} X^i(M))\arrow{r}\arrow{d}{X((\varphi_{i,j})_{i,j})} &\bigoplus_{i\geq 0} X^i(M)\arrow{d}{(\varphi_{i,j})_{i,j}}\\
X(\bigoplus_{j\geq 0} X^j(N))\arrow{r} &\bigoplus_{j\geq 0} X^j(N)
\end{tikzcd}
\]
where the horizontal maps are the canonical inclusions. Restricting to the $i$-th component of $X(\bigoplus_{i\geq 0} X^i(M))$ and the $j$-th component of $\bigoplus_{j\geq 0} X^j(N)$ we get that the diagram above commutes if and only if the following diagram commutes for all $i,j\geq 0$
\[\begin{tikzcd}
X^{i+1}(M)\arrow{r}{1}\arrow{d}{X(\varphi_{i,j-1})} &X^{i+1}(M)\arrow{d}{\varphi_{i+1,j}}\\
X^j(N)\arrow{r}{1} &X^j(N)
\end{tikzcd}
\]
where $\varphi_{i,-1}=0$ for all $i$. The commutativity of these diagrams is equivalent to $\varphi_{i,j}=X(\varphi_{i-1,j-1})=\dots=X^{j+1}(\varphi_{i-j-1,-1})=0$ for all $i>j$ and $\varphi_{i,j}=X(\varphi_{i-1,j-1})=\dots=X^i(\varphi_{0,j-i})$ for $i\leq j$. 
The final claim follows from the fact that $\varphi$ is an isomorphism if and only if $f^*(\varphi)=(\varphi_{i,j})_{i,j}$ is an isomorphism, and the fact that $(\varphi_{i,j})_{i,j}$ can be represented as an upper triangular (countably infinite) matrix with diagonal $(\varphi_{0,0},X(\varphi_{0,0}), X^2(\varphi_{0,0}),\dots)$. 
\end{proof}

Dually, there is a generalisation of the fact that a map between injectives is determined by the preimage of the socle. 
Using this, we will give an explicit description of the Nakayama functor in our setting:

\begin{prop}\label{definitionofnu}
Let $\mathcal{C}$, $X$, and $Y$ be as in Assumption \ref{standard1}. 
Then, the functor $\nu$ defined by the sequence 
\begin{equation}\label{definitionofnu:eq}
\begin{tikzcd}
f_*(X(f^*(M,h)))\arrow{r}{\chi} &(f_*f^*)(M,h)\arrow{r} &\nu(M,h)\arrow{r}&0
\end{tikzcd}
\end{equation}
is a Nakayama functor relative to $f_!\dashv f^*$, where the map $\chi$ is defined componentwise by $\chi_{i,j}\colon (Y^iX)(M)\to Y^j(M)$ given by 
\[
\chi_{i,j}=\begin{cases}-Y^i(h_1)&\text{for }j=i,\\Y^{i-1}(\counit{Y}{X}_M)&\text{for }j=i-1,\\0&\text{else.}\end{cases}
\]
\end{prop}

\begin{proof}
According to \cite[Proposition 3.3.1]{Kva16}, it suffices to define $\nu$ as an isomorphism  $\image f_!\to \image f_*$ where $\image f_!:=\{f_!(M)|M\in \mathcal{C}\}$ and $\image f_*:=\{f_*(M)|M\in \mathcal{C}\}$ denote the strict images of $f_!$ and $f_*$, respectively. 
Recall that by Lemma \ref{morphismsbetweenPprojectives} a morphism $\varphi\colon f_!(M)\to f_!(N)$ is uniquely defined by a sequence $(\varphi_{0,0},\varphi_{0,1},\dots)$ of morphisms $\varphi_{0,j}\colon M\to X^j(N)$ in $\mathcal{C}$. Using the dual of Lemma \ref{morphismsbetweenPprojectives} we let $\tilde{\varphi}\colon f_*(M)\to f_*(N)$ be the unique morphism defined by the sequence $(\tilde{\varphi}_{0,0}, \tilde{\varphi}_{1,0},\dots)$ where $\tilde{\varphi}_{i,0}=(\adj^{Y^i\dashv X^i})^{-1}(\varphi_{0,i})\colon Y^i(M)\to N$. We claim that that sending $f_!(M)$ to $f_*(M)$ and $\varphi$ to $\tilde{\varphi}$ gives an isomorphism $E\colon \image f_!\to \image f_*$ satisfying $E\circ f_!=f_*$. We have to prove that $E(\id)=\id$ and that $E(\varphi\circ \psi)=E(\varphi)\circ E(\psi)$. For the former, note that the identity on $f_!(M)$ is represented by the morphism with $\varphi_{0,0}=\id_M$ and $\varphi_{0,i}=0$ for all $i\neq 0$. Therefore $E(\id)$ is represented by the morphism $\tilde{\varphi}$ with $\tilde{\varphi}_{0,0}=\id$ and $\tilde{\varphi}_{i,0}=0$ for $i\neq 0$ by linearity of the adjunction isomorphism. For the latter note that $(\varphi\circ \psi)_{0,i}=\sum_j X^j(\varphi_{0,i-j})\circ \psi_{0,j}$. Therefore 
\begin{align*}
(E(\varphi\circ \psi))_{i,0}&=\sum_j (\adj^{Y^i\dashv X^i})^{-1}(X^j(\varphi_{0,i-j})\circ \psi_{0,j})\\
&=\sum_j(\adj^{Y^{i-j}\dashv X^{i-j}})^{-1}(\adj^{Y^j\dashv X^j})^{-1}(X^j(\varphi_{0,i-j})\circ \psi_{0,j})\\
&=\sum_j(\adj^{Y^{i-j}\dashv X^{i-j}})^{-1}(\varphi_{0,i-j}\circ (\adj^{Y^j\dashv X^j})^{-1}(\psi_{0,j}))
\end{align*}
where the last equality follows from naturality of the adjunction isomorphism in the second variable.
On the other hand, 
\begin{align*}(E(\varphi)\circ E(\psi))_{i,0}&=\sum_k \tilde{\varphi}_{k,0}\circ Y^k(\tilde{\psi}_{i-k,0})\\
&=\sum_k(\adj^{Y^k\dashv X^k})^{-1}(\varphi_{0,k})\circ Y^k((\adj^{Y^{i-k}\dashv X^{i-k}})^{-1}(\psi_{0,i-k}))\\
&=\sum_k(\adj^{Y^k\dashv X^k})^{-1}(\varphi_{0,k}\circ (\adj^{Y^{i-k}\dashv X^{i-k}})^{-1}(\psi_{0,i-k}))
\end{align*}
where the last equality follows from naturality of the adjunction isomorphism in the first variable. The claim follows using the substitution $k=i-j$. 

Hence, by \cite[Proposition 3.3.1]{Kva16} the object $\nu(M,h)$ can be calculated applying $E$ to the leftmost map in the exact sequence in Lemma \ref{globaldimension1} and taking the cokernel. This yields the exact sequence 
$f_*(X(M))\stackrel{\chi}{\to} f_*(M)\to \nu(M,h)\to0$.
The claim follows.
\end{proof}

\begin{mex}\label{runningexample3}
Let $\mathfrak{A}$ be a phylum on a locally bounded quiver $Q$. Let $X,Y$ be as in Example \ref{runningexample1} and let $Q_{\geq 0}$ denote the set of paths in $Q$: For a path $p=\alpha_n\alpha_{n-1}\cdots \alpha_1$ in $Q$, we let $t(p)=t(\alpha_n)$ and $s(p)=s(\alpha_1)$ denote the target and the source of the path, and we let 
\begin{align*}
F_p&:=F_{\alpha_n}\circ F_{\alpha_{n-1}}\circ \dots\circ F_{\alpha_1}\colon \mathcal{A}_{s(p)}\to \mathcal{A}_{t(p)}\\
G_p&:=G_{\alpha_1}\circ \dots\circ G_{\alpha_{n-1}}\circ G_{\alpha_n}\colon \mathcal{A}_{t(p)}\to \mathcal{A}_{s(p)}
\end{align*}
denote the induced functors. Then the functors $f_!\colon \mathcal{C}\to \mathcal{C}^{T(X)}$ and $f_*\colon \mathcal{C}\to \mathcal{C}^{W(Y)}$ applied to $M=(M_\mathtt{i})_{\mathtt{i}\in Q_0}$ are given by
\begin{align*}
f_!(M)&=\left(\bigoplus_{\substack{p\in Q_{\geq 0}\\t(p)=\mathtt{i}}}F_p(M_{s(p)}), f_!(M)_\alpha\right)_{\substack{\mathtt{i}\in Q_0\\\alpha\in Q_1}} \text{and }
f_*(M)&=\left(\bigoplus_{\substack{p\in Q_{\geq 0}\\s(p)=\mathtt{i}}}G_p(M_{t(p)}), f_*(M)_\alpha\right)_{\substack{\mathtt{i}\in Q_0\\\alpha\in Q_1}}
\end{align*}
where 
\begin{align*}
f_!(M)_\alpha&\colon \bigoplus_{\substack{p\in Q_{\geq 0}\\t(p)=s(\alpha)}} F_\alpha F_p(M_{s(p)})\to \bigoplus_{\substack{q\in Q_{\geq 0}
\\t(q)=t(\alpha)}}F_q(M_{s(q)})\\
f_*(M)_\alpha&\colon \bigoplus_{\substack{q\in Q_{\geq 0}\\s(q)=s(\alpha)}} G_q(M_{t(q)})\to \bigoplus_{\substack{p\in Q_{\geq 0}\\s(p)=t(\alpha)}} G_\alpha G_p(M_{t(p)})
\end{align*}
are the canonical inclusion and projection, respectively. For an object $(M_\mathtt{i}, M_\alpha)\in \rep \mathfrak{A}$, the exact  sequence in Lemma \ref{globaldimension1} at vertex $\mathtt{i}\in Q_0$ then reads as 
\[0\to \bigoplus_{\substack{p\in Q_{\geq 0}\\t(p)=\mathtt{i}}}\bigoplus_{\substack{\alpha\in Q_1\\t(\alpha)=s(p)}} F_pF_\alpha(M_{s(\alpha)})\to \bigoplus_{\substack{q\in Q_{\geq 0}\\t(q)=\mathtt{i}}} F_q(M_\mathtt{i})\to M_\mathtt{i}\to 0\]
where the left-most map is the difference between the maps induced by $F_pF_\alpha(M_{s(\alpha)})\xrightarrow{\id} F_{p\alpha} (M_{s(p\alpha)})$ and $F_pF_\alpha(M_{s(\alpha)})\xrightarrow{F_p(M_\alpha)} F_p(M_{s(p)})$. 
 Applying $\nu$ to the exact sequences in Lemma \ref{globaldimension1} gives the exact sequence
\[\bigoplus_{\substack{p\in Q_{\geq 0}\\s(p)=\mathtt{i}}}\bigoplus_{\substack{\alpha\in Q_1\\t(\alpha)=t(p)}} G_pF_\alpha(M_{s(\alpha)})\to \bigoplus_{\substack{q\in Q_{\geq 0}
\\s(q)=\mathtt{i}}} G_q(M_{t(q)})\to \nu(M_\mathtt{j},M_\alpha)_{\mathtt{i}}\to 0\]
at vertex $\mathtt{i}\in Q_0$, where the leftmost map is the difference between the maps induced by $G_{\alpha p}F_\alpha(M_{s(\alpha)})\xrightarrow{G_p(\counit{G_\alpha}{F_\alpha}_{M_{s(\alpha)}})} G_p(M_{s(\alpha)})$ and $G_pF_\alpha(M_{s(\alpha)})\xrightarrow{G_p(M_\alpha)} G_p(M_{s(\alpha)})$ using that $s(\alpha)=t(p)$. This gives an explicit description of $\nu$ on objects in the case of representations of phyla. 
\end{mex}

In the next example, we spell this out in more detail for the quivers considered before. 

\begin{ex}
For simplicity of notation we omit the subscripts of the natural transformations in the following example. In the notation of Example \ref{runningexample}, we obtain:
\begin{enumerate}[(i)]
\item Let $(M_\mathtt{1},M_\mathtt{2},M_\alpha)\in \rep(\Lambda_\mathtt{1}\xrightarrow{\Lambda_\alpha} \Lambda_\mathtt{2})$. The sequence in Lemma \ref{globaldimension1} reads as:
\[
\begin{tikzcd}[ampersand replacement=\&]
0\arrow{r} \&0\arrow{r}\arrow{d} \&M_\mathtt{1}\arrow{r}{\id}\arrow{d}{\begin{pmatrix}\id\\0\end{pmatrix}} \&M_\mathtt{1}\arrow{r}\arrow{d}{M_\alpha}\&0\\
0\arrow{r} \&\Lambda_\alpha\otimes M_\mathtt{1}\arrow{r}{\begin{pmatrix}\id\\-M_\alpha\end{pmatrix}} \&(\Lambda_\alpha\otimes M_\mathtt{1})\oplus M_\mathtt{2}\arrow{r}{(M_\alpha,1)} \&M_\mathtt{2}\arrow{r} \&0
\end{tikzcd}
\]
Here, the vertical maps have to be read as morphisms from ($\Lambda_\alpha\otimes$ the source) to the target. Applying $\nu$, one obtains an explicit description of the sequence \eqref{definitionofnu:eq} in this example
\[
\begin{tikzcd}[ampersand replacement=\&]
\Lambda_{\alpha^*}\otimes \Lambda_\alpha\otimes M_\mathtt{1}\arrow{d}{\varepsilon_{\alpha}}\arrow{r}{\begin{pmatrix}\varepsilon_{\alpha^*}\\-\Lambda_{\alpha^*}\otimes M_\alpha\end{pmatrix}} \&M_\mathtt{1}\oplus (\Lambda_{\alpha^*}\otimes M_\mathtt{2})\arrow{d}{(0,\varepsilon_{\alpha})}\arrow{r} \& \coker\begin{pmatrix}\varepsilon_{\alpha^*}\\-\Lambda_{\alpha^*}\otimes M_\alpha\end{pmatrix}\arrow{d}\arrow{r} \&0\\
\Lambda_\alpha\otimes M_\mathtt{1}\arrow{r}{-M_\alpha} \& M_\mathtt{2}\arrow{r} \&\coker(M_\alpha)\arrow{r} \&0
\end{tikzcd}
\]
where $\varepsilon_{\alpha^*}$ is the counit of the adjunction $\Lambda_{\alpha^*}\otimes_{\Lambda_\mathtt{2}} -\dashv \Lambda_\alpha\otimes_{\Lambda_\mathtt{1}}-$, i.e. the evaluation map $\Lambda_{\alpha^*}\otimes_{\Lambda_\mathtt{2}}\Lambda_\alpha\to \Lambda_{\mathtt{1}}$ while dually $\varepsilon_{\alpha}$ is the counit of the adjunction $\Lambda_\alpha\otimes_{\Lambda_\mathtt{1}} -\dashv \Lambda_{\alpha^*}\otimes_{\Lambda_\mathtt{2}} -$. 
The presence of the bimodule $\Lambda_\alpha$ complicates the situation. In the case that $\Lambda_\mathtt{1}=\Lambda_\mathtt{2}=\Lambda_\alpha$, one obtains that $\varepsilon_{\alpha^*}$ is the canonical isomorphism and thus, $\coker\begin{pmatrix}\varepsilon_{\alpha^*}\\-\Lambda_{\alpha^*}\otimes M_\alpha\end{pmatrix}\cong M_\mathtt{2}$. Hence, we obtain the ordinary cokernel functor. 
\item Let $(M_\mathtt{1},M_\mathtt{2}, M_\mathtt{3},M_\alpha,M_\beta)\in \rep(\Lambda_\mathtt{1}\xrightarrow{\Lambda_\alpha}\Lambda_\mathtt{2}\xrightarrow{\Lambda_\beta} \Lambda_\mathtt{3})$. Similarly, to Meta-Example \ref{runningexample3} we define $\Lambda_{\beta\alpha}=\Lambda_\beta\otimes_{\Lambda_{\mathtt{2}}} \Lambda_\alpha$, $\Lambda_{\alpha^*\beta^*}=\Lambda_{\alpha^*}\otimes_{\Lambda_{\mathtt{2}}} \Lambda_{\beta^*}$, and $M_{\beta\alpha}\colon \Lambda_{\beta\alpha}\otimes_{\Lambda_{\mathtt{1}}} M_{\mathtt{1}}\to M_{\mathtt{3}}$ as $M_\beta\circ (\Lambda_\beta\otimes M_\alpha)$.  With this notation we obtain for the sequence in Lemma \ref{globaldimension1}:
\[
\begin{tikzcd}[ampersand replacement=\&, row sep=7ex]
0\arrow{r} \&0\arrow{d}\arrow{r}\& M_\mathtt{1}\arrow{d}{\begin{pmatrix}1\\0\end{pmatrix}}\arrow{r}{1}\& M_\mathtt{1}\arrow{d}{M_\alpha}\arrow{r}\&0\\
0\arrow{r} \&\Lambda_\alpha\otimes M_\mathtt{1}\arrow{d}{\begin{pmatrix}1\\0\end{pmatrix}}\arrow{r}{\begin{pmatrix}1\\-M_\alpha\end{pmatrix}}\& (\Lambda_\alpha\otimes M_\mathtt{1})\oplus M_\mathtt{2}\arrow{r}{(M_\alpha,1)}\arrow{d}{\begin{pmatrix}1&0\\0&1\\0&0\end{pmatrix}}\& M_\mathtt{2}\arrow{d}{M_\beta}\arrow{r}\&[-2ex]0\\
0\arrow{r} \&(\Lambda_{\beta\alpha}\otimes M_\mathtt{1})\oplus (\Lambda_\beta\otimes M_\mathtt{2})\arrow{r}[yshift=0.7ex]{\begin{pmatrix}1&0\\-\Lambda_\beta\otimes M_\alpha &1\\0&-M_\beta\end{pmatrix}}\&(\Lambda_{\beta\alpha}\otimes M_1)\oplus (\Lambda_\beta\otimes M_2)\oplus M_\mathtt{3}\arrow{r}{(M_{\beta\alpha},M_\beta,1)}\& M_\mathtt{3}\arrow{r}\&0
\end{tikzcd}
\]
Applying $\nu$, we get an explicit description of the sequence \eqref{definitionofnu:eq}:
\[
\begin{tikzcd}[ampersand replacement=\&, column sep=15ex]
\begin{matrix}(\Lambda_{\alpha^*}\otimes \Lambda_\alpha\otimes M_\mathtt{1})\\
\oplus (\Lambda_{\alpha^*\beta^*}\otimes \Lambda_\beta\otimes M_\mathtt{2})\end{matrix}\arrow{d}{\begin{pmatrix}\varepsilon_\alpha&0\\0&\varepsilon_\alpha\end{pmatrix}} \arrow{r}{\begin{pmatrix}\varepsilon_\alpha&0\\-\Lambda_{\alpha^*}\otimes M_\alpha&\Lambda_{\alpha^*}\otimes \varepsilon_\beta\\0&-\Lambda_{\alpha^*\beta^*}\otimes M_\beta\end{pmatrix}}\&[4ex]\begin{matrix}M_\mathtt{1}\oplus (\Lambda_{\alpha^*}\otimes M_\mathtt{2})\\
\oplus (\Lambda_{\alpha^*\beta^*}\otimes M_\mathtt{3})\end{matrix}\arrow{d}{\begin{pmatrix}0&\varepsilon_\alpha&0\\0&0&\varepsilon_\alpha\end{pmatrix}}\arrow{r}\&[-5ex]\arrow{r}\nu(M)_\mathtt{1}\arrow{d} \&[-8ex]0\\
\begin{matrix}(\Lambda_\alpha\otimes M_\mathtt{1})\\\oplus (\Lambda_{\beta^*}\otimes \Lambda_\beta\otimes M_\mathtt{2})\end{matrix}\arrow{r}{\begin{pmatrix}-M_\alpha&\varepsilon_\beta\\0&-\Lambda_{\beta^*}\otimes M_\beta\end{pmatrix}}\arrow{d}{(0,\varepsilon_\beta)}\&M_\mathtt{2}\oplus (\Lambda_{\beta^*}\otimes M_\mathtt{3})\arrow{d}{(0,\varepsilon_\beta)}\arrow{r}\&\nu(M)_\mathtt{2}\arrow{d}\arrow{r} \&0\\
\Lambda_\beta\otimes M_\mathtt{2}\arrow{r}{-M_\beta}\&M_\mathtt{3}\arrow{r}\&\nu(M)_\mathtt{3} \arrow{r} \&0
\end{tikzcd}
\]
\item In the final example, let $(M_\mathtt{1},M_\mathtt{2}, M_\mathtt{3},M_\alpha,M_\beta)\in \rep(\Lambda_\mathtt{1}\xrightarrow{\Lambda_\alpha}\Lambda_\mathtt{3}\xleftarrow{\Lambda_\beta} \Lambda_\mathtt{3})$. The relative projective resolution in Lemma \ref{globaldimension1} reads as:
\[
\begin{tikzcd}[ampersand replacement=\&, row sep=9ex]
0\arrow{r} \&0\arrow{d}\arrow{r}\&M_\mathtt{1}\arrow{d}{\begin{pmatrix}0\\1\\0\end{pmatrix}}\arrow{r}{1}\&M_\mathtt{1}\arrow{d}{M_\alpha}\\
0\arrow{r} \&(\Lambda_\alpha\otimes M_\mathtt{1})\oplus (\Lambda_\beta\otimes M_\mathtt{2})\arrow{r}{\begin{pmatrix}-M_\alpha&-M_\beta\\1&0\\0&1\end{pmatrix}}\&M_\mathtt{3}\oplus (\Lambda_\alpha\otimes M_\mathtt{1})\oplus (\Lambda_\beta\otimes M_\mathtt{2})\arrow{r}{(1,M_\alpha,M_\beta)}\&M_\mathtt{3}\\
0\arrow{r} \&0\arrow{u}\arrow{r}\&M_\mathtt{2}\arrow{u}[swap]{\begin{pmatrix}0\\0\\1\end{pmatrix}}\arrow{r}{1}\&M_\mathtt{2}\arrow{u}{M_\beta}
\end{tikzcd}
\]
Applying $\nu$ yields an explicit description of the exact sequence \eqref{definitionofnu:eq}
\[
\begin{tikzcd}[ampersand replacement=\&, column sep=15ex]
\begin{matrix}(\Lambda_{\alpha^*}\otimes \Lambda_\alpha\otimes M_\mathtt{1})\\
\oplus (\Lambda_{\alpha^*}\otimes \Lambda_\beta\otimes M_\mathtt{2})\end{matrix}\arrow{d}{\begin{pmatrix}\varepsilon_\alpha&0\\0&\varepsilon_\alpha\end{pmatrix}}\arrow{r}{\begin{pmatrix}\varepsilon_{\alpha^*}&0\\-\Lambda_\alpha^*\otimes M_\alpha&-\Lambda_\alpha^*\otimes M_\beta\end{pmatrix}}\&M_\mathtt{1}\oplus (\Lambda_{\alpha^*}\otimes M_\mathtt{3})\arrow{d}{(0,\varepsilon_{\alpha})}\arrow{r}\&[-7ex]\nu(M)_\mathtt{1}\arrow{d}\arrow{r}\&[-7ex]0\\
(\Lambda_\alpha\otimes M_\mathtt{1})\oplus (\Lambda_\beta\otimes M_\mathtt{2})\arrow{r}{(-M_\alpha,-M_\beta)}\& M_\mathtt{3}\arrow{r}\&\nu(M)_\mathtt{3} \arrow{r}\& 0\\
\begin{matrix}(\Lambda_{\beta^*}\otimes \Lambda_\alpha\otimes M_\mathtt{1})\\
\oplus (\Lambda_{\beta^*}\otimes \Lambda_\beta\otimes M_\mathtt{2})\end{matrix}\arrow{r}{\begin{pmatrix}0&\varepsilon_{\beta^*}\\-\Lambda_{\beta^*}\otimes M_\alpha& -\Lambda_{\beta^*}\otimes M_\beta\end{pmatrix}}\arrow{u}{\begin{pmatrix}\varepsilon_\beta&0\\0&\varepsilon_\beta\end{pmatrix}}\&M_\mathtt{2}\oplus (\Lambda_{\beta^*}\otimes M_\mathtt{3})\arrow{u}{(0,\varepsilon_{\beta})}\arrow{r}\&\nu(M)_\mathtt{2}\arrow{r}\arrow{u} \&0
\end{tikzcd}
\]
\end{enumerate}
\end{ex}

\begin{lem}\label{tautau-adjoint}
Let $\mathcal{C}$, $X$ and $Y$ be as in Assumption \ref{standard1}. Then $\tau^-$ is left adjoint to $\tau$.
\end{lem}

\begin{proof}
Let $A$ and $A'$ be two arbitrary objects in $\mathcal{C}^{T(X)}$. By Lemma \ref{globaldimension1} and its dual we know that there exist exact sequences
$0\to A\xrightarrow{j} J_1\xrightarrow{t}J_{0}\to 0$ and 
$0\to Q_1\xrightarrow{s}Q_0\xrightarrow{p}A'\to 0$
where $Q_0$ and $Q_1$ are relative projective and $p$ is a right approximation with respect to the relative projective objects, and $J_0$ and $J_1$ are relative injective and $j$ is a left approximation with respect to the relative injective objects. Now let $g\colon \tau^-(A)\to A'$ be an arbitrary morphism. Consider the exact sequence
\[
0\to \nu^-(A)\xrightarrow{\nu^-(j)} \nu^-(J_1)\xrightarrow{\nu^-(t)}\nu^-(J_0)\to \tau^-(A)\to 0
\]
The composite $\nu^-(J_0)\to \tau^-(A)\xrightarrow{g}A'$ factors through $p$ via a morphism $u\colon \nu^-(J_0)\to Q_0$ since $p$ is a right approximation and $\nu^-(J_0)$ is relative projective. Hence, there is a commutative diagram
\begin{equation}\label{lifting morphism}
\begin{tikzcd}
&\nu^-(J_1)\arrow{r}{\nu^-(t)}\arrow{d}{v}&\nu^-(J_0)\arrow{r}{}\arrow{d}{u}&\tau^-(A)\arrow{d}{g}\arrow{r}&0 \\
0\arrow{r}&Q_1\arrow{r}{s}&Q_0\arrow{r}{p}&A'\arrow{r}&0
\end{tikzcd}
\end{equation}
with exact rows, where $v$ is induced from the commutativity of the right square. Applying $\nu$ to the leftmost square and composing with the isomorphisms $\nu\nu^-(J_1)\cong J_1$ and $\nu\nu^-(J_0)\cong J_0$, results in the commutative square
\[
\begin{tikzcd}
J_1\arrow{r}{t}\arrow{d}{\tilde{v}}&J_0\arrow{d}{\tilde{u}} \\
\nu(Q_1)\arrow{r}{\nu(s)}&\nu(Q_0)
\end{tikzcd}
\]
Since $\ker t \cong A$ and $\ker \nu(s)\cong \tau(A)$, we get an induced morphism $\overline{g}\colon A\to \tau(A)$. Note that $\overline{g}$ does not depend on the choice of $u$ and $v$ in \eqref{lifting morphism}. In fact, if $u'\colon \nu^-(J_0)\to Q_0$ and $v'\colon \nu^-(J_1)\to Q_1$ are two different morphisms making the diagram \eqref{lifting morphism} commute, then there exists a morphism $w\colon \nu^-(J_0)\to Q_1$ such that $s\circ w = u-u'$ and $w\circ \nu^-(t)=v-v'$. Hence, applying $\nu$ and composing with the isomorphism $\nu\nu^-(J_1)\cong J_1$, we get a morphism $\overline{w}\colon J_0\to \nu(Q_1)$ satisfying $\nu(s)\circ \overline{w} = \overline{u}-\overline{u'}$ and $\overline{w}\circ t=\overline{v}-\overline{v'}$. Therefore, the induced morphisms $A\to \tau(A)$ coming from $(\overline{u},\overline{v})$ and coming from $(\overline{u'},\overline{v'})$ coincide. Hence, there is a well defined map $\mathcal{C}^{T(X)}(\tau^-(A),A')\to \mathcal{C}^{T(X)}(A,\tau (A'))$ sending $g$ to $\overline{g}$. This map is an isomorphism, since the dual construction clearly gives an inverse $\mathcal{C}^{T(X)}(A,\tau (A'))\to \mathcal{C}^{T(X)}(\tau^-(A),A')$. A straightforward computation shows that it is a natural in $A$ and $A'$, which proves the claim.
\end{proof}

We continue by studying the derived functors of $\Kopf_X$ and $\soc_Y$ in our setup. 

\begin{lem}\label{Ltopvanishesonprojectives}
Let $\mathcal{C}$, $X$, and $Y$ be as in Assumption \ref{standard1}. 
Then the left and right derived functors $L_j\Kopf_{X}$ and $R^j\soc_{Y}$ exist for all $j>0$ and vanish on the image of $f_!$ and $f_*$, respectively.
\end{lem}

\begin{proof}
We only prove the claims about $\Kopf_X$, the proofs for the claims about $\soc_Y$ are dual. As $\Kopf_{X}$ is a left adjoint by Lemma \ref{topisaleftadjoint}, it is right exact. Note that $f^*$ is faithful. To prove existence, by \cite[Proposition 3.1.4]{Kva16}, it suffices to show that $f_!f^*$ is exact and $\Kopf_{X}\circ f_!f^*$ is exact. As $f^*\nu$ is left adjoint to $f_!$, see \cite[Lemma 3.2.2]{Kva16}, it follows that $f_!$ is exact and therefore the composite $f_!f^*$ is exact. Also, by Lemma \ref{topofprojective}, we have the identity $\Kopf_X f_!=\Id$, and hence the composite $\Kopf_Xf_!f^*=f^*$ is exact. This proves the claims. 
\end{proof}

Using Lemma \ref{globaldimension1}, the derived functors of $\Kopf_X$ and $\soc_Y$ admit an explicit description.

\begin{lem}\label{derivedoftop}
Let $\mathcal{C}$, $X$, and $Y$ be as in Assumption \ref{standard1}. Let $(M,h)\in \mathcal{C}^{T(X)}$. 
Then,  $L_1\Kopf_{X} (M,h)=\ker h_1$ and $L_i\Kopf_{X} (M,h)=0$ for all $i\geq 2$. Dually, $R^1\soc_Y(M,h)=\coker h_1$ and $R^j\soc_Y(M,h)=0$ for all $j\geq 2$. 
\end{lem}

\begin{proof}
By Lemma \ref{Ltopvanishesonprojectives}, $L_j\Kopf_{X}$ exists for all $j>0$.
Applying $\Kopf_{X}$ to the short exact sequence in Lemma \ref{globaldimension1} yields the long exact sequence 
\[
\begin{tikzcd}
\dots\arrow{r}&L_1\Kopf_{X}(f_!f^*(M,h))\arrow{r} &L_1\Kopf_{X}(M,h)\arrow{lld}\\
\Kopf_{X} f_!(X(f^*(M,h)))\arrow{r} &\Kopf_{X} f_! f^*(M,h)\arrow{r}&\Kopf_{X}(M,h)\arrow{r} &0.
\end{tikzcd}
\]
Since $L_i\Kopf_X(f_!f^*(M,h))=0$ and $L_i\Kopf_X(f_!Xf^*(M,h))=0$ for $i>0$ by Lemma \ref{Ltopvanishesonprojectives}, it follows that $L_i\Kopf_X(M,h)=0$ for $i>1$. Furthermore, using Lemma \ref{topofprojective} and Lemma \ref{Ltopvanishesonprojectives} we obtain the exact sequence
$0\to L_1\Kopf_{X}(M,h)\to X(M)\xrightarrow{h_1} M\to \Kopf_{X} (M,h)\to 0$.
The claim follows.
\end{proof}

The following lemma shows the equivalence of different formulations of Nakayama's lemma in our context. 

\begin{lem}\label{nakayama}
Let $\mathcal{C}$, $X$, and $Y$ be as in Assumption \ref{standard1}. 
The following are equivalent:
\begin{enumerate}[(1)]
\item For $M\neq 0$ in $\mathcal{C}$, there are no epimorphisms $X(M)\to M$ in $\mathcal{C}$.
\item For all $(M,h)\in \mathcal{C}^{T(X)}$ we have that $\Kopf_{X} (M,h)=0$ if and only if $(M,h)=0$.
\item For all morphisms $\varphi$ in $\mathcal{C}^{T(X)}$ we have that $\Kopf_{X} \varphi$ is an epimorphism if and only if $\varphi$ is an epimorphism.
\end{enumerate}
\end{lem}

\begin{proof}
(1) and (2) are equivalent by definition of $\Kopf_{X}$. To see that (3) implies (2) note that $M=0$ if and only if the morphism $0\to M$ is an epimorphism. For the remaining direction, (2) implies (3), consider the following diagram:
\[
\begin{tikzcd}
X(M)\arrow{r}{h_1}\arrow{d}{X(\varphi)} &M\arrow{r}{\pi_M}\arrow{d}{\varphi} &\Kopf_{X}(M,h)\arrow{d}{\Kopf_{X}(\varphi)}\arrow{r}&0\\
X(N)\arrow{d}{X(p)}\arrow{r}{g_1} &N\arrow{r}{\pi_N}\arrow{d}{p} &\Kopf_{X}(N,g)\arrow{r}\arrow{d}&0\\
X(\Coker \varphi)\arrow{r}{\overline{g_1}} &\Coker\varphi\arrow{r} &\Coker(\Kopf_{X}(\varphi))=\Kopf_{X}(\coker \varphi, \overline{g_1})\arrow{r}&0
\end{tikzcd}
\]
where equality on the lower right corner follows e.g. from the snake lemma. Thus, if $\Kopf_{X} \varphi$ is an epimorphism, then $\Kopf_{X}(\Coker\varphi, \overline{g_1})=0$. Hence, $\Coker \varphi=0$ implying that $\varphi$ is an epimorphism. 
\end{proof}

We are not able to show that the equivalent conditions of Lemma \ref{nakayama} hold in general. However, if $\mathcal{C}$ has enough projectives, or $X$ comes from a phylum, then the equivalent conditions hold, see Lemma \ref{specialcasenakayama} and Meta-Example \ref{nakayamaforphyla} below. 
We start with the following well-known lemma. We don't expect the equivalent conditions to hold in general since the situation seems very similar to \cite{Nee02} regarding the question whether the inverse limit of epimorphisms is an epimorphism. However, the counterexample in \cite{Nee02} is not a counterexample to our setup. 

\begin{lem}
Let $\mathcal{C}$ be an additive category. Let $M$ be an object such that the natural morphism $\coprod_{i\geq 0} M\to \prod_{i\geq 0} M$ is an isomorphism. Then, $M=0$.
\end{lem}

\begin{proof}
The proof is a version of the Eilenberg swindle. Define $f_1$ as the composition of the morphism $M\to \prod_{i\geq 0} M$ given by $0$ on the zeroth component and $1_M$ on all the other components. Denote the inverse of the natural homomorphism $\coprod_{i\geq 0} M\to \prod_{i\geq 0}M$ by $f_2$ and define $f_3$ as the map $\coprod_{i\geq 0} M\to M$ by zero on the zeroth component and $1_M$ on all the other components. Set $f:=f_3f_2f_1$. Similarly define $f'_i$ to be $1_M$ on every component for $i=1,3$. It is easy to see that the following diagram is commutative:
\[
\begin{tikzcd}[row sep=1ex]
&\prod_{i\geq 0} M\arrow{ddd}{g}\arrow{r}{f_2}&\coprod_{i\geq 0} M\arrow{ddd}{g'}\arrow{rd}{f_3'}\\
M\arrow{ru}{f_1'}\arrow[equals]{d}&&&M\arrow[equals]{d}\\
M\arrow{rd}{f_1} &&&M\\
&\prod_{i\geq 0} M\arrow{r}{f_2}&\coprod_{i\geq 0} M\arrow{ru}{f_3} 
\end{tikzcd}
\]
where $g$ and $g'$ are defined as follows: Denote by $p_i\colon \prod_{i\geq 0} M\to M$ the $i$-th projection and by $\iota_i\colon M\to \coprod_{i\geq 0} M$ the $i$-th inclusion. Then $g\colon \prod_{i\geq 0} M\to \prod_{i\geq 0} M$ is defined to be the map whose $i$-th component $\prod_{i\geq 0} M\to M$ is equal to $p_{i+1}$ while $g'\colon \coprod_{i\geq 0} M\to \coprod_{i\geq 0} M$ is defined to be the map whose $i$-th component $M\to \coprod_{i\geq 0} M$ is equal to $\iota_{i+1}$. 
Going from left to right through the upper part is equal to $f+1_M$ while the path through the lower part is equal to $f$. Thus $f+1_M=f$ and hence $1_M=0$, implying that $M=0$. 
\end{proof}

The following lemma shows that the analogue of Nakayama's lemma holds in case the category $\mathcal{C}$ has enough projectives.

\begin{lem}\label{specialcasenakayama}
Let $\mathcal{C}$, $X$, and $Y$ be as in Assumption \ref{standard1}. 
\begin{enumerate}[(i)]
\item Assume that the inverse limit $M^{\infty}$ (if it exists) of a sequence 
\[\dots\to M^1\to M^0\to 0\]of epimorphisms in $\mathcal{C}$ with $M^0\neq 0$ satisfies $M^{\infty}\neq 0$.
Then $X$ satisfies the equivalent conditions of Lemma \ref{nakayama}.
\item If $\mathcal{C}$ has enough projectives then $\mathcal{C}$ satisfies the assumption in (i).
\end{enumerate}
\end{lem}

\begin{proof}
Let $(M,h)\in \mathcal{C}^{T(X)}$ with $h_1$ an epimorphism. Consider the sequence
\[\dots\to X^2(M)\xrightarrow{X(h_1)} X(M)\xrightarrow{h_1} M.\] Let $X^\infty(M)$ denote the inverse limit of this sequence. Note that this inverse limit exists because it can be written as the kernel of the morphism 
\[\prod_{i\geq 0}X^i(M)\to \prod_{j\geq 0} X^j(M)\]
which is given as the difference between the projection $p_k\colon \prod_{i\geq 0} X^i(M)\to X^k(M)$ and the composite $X^k(h_1)\circ p_{k+1}\colon \prod_{i\geq 0} X^i(M)\to X^k(M)$ on component $k$. 
Note that there is an isomorphism $X(X^\infty(M))\cong X^\infty(M)$ as the former computes $\varprojlim_{i\geq 1} X^i(M)=\varprojlim X^i(M)$. It follows that
\[\coprod_{i\geq 0} X^\infty(M)\cong\coprod_{i\geq 0} X^i(X^\infty(M))\cong\prod_{i\geq 0} X^i(X^\infty(M))\cong\prod X^\infty(M).\] 
By the preceding lemma, $X^\infty(M)=0$ hence it follows that $M=0$ by assumption.
 
It is well-known that in the case of a category with enough projectives, the inverse limit of an inverse system of non-zero epimorphisms is non-zero, see \cite[Introduction]{Nee02} for the dual property. For the convenience of the reader we spell out the argument: Let $\dots\xrightarrow{h^2} M^1\xrightarrow{h^1} M^0$ be a sequence of non-zero epimorphisms in $\mathcal{C}$ with inverse limit $M^\infty$ and let $h^\infty\colon M^\infty\to M^0$ be the canonical morphism. Let $g\colon P\twoheadrightarrow M^0$ be an epimorphism with $P$ projective. Choose inductively homomorphisms $g^i\colon P\to M^i$   such that $g^0=g$ and $h^i\circ g^i=g^{i-1}$ . It follows from the universal property of $X^\infty(M)$, that there exists a map $g'\colon P\to M^\infty$ such that $g=h^\infty g'$. Hence if $M^0\neq 0$ then $g\neq 0$ and therefore $g'\neq 0$ and in particular $M^\infty\neq 0$.  This proves the claim.
\end{proof}


Dual statements hold for $\soc_Y$ and a sequence of monomorphisms. 
One instance, where Nakayama's lemma holds is the case of phyla on locally bounded quivers:

\begin{mex}\label{nakayamaforphyla}
Let $\mathfrak{A}$ be a phylum on a locally bounded quiver $Q$. Let $\mathcal{C}=\prod_{\mathtt{i}\in Q_0} \mathcal{A}_\mathtt{i}$ and let $X$ be as in Example \ref{runningexample1}. Then the equivalent conditions in Lemma \ref{nakayama} hold for $X$: Indeed for $\mathtt{i}\in Q_0$ let $d(\mathtt{i})$ be the length of the longest path from a source vertex to $\mathtt{i}$. Assume there exists an epimorphism $X(M)\to M$ where $M=(M_\mathtt{i})_{\mathtt{i}\in Q_0}$ is nonzero and choose a vertex $\mathtt{i}\in Q_0$ with $M_\mathtt{i}\neq 0$ and $d(\mathtt{i})$ as small as possible. For any arrow $\alpha$ with target $\mathtt{i}$ we then get that $M_{s(\alpha)}=0$ by assumption. Therefore $X(M)_\mathtt{i}=\bigoplus_{\substack{\alpha\in Q_1\\t(\alpha)=\mathtt{i}}}F_\alpha(M_{s(\alpha)})=0$. If there were an epimorphism $X(M)_\mathtt{i}\to M_\mathtt{i}$ it would follow that $M_\mathtt{i}=0$, a contradiction. Hence, $M=0$ and the claim follows.
\end{mex}

\begin{lem}
Let $\mathcal{C}$, $X$, and $Y$ be as in Assumption \ref{standard1}. 
Then, $\soc_{Y}\circ \nu =\Kopf_{X} $ on the full subcategory $\image f_!$. 
\end{lem}

\begin{proof}
It follows from Lemma \ref{topofprojective} that this equality holds on objects. In the description of Lemma \ref{morphismsbetweenPprojectives}, it then follows that $\Kopf_X \varphi=\varphi_{0,0}$. For the left hand side note that by Lemma \ref{morphismsbetweenPprojectives}, $\varphi$ is determined by maps $\varphi_{0,k}\colon M\to X^k(N)$ for $k\geq 0$. This is then sent by $\nu$ to a map $\psi$ in $(\Id\Downarrow Y)$ determined by $(\adj^{Y^k\dashv X^k})^{-1}(\varphi_{0,k})\colon Y^k(M)\to N$ according to the dual of Lemma \ref{morphismsbetweenPprojectives}. But the socle of that map is given by $(\adj^{Y^0\dashv X^0})^{-1}(\varphi_{0,0})=\varphi_{0,0}$ proving the claim. 
\end{proof}

The following theorem shows that the category $\relGproj{P}{\mathcal{C}^{T(X)}}$ can be described as a monomorphism category.

\begin{thm}\label{characterisationsofmono}
Let $\mathcal{C}$, $X$, and $Y$ be as in Assumption \ref{standard1}. 
Assume furthermore for $M\neq 0$ there are no monomorphisms $M\to Y(M)$ in $\mathcal{C}$. Then, the following are equivalent for $(M,h)\in \mathcal{C}^{T(X)}$. 
\begin{enumerate}[(1)]
\item $(M,h)\in \mathcal{G}_P \proj(\mathcal{C}^{T(X)})$,
\item $L_1\nu (M,h) =0$,
\item $(M,h)\in \im \nu^-$,
\item $(M,h)$ is a submodule of $f_!(N)$ for some $N\in \mathcal{C}$,
\item $h_1\colon X(M)\to M$ is a monomorphism,
\item $L_1 \Kopf_X(M,h)=0$.
\end{enumerate}
\end{thm}

\begin{proof}
The equivalence of (1) and (3) is given by Corollary \ref{Gorensteinprojectiveimagenu}. By Lemma \ref{globaldimension1}, $f_!\dashv f^*$ is Iwanaga-Gorenstein of dimension $n\leq 1$, thus $L_i\nu=0$ for all $i\geq 2$. It follows that (1) and (2) are equivalent. By \cite[Lemma 4.2.3]{Kva16}, $(M,h)\in \image \nu^-$ implies that $(M,h)$ is a submodule of a relative projective, in particular, it is a submodule of some $f_!(N)$ and hence (3) implies (4). 
Conversely, since $f_!\dashv f^*$ is Iwanaga--Gorenstein of dimension $n\leq 1$ we have that $\dim_{\mathcal{G}_P\proj(\mathcal{C}^{T(X)})}(\mathcal{C}^{T(X)})\leq 1$ by Theorem \ref{resolution_dimension} and hence any submodule of an object in $\relGproj{P}{\mathcal{C}^{T(X)}}$ is in $\relGproj{P}{\mathcal{C}^{T(X)}}$ (cf. \cite[Lemma 2.3]{Sto14}). It follows that (4) implies (1). 
By Lemma \ref{derivedoftop}, $L_1\Kopf_{X}(M,h)=\ker h_1$. Thus, (5) and (6) are equivalent. 
By definition of $f_!(M)$, it is easy to see that (4) implies (5).

For the missing direction $(6)\Rightarrow (2)$, note that $L_1\Kopf_X(M,h)=0$ implies that $\Kopf_X \psi$ is a monomorphism  where $\psi\colon (f_!X)(M)\to f_!(M)$ is the leftmost map in the exact sequence in Lemma \ref{globaldimension1}. But $\Kopf_X \psi=\soc_Y\nu(\psi)$ by the foregoing lemma. By the dual of Lemma \ref{nakayama} this is a monomorphism if and only if $\nu(\psi)$ is a monomorphism. Thus, if and only if $L_1\nu(M,h)=0$. 
\end{proof}

There is a dual characterisation of relative Gorenstein injectives.
%

Let $\Lambda$ be a generalized species on a locally bounded quiver. Consider the subcategory $\repfd (\Lambda)$ of $\rep (\Lambda)$ consisting of all species of finite total dimension. Analogously we can define the subcategories $\Monofd(\Lambda)$ and $\Epifd(\Lambda)$. All the functors $f_!$, $f^*$, $f_*$, $\nu$ and $\nu^-$ restrict to $\repfd(\Lambda)$ and it is straightforward to adjust the proofs to see that being relative Gorenstein projective in $\repfd (\Lambda)$ is equivalent to being contained in $\Monofd(\Lambda)$, which in turn is equivalent to say that \[M_{\mathtt{i},\arrowin}=(M_\alpha)_\alpha\colon \bigoplus_{\substack{\alpha\in Q_1\\t(\alpha)=\mathtt{i}}}F_\alpha(M_{s(\alpha)})\to M_{\mathtt{i}}\] is a monomorphism for all $\mathtt{i}\in Q_0$. Dually, being relative Gorenstein injective means that $M_{\mathtt{i},\arrowout}$ is an epimorphism. 

\begin{thm}\label{thmCmaintext}
Let $\Lambda$ be a dualisable pro-species on a locally bounded quiver $Q$. 
\begin{enumerate}[(i)]
\item\label{thmCmaintext:i} The categories $\Monofd(\Lambda)$ and $\Epifd(\Lambda)$ have almost split sequences. 
\item\label{thmCmaintext:ii} Let $g$ be a right almost split morphism in $\rep(\Lambda)$. Then $\nu^-(g)$ is a split epimorphism or a right almost split morphism in $\Mono(\Lambda)$. A similar statement holds for $\repfd (\Lambda)$, and a dual statement holds for the epimorphism category. 
\item\label{thmCmaintext:iii} Let $0\to A''\xrightarrow{g'} A\xrightarrow{g} A'\to 0$ be an almost split sequence in $\rep(\Lambda)$ with $A'\in \Epi(\Lambda)$ such that $A'$ is not projective in $\Epi(\Lambda)$. Then, 
\[0\to \nu^-(A'')\xrightarrow{\nu^-(g')} \nu^-(A)\xrightarrow{\nu^-(g)} \nu^-(A')\to 0\]
is a direct sum of an almost split sequence and a split exact sequence in $\Mono(\Lambda)$. A similar statement holds for $\repfd (\Lambda)$, and a dual statement holds for the epimorphism category 
\end{enumerate}
\end{thm}

\begin{proof}
 Since $\repfd(\Lambda)$ can be identified with finitely presented modules over a locally bounded category, it follows that $\repfd(\Lambda)$ is a dualizing variety, and therefore has almost split sequences, see page 343 in \cite{AR74}. It follows by Corollary \ref{Existence of almost split sequences} that $\Monofd(\Lambda)$ has almost split sequences. This together with the dual argument proves \eqref{thmCmaintext:i}.
Part \eqref{thmCmaintext:ii} is a special case of Corollary \ref{thmCgeneralform} which satisfies our assumptions by Lemma \ref{globaldimension1}.
Part \eqref{thmCmaintext:iii} is a special case of Corollary \ref{imagesofseveralalmostsplitsequences}, again using Lemma \ref{globaldimension1}. 
\end{proof}

\section{Almost split sequences via the preprojective algebra} \label{sec:preprojective}

In this section, we define the preprojective monad associated to an endofunctor $X$ with an ambidextrous adjunction. Furthermore we show that the Eilenberg--Moore category of the preprojective monad can also be realised as the Eilenberg--Moore category of the free monad associated to the Auslander--Reiten translation, generalising \cite{Ri98}. Finally, we provide a proof of Theorem \ref{thmD}.

We define the category $\mathcal{C}^{\Pi(X)}$ as the full subcategory of $(X,Y\Downarrow \Id)$ consisting of those morphisms $h_1\colon X(M)\to M, h_1'\colon Y(M)\to M$ such that $h_1'Y(h_1)\unit{X}{Y}_M=h_1X(h_1')\unit{Y}{X}_M$. The notation suggests that this category can be realized as an Eilenberg--Moore category over a monad $\Pi(X)$. We will show that this is indeed the case under further assumptions on $\mathcal{C}$. A possible such assumption is that $\mathcal{C}$ is cocomplete. This should be compared to the fact that the preprojective algebra is usually infinite dimensional although the path algebra is finite dimensional. In the remainder of the paper we will however freely use the notation $\mathcal{C}^{\Pi(X)}$ even if the monad $\Pi(X)$ is not defined. 

To construct the preprojective monad we start by adapting the notion of a quotient monad to our setting of abelian categories. It should be compared to the quotient of a ring by the ideal generated by a set of elements. For further reading on quotient monads, see \cite{BHKR15}.

\begin{defn}
Let $\mathcal{C}$ be an abelian category. Let $T$ be a right exact monad on $\mathcal{C}$, let $E$ be a right exact  endofunctor on $\mathcal{C}$ and let $\varphi\colon E\to T$ be a natural transformation. Then the \emphbf{quotient monad} $T'=T/\varphi$ is defined to be the functor given on objects by 
\[T'(M):=\coker(TET(M)\xrightarrow{T(\varphi_{T(M)})} T^3(M)\xrightarrow{T(\mu)} T^2(M)\xrightarrow{\mu} T(M))\] 
and on morphisms by the universal property of the cokernel inducing the following commutative diagram
\[
\begin{tikzcd}
TET(M)\arrow{r}\arrow{d}{TET(f)} &T(M)\arrow{d}{T(f)}\arrow{r}&T'(M)\arrow[dashed]{d}{T'(f)}\\
TET(N)\arrow{r} &T(N)\arrow{r}&T'(N)
\end{tikzcd}
\]
for a morphism $f\colon M\to N$ in $\mathcal{C}$. 
\end{defn}

To state the universal property of the quotient monad we recall the definition of a morphism between monads.

\begin{defn}\label{Morphism of monad}
Let $(T_1,\eta_1,\mu_1)$ and $(T_2,\eta_2,\mu_2)$ be monads on a category $\mathcal{C}$. A morphism $(T_1,\eta_1,\mu_1)\to (T_2,\eta_2,\mu_2)$ of monads is a natural transformation $T_1\xrightarrow{\delta} T_2$ making the following diagrams commutative 
\begin{equation*}
\begin{tikzcd}
T_1T_1 \arrow{d}[swap]{\delta_{T_2}\circ T_1(\delta)} \arrow{r}{\mu_1} & T_1 \arrow{d}{\delta} \\
T_2T_2 \arrow{r}{\mu_2} & T_2 
\end{tikzcd}
\qquad 
\begin{tikzcd}
\operatorname{Id} \arrow{d}{1} \arrow{r}{\eta_1} & T_1 \arrow{d}{\delta} \\
\operatorname{Id} \arrow{r}{\eta_2} & T_2. 
\end{tikzcd}
\end{equation*}
\end{defn}

Note that a morphism $\delta\colon T_1\to T_2$ of monads induces a faithful functor $\delta^*\colon \mathcal{C}^{T_2}\to \mathcal{C}^{T_1}$ between the Eilenberg--Moore categories, given by sending an object $g\colon T_2(M)\to M$ to $g\circ \delta_M\colon T_1(M)\to M$ and acting as identity on morphisms, see \cite[Proposition 4.5.9]{Bor94a}.

The following lemma shows that the quotient monad has indeed the structure of a monad. 

\begin{lem}
Let $\mathcal{C}$ be an abelian category. Let $T$ be a right exact monad on $\mathcal{C}$ and let $E$ be a right exact endofunctor on $\mathcal{C}$. Let $\varphi\colon E\to T$ be a natural transformation and denote by $T'=T/\varphi$ the corresponding quotient monad. Then, there exist unique natural transformations $\mu'\colon T'T'\to T'$ and $\eta'\colon \Id\to T'$ such that the triple $(T',\eta',\mu')$ is a monad, and such that the projection $p\colon T\to T'$ is a morphism of monads.
\end{lem}

\begin{proof}
Let $\psi$ denote the composite $TET\xrightarrow{T(\varphi_T)}T^3\xrightarrow{T(\mu)}T^2\xrightarrow{\mu}T$. A straightforward computation shows that  
$\mu\circ T(\psi) = \psi \circ \mu_{ET} \quad \text{and} \quad \mu\circ \psi_T = \psi\circ TE(\mu)$.
Hence, the universal property of the cokernel defines a morphism $TT'\xrightarrow{l} T'$ and a commutative diagram
\begin{equation*}
\begin{tikzcd}
T^2ET \arrow{d}{\mu_{ET}} \arrow{r}{T(\psi)} & T^2 \arrow{d}{\mu} \arrow{r}{T(p)} & TT' \arrow{r} \arrow[dashed]{d}{l} & 0 \\
TET \arrow{r}{\psi} & T \arrow{r}{p} & T' \arrow{r} & 0
\end{tikzcd}
\end{equation*}
with exact rows. One can check that the composite 
$TETT'\xrightarrow{\psi_{T'}}TT'\xrightarrow{l}T' $
vanishes. 
Since there is a right exact sequence
$TETT'\xrightarrow{\psi_{T'}}TT'\xrightarrow{p_{T'}}T'T'\to 0$ 
it follows that the morphism $l$ factors through $p_{T'}$, i.e. that there exists a morphism $\mu'\colon T'T'\to T'$ satisfying $\mu'\circ p_{T'} =l$. Now let $\eta' := p\circ \eta \colon \Id\to T'$. It is now straightforward to check that the triple $(T',\eta',\mu')$ is a monad on $\mathcal{C}$. 
Finally, $\mu'$ and $\eta'$ are the unique natural transformations making the diagrams
\begin{equation*}
\begin{tikzcd}
TT \arrow{d}{} \arrow{r}{\mu} & T \arrow{d}{} \\
T'T' \arrow{r}{\mu'} & T'
\end{tikzcd}
\quad 
\begin{tikzcd}
\operatorname{Id} \arrow{d}{\id} \arrow{r}{\eta} & T \arrow{d}{} \\
\operatorname{Id} \arrow{r}{\eta'} & T'
\end{tikzcd}
\end{equation*}
commutative, since the leftmost vertical morphisms in these diagrams are epimorphisms. In other words, $\mu'$ and $\eta'$ are the unique natural transformations making the projection $p\colon T\to T'$ into a morphism of monads. This proves the lemma.
\end{proof}

\begin{lem}\label{modulecategoryofquotientmonad}
Let $\mathcal{C}$ be an abelian category. Let $T$ be a right exact monad on $\mathcal{C}$ and let $E$ be a right exact endofunctor on $\mathcal{C}$. Let $\varphi\colon E\to T$ be a natural transformation and denote by $T'=T/\varphi$ the corresponding quotient monad with projection morphism $p\colon T\to T'$. The functor $p^*\colon \mathcal{C}^{T'}\to \mathcal{C}^T, (M,h)\mapsto (M,h\circ p_M)$ is full and faithful, and its image consists of all morphisms $h\colon T(M)\to M$ satisfying $h\circ \varphi_M =0$
\end{lem}

\begin{proof}
It is immediate that the functor is faithful. We show that it is full. Let $h\colon T'(M)\to M$ and $\tilde{h}\colon T'(\tilde{M})\to \tilde{M}$ be two objects in $\mathcal{C}^{T'}$, and let $\theta\colon p^*(h)\to p^*(\tilde{h})$ be a morphism in $\mathcal{C}^T$, i.e. a morphism $\theta\colon M\to \tilde{M}$ in $\mathcal{C}$ satisfying $\theta\circ h\circ p_{M} = \tilde{h}\circ p_{\tilde{M}} \circ T(\theta)$. By naturality of $p$ we have the equality $p_{\tilde{M}}\circ T(\theta) = T'(\theta)\circ p_{M}$ and hence $\theta\circ h\circ p_{M}=\tilde{h}\circ T'(\theta)\circ p_{M}$. Since $p_M$ is an epimorphism, it follows that $\theta\circ h=\tilde{h}\circ T'(\theta)$, and hence $\theta$ is also a morphism between $h$ and $\tilde{h}$ in $\mathcal{C}^{T'}$. This shows that $p^*$ is full. 

We prove the statement describing the image of $p^*$: Let $h\colon T(M)\to M$ be an object in $\mathcal{C}^{T}$, and assume $h\circ \varphi_M =0$. Let $\psi= \mu\circ T(\mu)\circ T(\varphi_T)$ as in the preceding lemma. Then it follows that $h\circ \psi_M=0$. 
Hence, there exists a morphism $h'\colon T'(M)\to M$ such that $h=h'\circ p_M$. A straightforward computation shows that $h'$ is an object in $\mathcal{C}^{T'}$. 
\end{proof}

\begin{defn}
Let $X\dashv Y\dashv X$ be an ambidextrous adjunction. Assume that $T(X,Y)$ exists (e.g. by assuming that $\mathcal{C}$ is cocomplete). Let $\rho\colon \Id\to T(X,Y)$ be the composition of the unit of the adjunction $\Id\to XY$ and the embedding $XY\to T(X,Y)$. Dually, let $\rho'\colon \Id\to T(X,Y)$ be the composition of the unit of the adjunction $\Id\to YX$ and the embedding $YX\to T(X,Y)$. Then the \emphbf{preprojective monad} $\Pi(X,Y)$ is defined to be the quotient monad $T(X,Y)/(\rho-\rho')$. 
\end{defn}

Note that by Lemma \ref{modulecategoryofquotientmonad} the Eilenberg--Moore category of the preprojective monad is in fact equivalent to the subcategory $\mathcal{C}^{\Pi(X)}$ of $(X,Y\Downarrow \Id)$ as defined at the start of this section. 

The following result is a generalisation of a result of Ringel in the classical setting of the module category of a finite quiver, see \cite{Ri98}. Let us recall the definition of the Auslander--Reiten translation $\tau$ in our setup. Let $\mathcal{C}$, $X$, and $Y$ be as in Assumption \ref{standard1}. By Lemma \ref{globaldimension1}, there exists a short exact sequence
\[0\to (f_!Xf^*)(M,h)\to (f_!f^*)(M,h)\to (M,h)\to 0.\]
Thus, by Remark \ref{hereditaryimpliestauisafunctor}, $\tau=L_1\nu\colon (X\Downarrow \Id)\to (X\Downarrow \Id)$ can be computed by applying $\nu$ to the above sequence to obtain the exact sequence
\[0\to \tau(M,h)\xrightarrow{\kappa} (f_*Xf^*)(M,h) \xrightarrow{\chi}f_*f^*(M,h) \to \nu(M,h)\to 0.\]

\begin{thm}\label{ringelsdescriptionofpreprojective}
Let $\mathcal{C}$, $X$, and $Y$ be as in Assumption \ref{standard1}. 
Then there are isomorphisms  $(\tau^-\Downarrow \Id)\cong \mathcal{C}^{\Pi(X)}\cong (\Id\Downarrow \tau)$ making the following diagram commute:
\[
\begin{tikzcd}[row sep=1ex]
(\tau^{-}\Downarrow \Id)\arrow{r}{\cong}\arrow{rd} &\mathcal{C}^{\Pi(X)}\arrow{d}&(\Id\Downarrow \tau)\arrow{l}[swap]{\cong}\arrow{dl}\\
&\mathcal{C}^{T(X)}
\end{tikzcd}
\]
where the vertical functors are the forgetful functors. 
\end{thm}

\begin{proof}
Firstly we define the functor $F\colon \mathcal{C}^{\Pi(X)}\to (\Id\Downarrow \tau)$: Let $(M,h_1,h_1')\in \mathcal{C}^{\Pi(X)}$. Using the adjunction $Y\dashv X$, the map $h_1'\colon Y(M)\to M$ corresponds to a map $\adj^{Y\dashv X}(h_1')\colon M\to X(M)$. Using the adjunction $f^*\dashv f_*$, this map gives rise to a morphism $\varphi\colon (M,h_1)\to f_*(X(M))$ in $(X\Downarrow \Id)$. The definition of $\tau$, using Proposition \ref{globaldimension1} and Remark \ref{hereditaryimpliestauisafunctor}, gives rise to the following  diagram
\begin{equation}
\label{diagramdefiningpsi}
\begin{tikzcd}
{}&{}&(M,h_1)\arrow{d}{\varphi}\arrow{rd}{0}\arrow[dashed]{ld}[swap]{\psi}\\
0\arrow{r}&\tau(M,h_1)\arrow{r}{\kappa} &f_*(X(M))\arrow{r}{\chi} &f_*(M).
\end{tikzcd}
\end{equation}
As a next step we show commutativity of the right hand triangle. By the universal property of the kernel, this means that there is a unique map $\psi$ making the diagram commutative. 

Since every map $(M,h_1)\to f_*(M)$ is uniquely determined by a map $M\to M$ and by Proposition \ref{definitionofnu} the only non-zero components of $f_*(X(M))\to f_*(M)$ whose image lies in $M$ are $-h_1\colon X(M)\to M$ and $\varepsilon^{Y\dashv X}_M\colon (YX)(M)\to M$, it suffices to check that the diagram 
\[
\begin{tikzcd}[ampersand replacement=\&, column sep=12ex]
M\arrow{d}[swap]{\begin{pmatrix}\varphi_0\\\varphi_1\end{pmatrix}}\arrow{rd}{0}\\
X(M)\oplus (YX)(M)\arrow{r}[swap]{(-h_1,\varepsilon^{Y\dashv X}_M)}\&M
\end{tikzcd}
\]
commutes. By definition, $\varphi=\adj^{f^*\dashv f_*}(\adj^{Y\dashv X}(h_1'))$. By Remark \ref{eilenbergmooreforfreemonad} we obtain that $\varphi_0=\adj^{Y\dashv X}(h_1')=X(h_1')\eta^{Y\dashv X}_M$ and $\varphi_1=Y(X(h_1')\eta^{Y\dashv X}_M)Y(h_1)\eta^{X\dashv Y}_M$.  Thus, the commutativity follows from 
\begin{align*}
&(-h_1,\varepsilon^{Y\dashv X}_M)\begin{pmatrix}X(h_1') \eta^{Y\dashv X}\\YX(h_1')Y(\eta^{Y\dashv X}_M)Y(h_1)\eta^{X\dashv Y}_M\end{pmatrix}\\
=&-h_1X(h_1')\eta^{Y\dashv X}_M+\varepsilon^{Y\dashv X}_M YX(h_1')Y(\eta^{Y\dashv X}_M)Y(h_1)\eta^{X\dashv Y}_M\\
=&-h_1X(h_1')\eta^{Y\dashv X}_M+h_1'\varepsilon^{Y\dashv X}_{Y(M)} Y(\eta^{Y\dashv X}_M) Y(h_1)\eta^{X\dashv Y}_M\\
=&-h_1 X(h_1')\eta^{Y\dashv X}_M+h_1' Y(h_1) \eta^{X\dashv Y}_M=0
 \end{align*}
by definition of $\mathcal{C}^{\Pi(X)}$. We define $F$ on objects via $F(M,h_1,h_1')=(M,h_1,\psi)$.

To define $F$ on morphisms let $a\colon (M,h_1,h_1')\to (\tilde{M},\tilde{h}_1,\tilde{h}_2)$ be a morphism in $\mathcal{C}^{\Pi(X)}$. In particular, $a$ defines a morphism $(M,h_1)\to (\tilde{M},\tilde{h}_1)$ in $(X\Downarrow \Id)$. We define $F(a)=a$. To check that $F$ is well-defined it suffices to show that $\tau(a)\psi=\tilde{\psi}a$. Consider the following diagram
\[
\begin{tikzcd}[column sep=3ex, row sep=2ex]
{}&&&(M,h_1)\arrow[dashed]{rdd}[near start]{a}\arrow{lld}[swap]{\psi}\arrow{d}{\varphi}\arrow{rrd}{0}\\
0\arrow{r}&\tau(M,h_1)\arrow[dashed]{rdd}{\tau(a)}\arrow{rr}{\kappa} &&f_*(X(M))\arrow[dashed]{rdd}[near start, swap]{(f_*Xf^*)(a)}\arrow{rr}{\chi} &&f_*(M)\arrow[dashed]{rdd}{(f_*f^*)(a)}\\
&&&&(\tilde{M},\tilde{h}_1)\arrow{lld}[swap]{\tilde{\psi}}\arrow{d}{\tilde{\varphi}}\arrow{rrd}{0}\\
&0\arrow{r}&\tau(\tilde{M},\tilde{h}_1)\arrow{rr}{\tilde{\kappa}}&&f_*(X(\tilde{M}))\arrow{rr}{\tilde{\chi}}&&f_*(M)
\end{tikzcd}
\]

 This equation holds since
\[\tilde{\kappa}\tau(a)\psi=f_*(Xf^*(a))\kappa\psi
= f_*(Xf^*(a))\varphi
=\tilde{\varphi}a
=\tilde{\kappa}\tilde{\psi}a\]
and the fact that $\tilde{\kappa}$ is a monomorphism. This finishes the definition of $F$. It is clear from the definition that $F$ defines a functor. 

To define its inverse $G\colon (\Id\Downarrow \tau)\to \mathcal{C}^{\Pi(X)}$ let $(M,h_1,\psi)\in (\Id\Downarrow \tau)$ where $\psi\colon (M,h_1)\to \tau(M,h_1)$. Consider the sequence
\[0\to \tau(M,h_1)\xrightarrow{\kappa} f_*(X(M))\xrightarrow{\chi} f_*(M).\]

Write $\tau(M,h_1)=(N,g_1)$, and let $\kappa_i\colon N\to (Y^iX)(M)$ denote the $i$th component of the map $f^*(\kappa)\colon N\to \prod_{i\geq 0}(Y^iX)(M)$. By Remark \ref{eilenbergmooreforfreemonad} we have that 
\begin{equation}\label{some name}
\kappa_i= Y^{i}(\kappa_0)\circ Y^{i-1}(\operatorname{adj}^{X\dashv Y}(g_1))\circ \cdots \circ Y(\operatorname{adj}^{X\dashv Y}(g_1))\circ \operatorname{adj}^{X\dashv Y}(g_1).
\end{equation}
Since $\chi\kappa=0$, it follows in particular that 
\begin{equation}\label{some name 2}
h_1\circ \kappa_0 = \counit{Y}{X}_M\circ \kappa_1.
\end{equation}

Define $h_1':Y(M)\to M$ as $h_1'=(\adj^{Y\dashv X})^{-1}(\kappa_0) Y(f^*(\psi))$. Consider the following diagram

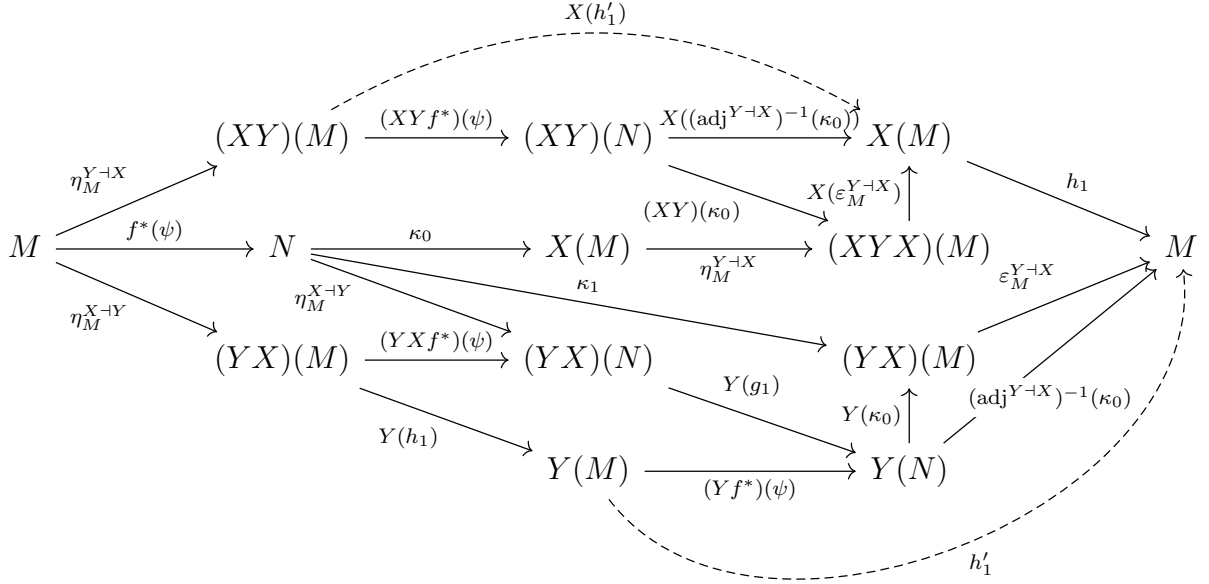
\[
\begin{tikzcd}[column sep=10.5ex]
{}&(XY)(M)\arrow[bend left, dashed]{rr}{X(h_1')}\arrow{r}{(XYf^*)(\psi)}&(XY)(N)\arrow{rd}[swap]{(XY)(\kappa_0)}\arrow{r}{X((\adj^{Y\dashv X})^{-1}(\kappa_0))}&X(M)\arrow{rd}{h_1}\\
M\arrow{ru}{\unit{Y}{X}_M}\arrow{rd}[swap]{\unit{X}{Y}_M}\arrow{r}{f^*(\psi)}&N\arrow{rd}[swap, near start]{\unit{X}{Y}_M}\arrow{rrd}{\kappa_1}\arrow{r}{\kappa_0}&X(M)\arrow{r}[swap]{\unit{Y}{X}_M}&(XYX)(M)\arrow{u}{X(\counit{Y}{X}_M)}&M\\
&(YX)(M)\arrow{r}{(YXf^*)(\psi)}\arrow{rd}[swap]{Y(h_1)}&(YX)(N)\arrow{rd}[near start]{Y(g_1)}&(YX)(M)\arrow{ru}{\counit{Y}{X}_M}\\
&&Y(M)\arrow[bend right=75, dashed]{rruu}[swap]{h_1'}\arrow{r}[swap]{(Yf^*)(\psi)}&Y(N)\arrow{u}{Y(\kappa_0)}\arrow{ruu}[description, yshift=-3ex]{(\adj^{Y\dashv X})^{-1}(\kappa_0)}
\end{tikzcd}
\]

We claim that it is commutative. To check that the small internal polygons commute follows (from top left to bottom right) from 
the definition of $h_1'$, 
$\unit{X}{Y}$ being a natural transformation,
the definition of $(\adj^{Y\dashv X})^{-1}(\kappa_0)$,
the equation \eqref{some name 2} and the triangle identity,
$\unit{X}{Y}$ being a natural transformation,
$\psi$ being a homomorphism in $(X\Downarrow \Id)$,
the equation \eqref{some name} for $i=1$,
the definition of $(\adj^{Y\dashv X})^{-1}(\kappa_0)$,
the definition of $h_1'$.
Reading the boundary we obtain that $h_1X(h_1')\unit{Y}{X}_M=h_1'Y(h_1)\unit{X}{Y}_M$, i.e. $(M,h_1,h_1')\in \mathcal{C}^{\Pi(X)}$. 

To define $G$ on morphisms, consider a morphism $a\colon (M,h_1,\psi)\to (\tilde{M},\tilde{h}_1,\tilde{\psi})$ in $(\Id\Downarrow \tau)$, i.e. a morphism $a\colon (M,h_1)\to (\tilde{M},\tilde{h}_1)$ in $\mathcal{C}^{T(X)}$ satisfying $\tilde{\psi} a=\tau(a) \psi$. Define $G(a)=a$. Well-definedness follows from commutativity of the following diagram:
\[
\begin{tikzcd}[column sep=10ex]
Y(M)\arrow{r}{Yf^*(\psi)}\arrow{d}{(Yf^*)(a)}&(Yf^*)(\tau(M,h_1))\arrow{r}{(\adj^{Y\dashv X})^{-1}(\kappa_0)}\arrow{d}{(Yf^*)(\tau(a))}&M\arrow{d}{f^*(a)}\\
Y(\tilde{M})\arrow{r}{Yf^*(\tilde{\psi})}&(Yf^*)(\tau(\tilde{M},\tilde{h}_1))\arrow{r}{(\adj^{Y\dashv X})^{-1}(\tilde{\kappa}_0)} &\tilde{M}
\end{tikzcd}
\]
Commutativity of the left square follows from $a$ being a morphism in $(\Id\Downarrow \tau)$ while commutativity of the right square follows from applying the adjunction isomorphisms corresponding to $Y\dashv X$ and $f^*\dashv f_*$ to the   square 
\[
\begin{tikzcd}
\tau(M,h_1)\arrow{r}{\kappa}\arrow{d}{\tau(a)} &(f_*X)(M)\arrow{d}{(f_*X)(f^*(a))}\\
\tau(\tilde{M}, \tilde{h}_1)\arrow{r}{\tilde{\kappa}} &(f_*X)(\tilde{M})
\end{tikzcd}
\] 
which commutes by definition of $\tau(a)$. 

That $GF(a)=a$ and $FG(a)=a$ on morphisms is clear. Thus $F$ and $G$ are fully-faithful. 

We are left showing that $GF=\Id$ and $FG=\Id$ on objects. Let $(M,h_1,h_1')\in \mathcal{C}^{\Pi(X)}$. Then $F((M,h_1,h_1'))=(M,h_1,\psi)$ where $\psi$ is defined by the diagram \eqref{diagramdefiningpsi} where $\varphi$ is determined by the $0$-th component $\varphi_0=\adj^{Y\dashv X}(h_1')$ of the underlying map in $\mathcal{C}$. Define $\tilde{h}_1'$ via $(M,h_1,\tilde{h}_1')=(GF)(M,h_1,h_1')$, i.e. $\tilde{h}_1'=(\adj^{Y\dashv X})^{-1}(\kappa_0)\circ Y(f^*(\psi))$. Applying the adjunction isomorphism to this equation we obtain that $\adj^{Y\dashv X}(\tilde{h}_1')=\kappa_0f^*(\psi)$, which by the commutativity of \eqref{diagramdefiningpsi} is equal to $\varphi_0=\adj^{Y\dashv X}(h_1')$. Applying $(\adj^{Y\dashv X})^{-1}$ we obtain $h_1'=\tilde{h}_1'$. This proves $GF=\Id$ on objects. 

To prove $FG=\Id$ on objects, let $(M,h_1,\psi)\in (\Id\Downarrow \tau)$. Then $G(M,h_1,\psi)=(M,h_1,h_1')$ where $h_1'=(\adj^{Y\dashv X})^{-1}(\kappa_0)Y(f^*(\psi))$. Define $\tilde{\psi}$ via $(FG)(M,h_1,\psi)=(M,h_1,\tilde{\psi})$. We want to show that $\psi=\tilde{\psi}$. For this, it suffices to show that $\psi$ makes \eqref{diagramdefiningpsi} commutative. Since $\varphi$ and $\kappa$ are determined in degree $0$, this follows from $\varphi_0=\adj^{Y\dashv X}(h_1')=\kappa_0f^*(\psi)$. This finishes the proof that $\mathcal{C}^{\Pi(X)}\cong (\Id\Downarrow \tau)$.

The statement that $(\tau^-\Downarrow \Id)$ is equivalent to $(\Id\Downarrow \tau)$ follows from adjointness of $\tau^-$ and $\tau$, see Lemma \ref{tautau-adjoint}. The statement about the forgetful functor follows immediately from the definition of $F$ and $G$. 
\end{proof}

As in the classical case we obtain functors between the Eilenberg--Moore category of the preprojective monad and the Eilenberg--Moore category of the free monad: Denote by $g^*\colon (\Id\Downarrow \tau)\to (X\Downarrow \Id)$ the forgetful functor $(X,h_1,\psi)\mapsto (X,h_1)$ and by $j\colon (X\Downarrow \Id)\to (\Id\Downarrow \tau)$ the functor $(M,h_1)\mapsto (M,h_1,0)$.

\begin{lem}\label{isomorphismgamma}
Let $\mathcal{C}$, $X$, and $Y$ be as in Assumption \ref{standard1}. 
Then there exists an adjunction $j\circ \nu\dashv \nu^-\circ g^*$. 
Furthermore $g^*((\adj^{j\circ \nu\dashv \nu^-\circ g^*})^{-1})=(\adj^{\nu\dashv \nu^-})^{-1}$.
\end{lem}

\begin{proof}
Let $(M,h_1)\in (X\Downarrow \Id)$ and $(N,k_1,\psi)\in (\Id\Downarrow \tau)$. It is sufficient to show that any morphism $a\colon \nu(M,h_1)\to (N,k_1)$ in $\mathcal{C}^{T(X)}$ is also a morphism
\[
a\colon j\nu(M,h_1)\to (N,k_1,\psi)
\]
in $(\operatorname{Id}\Downarrow \tau)$. To prove this we need to show that the following diagram commutes 
\[
\begin{tikzcd}
\nu(M,h_1)\arrow{r}{0}\arrow{d}{a} &\tau \nu(M,h_1)\arrow{d}{\tau(a)}\\
(N,k_1)\arrow{r}{\psi} &\tau (N,k_1)
\end{tikzcd}
\]
Let $\tau^-\nu(M,h_1)\xrightarrow{s} \nu(M,h_1)$ and $\tau^-(N,k_1)\xrightarrow{t} (N,k_1)$ denote the corresponding objects in $(\tau^-\Downarrow \operatorname{Id})$ obtained via the isomorphism $(\operatorname{Id}\Downarrow \tau)\cong (\tau^-\Downarrow \operatorname{Id})$ in Theorem \ref{ringelsdescriptionofpreprojective}. Then the diagram above commutes if and only if the diagram  
\[
\begin{tikzcd}
\tau^-\nu(M,h_1)\arrow{r}{s}\arrow{d}{\tau^-(a)} & \nu(M,h_1)\arrow{d}{a}\\
\tau^-(N,k_1)\arrow{r}{t} & (N,k_1)
\end{tikzcd}
\]
commutes. But this diagram commutes trivially since $\nu(M,h_1)\in \relGinj{I}{\mathcal{C}^{T(X)}}$ by Corollary \ref{Gorensteinprojectiveimagenu} and Lemma \ref{globaldimension1}, and hence $\tau^-\nu(M,h_1)=R^1\nu^-\nu(M,h_1)=0$ by the dual of Theorem \ref{characterisationsofmono}. The second claim follows from the definition of the adjunction isomorphism.
\end{proof}

\begin{rmk}
Note that even in the classical case, $j$ and $g^*$ are not adjoint, so this adjunction doesn't follow from a composition of two adjunctions. 
\end{rmk}

Finally, we are able to prove Theorem \ref{thmD}:

\begin{thm}\label{thmDmaintext}
Let $\mathcal{C}$, $X$, and $Y$ be as in Assumption \ref{standard1}. Assume furthermore that $\mathcal{C}^{T(X)}$ is Krull--Schmidt. Let
\begin{equation}\label{ARinpreproj}
0\to (L,h_1,\psi_1)\to (M,\tilde{h}_1,\tilde{\psi}_1)\to (N,\hat{h}_1,\hat{\psi}_1)\to 0
\end{equation}

be an almost split sequence in $\mathcal{C}^{\Pi(X)}$.
If $\nu^- (N,\hat{h}_1)$ is not projective then the sequence 
\[0\to \nu^-(L,h_1)\to \nu^-(M,\tilde{h}_1)\to \nu^-(N,\hat{h}_1)\to 0\]
is either split exact or a sum of an almost split sequence in $\relGproj{P}{\mathcal{C}^{T(X)}}$ and a sequence of the form $0\to (L',h')\xrightarrow{\id} (L',h')\to 0\to 0$ in $\relGproj{P}{\mathcal{C}^{T(X)}}$. 

A dual statement holds for relative Gorenstein injectives. 
\end{thm}

\begin{proof}
Application of $\nu^-g^*$ to the almost split sequence \eqref{ARinpreproj}  yields the left exact sequence
\[0\to \nu^-(L,h_1)\to \nu^-(M,\tilde{h}_1)\to \nu^-(N,\hat{h}_1).\]
We will show that this sequence satisfies that for every $(K,\ell_1)\in \relGproj{P}{\mathcal{C}^{T(X)}}$ and every $b\colon (K,\ell_1)\to \nu^-(N,\hat{h}_1)$ which is not a split epimorphism, $b$ factors through $\nu^-(M,\tilde{h}_1)$. Using the adjunction isomorphism $(\adj^{j\circ \nu\dashv \nu^-\circ g^*})^{-1}$ in Lemma \ref{isomorphismgamma} we obtain a map 
\[j\nu(K,\ell_1)\xrightarrow{(\adj^{j\circ \nu\dashv \nu^-\circ g^*})^{-1}(b)} (N,\hat{h}_1,\hat{\psi}_1).\]
We claim that $(\adj^{j\circ \nu\dashv \nu^-\circ g^*})^{-1}(b)$ is not a split epimorphism. Assume to the contrary that there exists $\delta\colon (N,\hat{h}_1,\hat{\psi}_1)\to j\nu(K,\ell_1)$ such that $(\adj^{j\circ \nu\dashv \nu^-\circ g^*})^{-1}(b)\delta=\id_{(N,\hat{h}_1,\hat{\psi}_1)}$. Applying $g^*$ we obtain that $(\adj^{\nu\dashv \nu^-})^{-1}(b)g^*(\delta)=\id_{(N,h_1)}$ using Lemma \ref{isomorphismgamma}. Now applying $\nu^-$ to this and using that $\nu^-\nu(K,\ell_1)\cong (K,\ell_1)$ via the unit of the adjunction since $(K,\ell_1)\in \relGproj{P}{\mathcal{C}^T(X)}$, we get a morphism $\nu^-(N,\hat{h}_1)\to (K,\ell_1)$ such that the composite
\[\nu^-(N,\hat{h}_1)\to (K,\ell_1)\xrightarrow{b} \nu^-(N,\hat{h}_1)\]
is the identity. This contradicts the assumption that $b$ is not a split epimorphism. 

Since \eqref{ARinpreproj} is an almost split sequence, it follows that $\adj^{j\circ \nu\dashv \nu^-\circ g^*}(b)$ factors through $(M,\tilde{h}_1,\tilde{\psi}_1)$ via a map $\xi\colon j\nu(K,\ell_1)\to (M,\tilde{h}_1,\tilde{\psi}_1)$. Again applying $g^*$ and the adjunction isomorphism we see that $b$ factors through $\nu^-(M,\tilde{h}_1)$ via the map $\adj^{\nu\dashv \nu^-}(g^*(\xi))$. 

We are left with proving that the sequence $0\to \nu^-(L,h_1)\to \nu^-(M,\tilde{h}_1)\to \nu^-(N,\hat{h}_1)\to 0$ is right exact. Since $\nu^-(N,\hat{h}_1)$ is not projective, there exists a non-split epimorphism $(K',\ell_1')\to \nu^-(N,\hat{h}_1)$. Composing with the epimorphism $\counit{f_!}{f^*}_{(K',\ell_1')}\colon f_!f^*(K',\ell_1')\to (K',\ell_1')$, we get a non-split epimorphism $f_!f^*(K',\ell_1')\to \nu^-(N,\hat{h}_1)$. Since $f_!f^*(K',\ell_1')$ is relative projective, it is contained in  $\relGproj{P}{\mathcal{C}^{T(X)}}$. By the argument above the obtained non-split epimorphism will therefore factor through $\nu^-(M,\tilde{h}_1)$. This proves that the sequence is right exact. 
Hence, if the sequence is not split exact then it must be a sum of an almost split sequence and a sequence $0\to (L',h')\xrightarrow{\id} (L',h')\to 0\to 0$ by Lemma \ref{right minimal gives almost split seq} and the fact that in a Krull--Schmidt category there are minimal versions of morphisms. This proves the claim. 
\end{proof}

We conclude the paper with the following question, which to the best of our knowledge is even open in the classical case of the preprojective algebra of a quiver:

\begin{q}
Does the functor $g^*\colon \mathcal{C}^{\Pi(X)}\to \mathcal{C}^{T(X)}$ send almost split sequences to direct sums of almost split sequences and split sequences? In the classical case, since the embedding of $\Bbbk Q$ into $\Pi(Q)$ splits, it is easy to see that an almost split sequence in $\Pi(Q)$ gets sent to a sum of almost split sequences and split sequences as an element of (the socle of) $\Ext^1_{\Bbbk Q}(\tau(M),M)$. In small examples, it seems to us that this is even true on the level of short exact sequences.
\end{q}

\bibliographystyle{alpha}
\bibliography{publication}

\begin{thebibliography}{GKKP19}

\bibitem[AK13]{AK13}
Hideto Asashiba and Mayumi Kimura.
\newblock Presentations of {G}rothendieck constructions.
\newblock {\em Communications in Algebra}, 41(11):4009--4024, 2013.

\bibitem[AR74]{AR74}
Maurice Auslander and Idun Reiten.
\newblock Stable equivalence of dualizing {$R$}-varieties.
\newblock {\em Advances in Math.}, 12:306--366, 1974.

\bibitem[ARS95]{ARS95}
Maurice Auslander, Idun Reiten, and Sverre~Olaf Smal{\o}.
\newblock {\em {Representation Theory of {A}rtin Algebras}}.
\newblock Cambridge University Press, 1995.

\bibitem[ARV11]{ARV11}
J.~Ad{\'a}mek, J.~Rosick{\'y}, and E.~M. Vitale.
\newblock {\em {Algebraic theories}}, volume 184 of {\em {Cambridge Tracts in
  Mathematics}}.
\newblock Cambridge University Press, Cambridge, 2011.

\bibitem[Asa13]{Asa13}
Hideto Asashiba.
\newblock Gluing derived equivalences together.
\newblock {\em Advances in Mathematics}, 235:134--160, 2013.

\bibitem[Aus66]{Aus66}
Maurice Auslander.
\newblock {Coherent functors}.
\newblock In {\em {Proc. {C}onf. {C}ategorical {A}lgebra ({L}a {J}olla,
  {C}alif., 1965)}}, pages 189--231. Springer, New York, 1966.

\bibitem[Bae01]{Bae01}
John~C. Baez.
\newblock {This week's finds in mathematical physics 174}.
\newblock Available at http://www.math.ucr.edu/home/baez/week174.html, 2001.

\bibitem[BBOS20]{BBOS20}
Ulrich Bauer, Magnus Botnan, Steffen Oppermann, and Johan Steen.
\newblock Cotorsion torsion triples and filtered hierarchical clustering.
\newblock {\em Advances in Mathematics}, 369:107171, 51, 2020.

\bibitem[BHKR15]{BHKR15}
Marcello~M. Bonsangue, Helle~H. Hansen, Alexander Kurz, and Jurrian Rot.
\newblock {Presenting distributive laws}.
\newblock {\em Logical Methods in Computer Science}, 11(3):3:2, 23, 2015.

\bibitem[Bor94]{Bor94a}
Francis Borceux.
\newblock {\em Handbook of categorical algebra. 2}, volume~51 of {\em
  Encyclopedia of Mathematics and its Applications}.
\newblock Cambridge University Press, Cambridge, 1994.
\newblock Categories and Structures.

\bibitem[BR07]{BR07}
Apostolos Beligiannis and Idun Reiten.
\newblock {Homological and homotopical aspects of torsion theories}.
\newblock {\em Memoirs of the American Mathematical Society},
  188(883):viii+207, 2007.

\bibitem[CW11]{CW11}
John Clark and Robert Wisbauer.
\newblock {Idempotent monads and {$\star$}-functors}.
\newblock {\em Journal of Pure and Applied Algebra}, 215(2):145--153, 2011.

\bibitem[DR74a]{DR74a}
Vlastimil Dlab and Claus~Michael Ringel.
\newblock {Repr{\'e}sentations des graphes valu{\'e}s}.
\newblock {\em Comptes Rendus de l'Acad{\'e}mie des Sciences Paris, Series A},
  278:537--540, 1974.

\bibitem[DR74b]{DR74b}
Vlastimil Dlab and Claus~Michael Ringel.
\newblock {\em {Representations of graphs and algebras}}.
\newblock Number~8 in {Carleton Mathematical Lecture Notes}. Departement of
  Mathematics, Carleton University, Ottawa, Ontario, 1974.
\newblock pp. iii+86.

\bibitem[DR75]{DR75}
Vlastimil Dlab and Claus~Michael Ringel.
\newblock {On algebras of finite representation type}.
\newblock {\em Journal of algebra}, 33:306--394, 1975.

\bibitem[DR76]{DR76}
Vlastimil Dlab and Claus~Michael Ringel.
\newblock {Indecomposable representations of graphs and algebras}.
\newblock {\em Memoirs of the American Mathematical Society}, 6:v+57, 1976.

\bibitem[EH17]{EH17}
Ben Elias and Matthew Hogancamp.
\newblock Categorical diagonalization.
\newblock Preprint, arXiv: 1707.04349, 2017.

\bibitem[EM65]{EM65}
Samuel Eilenberg and John~C. Moore.
\newblock {Adjoint functors and triples}.
\newblock {\em Illinois Journal of Mathematics}, 9:381--398, 1965.

\bibitem[FGR75]{FGR75}
Robert~M. Fossum, Phillip~A. Griffith, and Idun Reiten.
\newblock {\em {Trivial extensions of abelian categories}}, volume 456 of {\em
  {Lecture Notes in Mathematics}}.
\newblock Springer, Berlin-New York, 1975.
\newblock Homological algebra of trivial extensions of abelian categories with
  applications to ring theory.

\bibitem[Fir16]{Fir16}
Uriya~A. First.
\newblock Categorical realizations of quivers.
\newblock {\em Communications in Algebra}, 44(6):2567--2582, 2016.

\bibitem[Gab73]{Gab73}
Peter Gabriel.
\newblock {Indecomposable representations. {II}}.
\newblock In {\em {Symposia {M}athematica, {V}ol. {XI} ({C}onvegno di {A}lgebra
  {C}ommutativa, {INDAM}, {R}ome, 1971}}, pages 81--104. Academic Press,
  London, 1973.

\bibitem[Geu17]{Geu17}
Jan Geuenich.
\newblock {\em Quiver modulations and potentials}.
\newblock phdthesis, 2017.

\bibitem[GKKP19]{GKKP19b}
Nan Gao, Julian K\"ulshammer, Sondre Kvamme, and Chrysostomos Psaroudakis.
\newblock A functorial approach to monomorphism categories for species {II}.
\newblock In preparation, 2019.

\bibitem[GLS17]{GLS16}
Christof Geiss, Bernard Leclerc, and Jan Schr{\"o}er.
\newblock {Quiver with relations for symmetrizable {C}artan matrices {I}:
  {F}oundations}.
\newblock {\em Inventiones Mathematicae}, 209(1):61--158, 2017.

\bibitem[GT92]{GT92}
Ronald Gentle and Gordana Todorov.
\newblock {Approximations, adjoint functors and torsion theories}.
\newblock In {\em {Proceedings of the {S}ixth {I}nternational {C}onference on
  {R}epresentations of {A}lgebras ({O}ttawa, {ON}, 1992)}}, volume~14, page~15,
  1992.

\bibitem[Hum08]{Hum08}
James~E. Humphreys.
\newblock {\em Representations of semisimple {L}ie algebras in the {BGG}
  category {$\mathcal{O}$}}, volume~94 of {\em Graduate Studies in
  Mathematics}.
\newblock American Mathematical Society, Providence, RI, 2008.

\bibitem[Iov06]{Iov06}
Miodrag~Christian Iovanov.
\newblock When {$\Pi$} is isomorphic to {$\oplus$}.
\newblock {\em Communications in Algebra}, 34(12):4551--4562, 2006.

\bibitem[Jan03]{Jan03}
Jens~Carsten Jantzen.
\newblock {\em {Representations of algebraic groups}}, volume 107 of {\em
  {Mathematical Surveys and Monographs}}.
\newblock American Mathematical Society, Providence, RI, second edition, 2003.

\bibitem[JK11]{JK11}
Peter J{\o}rgensen and Kiriko Kato.
\newblock Symmetric {A}uslander and {B}ass categories.
\newblock {\em Mathematical Proceedings of the Cambridge Philosophical
  Society}, 150(2):227--240, 2011.

\bibitem[Joh02]{Joh02}
Peter~T. Johnstone.
\newblock {\em Sketches of an elephant: a topos theory compendium. Vol. 1},
  volume~43 of {\em Oxford Logic Guides}.
\newblock The Clarendon Press, Oxford University Press, New York, 2002.

\bibitem[Kas61]{Kas61}
Friedrich Kasch.
\newblock {Projektive {F}robenius-{E}rweiterungen}.
\newblock {\em S.-B. Heidelberger Akad. Wiss. Math.-Nat. Kl.}, 1960/61:87--109,
  1960/61.

\bibitem[Kel96]{Kel96}
Bernhard Keller.
\newblock Derived categories and their uses.
\newblock In {\em Handbook of algebra, {V}ol. 1}, volume~1, pages 671--701.
  1996.

\bibitem[Kra98]{Kra98}
Henning Krause.
\newblock {Functors on locally finitely presented additive categories}.
\newblock {\em Colloquium Mathematicum}, 75(1):105--132, 1998.

\bibitem[Kra15]{Kra15}
Henning Krause.
\newblock Krull--{S}chmidt categories and projective covers.
\newblock {\em Expositiones Mathematicae}, 33(4):535--549, 2015.

\bibitem[KS98]{KS98}
Henning Krause and Manuel Saor{\'\i}n.
\newblock On minimal approximations of modules.
\newblock In {\em Trends in the representation theory of finite-dimensional
  algebras ({S}eattle, {WA}, 1997)}, volume 229 of {\em Contemporary
  Mathematics}, pages 227--236. 1998.

\bibitem[K{\"u}l17]{Kul17}
Julian K{\"u}lshammer.
\newblock {Pro-species of Algebras {I}: Basic properties}.
\newblock {\em Algebras and Representation Theory}, 20(5):1215--1238, 2017.

\bibitem[Kva18]{Kva17}
Sondre Kvamme.
\newblock Gorenstein projective objects in functor categories.
\newblock {\em Nagoya Mathematical Journal}, pages 1--41, 2018.

\bibitem[Kva20]{Kva16}
Sondre Kvamme.
\newblock {A generalization of the {N}akayama functor}.
\newblock {\em Algebras and Representation Theory}, 23(4):1319--1353, 2020.

\bibitem[Lau06]{Lau06}
Aaron~D. Lauda.
\newblock {Frobenius algebras and ambidextrous adjunctions}.
\newblock {\em Theory and Applications of Categories}, 16(4):84--122, 2006.

\bibitem[Li12]{Li12}
Fang Li.
\newblock {Modulation and natural valued quiver of an algebra}.
\newblock {\em Pacific Journal of Mathematics}, 256(1):105--128, 2012.

\bibitem[LO17]{LO17}
Boris Lerner and Steffen Oppermann.
\newblock {A recollement approach to {G}eigle-{L}enzing weight projective
  varieties}.
\newblock {\em Nagoya Mathematical Journal}, 226:71--105, 2017.

\bibitem[Lur09]{Lur09}
Jacob Lurie.
\newblock {\em Higher topos theory}, volume 170 of {\em Annals of Mathematics
  Studies}.
\newblock Princeton University Press, Princeton, NJ, 2009.

\bibitem[LY17]{LY15}
Fang Li and Chang Ye.
\newblock {Representations of {F}robenius-type triangular matrix algebras}.
\newblock {\em Acta Mathematica Sinica (English Series)}, 33(3):341--361,
  December 2017.

\bibitem[LZ13]{LZ13}
Xiu-Hua Luo and Pu~Zhang.
\newblock {Monic representations and {G}orenstein-projective modules}.
\newblock {\em Pacific Journal of Mathematics}, 264(1):163--194, 2013.

\bibitem[LZ17]{XZ17}
Xiu-Hua Luo and Pu~Zhang.
\newblock Separated monic representations {I}: {G}orenstein-projective modules.
\newblock {\em Journal of Algebra}, 479:1--34, 2017.

\bibitem[{Mac}98]{McL98}
Saunders {Mac Lane}.
\newblock {\em {Categories for the working mathematician}}, volume~5 of {\em
  {Graduate Texts in Mathematics}}.
\newblock Springer-Verlag, New York, second edition, 1998.

\bibitem[Mor65]{Mor65}
Kiiti Morita.
\newblock {Adjoint pairs of functors and {F}robenius extensions}.
\newblock {\em Sci. Rep. Tokyo Kyoiku Daigaku Sect. A}, 9:40--71, 1965.

\bibitem[Moz20]{Moz20}
Sergey Mozgovoy.
\newblock Quiver representations in abelian categories.
\newblock {\em Journal of Algebra}, 541:35--50, 2020.

\bibitem[MW13]{MW13}
Bachuki Mesablishvili and Robert Wisbauer.
\newblock Q{F} functors and (co)monads.
\newblock {\em Journal of Algebra}, 376:101--122, 2013.

\bibitem[Nee02]{Nee02}
Amnon Neeman.
\newblock {A counterexample to a 1961 ``theorem'' in homological algebra}.
\newblock {\em Inventiones Mathematicae}, 148(2):397--420, 2002.
\newblock With an appendix by P. Deligne.

\bibitem[Pos11]{Pos11}
Leonid Positselski.
\newblock Mixed {A}rtin-{T}ate motives with finite coefficients.
\newblock {\em Moscow Mathematical Journal}, 11(2):317--402, 407--408, 2011.

\bibitem[Rin98]{Ri98}
Claus~Michael Ringel.
\newblock {The preprojective algebra of a quiver}.
\newblock In {\em {Algebras and modules, {II} ({G}eiranger, 1996)}}, volume~24
  of {\em {CMS Conference Proceedings}}, pages 467--480. American Mathematical
  Society, Providence, RI, 1998.

\bibitem[RS06]{RS06}
Claus~Michael Ringel and Markus Schmidmeier.
\newblock {Submodule catgories of wild representation type}.
\newblock {\em Jounal of Pure and Applied Algebra}, 205(2):412--422, 2006.

\bibitem[RS08]{RS08}
Claus~Michael Ringel and Markus Schmidmeier.
\newblock {The {A}uslander-{R}eiten translation in submodule categories}.
\newblock {\em Transactions of the American Mathematical Society},
  360(2):691--716, 2008.

\bibitem[RZ17]{RZ17}
Claus~Michael Ringel and Pu~Zhang.
\newblock Representations of quivers over the algebra of dual numbers.
\newblock {\em Journal of Algebra}, 475:327--360, 2017.

\bibitem[\v{S}14]{Sto14}
Jan \v{S}\v{t}ov{\'i}\v{c}ek.
\newblock {Derived equivalences induced by big cotilting modules}.
\newblock {\em Advances in Mathematics}, 263:45--87, 2014.

\bibitem[XZ12]{XZ12}
Bao-Lin Xiong and Pu~Zhang.
\newblock {Gorenstein-projective modules over triangular matrix {A}rtin
  algebras}.
\newblock {\em Journal of Algebra and its Applications}, 11(4):1250066, 14,
  2012.

\bibitem[XZZ14]{XZZ14}
Bao-Lin Xiong, Pu~Zhang, and Yue-Hui Zhang.
\newblock {Auslander-{R}eiten translation in monomorphism categories}.
\newblock {\em Forum Mathematicum}, 26(3):863--912, 2014.

\bibitem[Zha11]{Zha11}
Pu~Zhang.
\newblock {Monomorphism categories, cotilting theory, and
  {G}orenstein-projective modules}.
\newblock {\em Journal of Algebra}, 339:181--202, 2011.

\bibitem[ZX19]{XZ19}
Pu~Zhang and Bao-Lin Xiong.
\newblock Separated monic representations {II}: {F}robenius subcategories and
  {RSS} equivalences.
\newblock {\em Transactions of the American Mathematical Society},
  372(2):981--1021, 2019.

\end{thebibliography}

\end{document}